\documentclass[12pt,reqno]{amsart}
\usepackage{amsmath,amsthm,amsfonts}
\usepackage{caption}
\usepackage{mathtools}
\usepackage{cite}
\usepackage{enumitem}
\usepackage{color}
\usepackage{url}
\usepackage[usenames,dvipsnames]{xcolor}
\usepackage{graphicx}
\usepackage{xspace}
\usepackage{hyperref}
\usepackage{cleveref}
\usepackage{leftidx}
\usepackage[margin=1.1in]{geometry}
\usepackage{comment}
\usepackage{mathrsfs}

\usepackage{hyperref}
\usepackage{color}
\definecolor{darkred}{rgb}{0.4,0,0}
\definecolor{darkgreen}{rgb}{0,0.5,0}
\definecolor{darkblue}{rgb}{0,0,0.4}
\hypersetup{
	pdftitle={},		% title
	pdfauthor={NS,PY},	% author
	pdfsubject={math},		% subject of the document
	pdfkeywords={equidistribution, homogeneous space, Dirichlet improvability},	% list of keywords
	pdfnewwindow=true,		% links in new window
	colorlinks=true,		% false: boxed links; true: colored links
	linkcolor=black,		% color of internal links (change box color with linkbordercolor)
	citecolor=black,		% color of links to bibliography
	filecolor=black,		% color of file links
	urlcolor=black,		% color of external links
	pdfborder={0 0 0},
	breaklinks=true
}

\newcommand{\Sp}{\operatorname{Sp}}
\newcommand{\GSp}{\operatorname{GSp}}
\newcommand{\GL}{\operatorname{GL}}
\newcommand{\PGL}{\operatorname{PGL}}
\newcommand{\SL}{\operatorname{SL}}
\newcommand{\SO}{\operatorname{SO}}

\newcommand{\Hom}{\operatorname{Hom}}

\newcommand{\Int}{\operatorname{Int}}

\newcommand{\Gr}{\operatorname{Gr}}
\usepackage{amssymb}

\makeatletter
\renewcommand{\paragraph}{%
	\@startsection {paragraph}{4}
	{\z@} \z@ {-\fontdimen 2\font }\bfseries
}
\makeatother

\makeatletter % makes "@" a normal character so that it can be used in macros
\def\@cite#1#2{{\m@th\upshape\bfseries%
		[{#1\if@tempswa{\m@th\upshape\mdseries, #2}\fi}]}}
\makeatother % returns "@" to its normal state

\numberwithin{equation}{section}

\theoremstyle{plain}
\newtheorem{thm}{Theorem}[section]
\newtheorem{cor}[thm]{Corollary}
\newtheorem{prop}[thm]{Proposition}
\newtheorem{lem}[thm]{Lemma}

\theoremstyle{definition}
\newtheorem{definition}[thm]{Definition}

\theoremstyle{remark}
\newtheorem{remark}[thm]{Remark}

\newcommand{\rnc}{\renewcommand}
\newcommand{\lip}{\langle}
\newcommand{\rip}{\rangle}
\newcommand{\ann}[1]{}
\newcommand{\eb}{\emph}%[1]{\emph{\textbf{#1}}}

\newcommand\bfF{\mathbf{F}}
\newcommand\bfG{\mathbf{G}}
\newcommand\bfH{\mathbf{H}}

\newcommand\bfM{\mathbf{M}}
\newcommand\bfN{\mathbf{N}}

\newcommand\bfP{\mathbf{P}}

\newcommand\bfS{\mathbf{S}}
\newcommand\bfT{\mathbf{T}}
\newcommand\bfU{\mathbf{U}}

\newcommand\bfW{\mathbf{W}}

\newcommand\bfX{\mathbf{X}}
\newcommand\bfZ{\mathbf{Z}}
\newcommand\bfp{\mathbf{p}}
\newcommand\bfq{\mathbf{q}}

\newcommand\bfw{\mathbf{w}}
\newcommand\bfx{\mathbf{x}}
\newcommand\bbC{\mathbb{C}}

\newcommand\bbF{\mathbb{F}}
\newcommand\bbG{\mathbb{G}}

\newcommand\bbK{\mathbb{K}}
\newcommand\bbL{\mathbb{L}}

\newcommand\bbN{\mathbb{N}}

\newcommand\bbP{\mathbb{P}}
\newcommand\bbQ{\mathbb{Q}}
\newcommand\bbR{\mathbb{R}}

\newcommand\bbZ{\mathbb{Z}}
\newcommand\cF{\mathcal{F}}
\newcommand\cG{\mathcal{G}}
\newcommand\cH{\mathcal{H}}
\newcommand\cI{\mathcal{I}}
\newcommand{\cJ}{\mathcal{J}} 

\newcommand\cL{\mathcal{L}}
\newcommand\cM{\mathcal{M}}

\newcommand\cP{\mathcal{P}}

\newcommand\cT{\mathcal{T}}
\newcommand\cU{\mathcal{U}}

\newcommand\cW{\mathcal{W}}

\newcommand\cX{\mathcal{X}}

\newcommand\ff{\mathfrak{f}}
\newcommand\fg{\mathfrak{g}}
\newcommand\fh{\mathfrak{h}}
\newcommand{\GS}{\mathfrak{S}}

\DeclarePairedDelimiter\abs{\lvert}{\rvert}%
\DeclarePairedDelimiter\norm{\lVert}{\rVert}%

\makeatletter
\let\oldabs\abs
\def\abs{\@ifstar{\oldabs}{\oldabs*}}
\let\oldnorm\norm
\def\norm{\@ifstar{\oldnorm}{\oldnorm*}}
\makeatother

\newcommand\Weyl{\cW}
\newcommand\nba{\lip\beta,\alpha\rip}
\newcommand\nla{\lip\lambda,\alpha\rip}
\newcommand\nlia{\lip\lambda_i,\alpha\rip}

\newcommand\typeA{A}

\newcommand\typeC{C}

\newcommand{\charS}{X^{\ast}(\bfS)}
\newcommand{\cocharS}{X_{\ast}(\bfS)}
\newcommand{\Gm}{\mathbb{G}_{\mathrm{m}}}

\newcommand{\gM}{g_{{}_\mathcal{M}}}
\newcommand{\Qbar}{\overline{\bbQ}}
\newcommand{\bfHrm}{\bfH_{r,m}}
\newcommand{\bfTrm}{\bfT_{r,m}}
\newcommand{\bfNrm}{\bfN_{r,m}}
\newcommand{\Wrm}{\cW_{r,m}}

\newcommand{\MmlR}{M_{m,l}(\bbR)}

\newcommand{\Wrml}{\cW_{r}(m,l)}
\newcommand{\wrml}{\cW'_{r}(m,l)}
\newcommand{\interval}{B}
\newcommand{\eps}{\varepsilon}
\newcommand{\Lphi}{\cL_\phi}
\newcommand{\LA}{\cL_A}
\newcommand{\LM}{\cL_\cM}

\newcommand{\Aext}{A^{\mathrm{ext}}}
\newcommand{\gA}{u_A}

\rnc{\Re}{\operatorname{Re}}
\rnc{\Im}{\operatorname{Im}}

\DeclareMathOperator{\rank}{rank}
\DeclareMathOperator{\diag}{diag}
\DeclareMathOperator{\End}{End}
\DeclareMathOperator{\Lie}{Lie}

\DeclareMathOperator{\Tran}{Tran}
\DeclareMathOperator{\Gal}{Gal}
\DeclareMathOperator{\Ad}{Ad}
\DeclareMathOperator{\Res}{Res}
\DeclareMathOperator{\pf}{Pff}
\DeclareMathOperator{\Sym}{Sym}
\DeclareMathOperator{\covol}{Covol}

\title[Equidistribution of expanding degenerate manifolds]{Equidistribution of expanding degenerate manifolds in the space of lattices}
\author[Shah]{Nimish~A.~Shah}
\address{The Ohio State University, Columbus, OH 43210}
\email{shah@math.osu.edu}
\author[Yang]{Pengyu~Yang}
\address{Morningside Center of Mathematics, Chinese Academy of Sciences, Beijing, China 100871}
\email{yangpengyu@amss.ac.cn}

\thanks{N.A.Shah was supported by NSF grant DMS-1700394.}
\thanks{P. Yang is supported by the National Key R\&D Program of China 2022YFA1007500 and NSFC grant 22AAA00245.}

\begin{document}
	
	\begin{abstract}
		For the space of unimodular lattices in a Euclidean space, we give necessary and sufficient conditions for equidistribution of expanding translates of any real-analytic submanifold under a diagonal flow. This work extends the earlier result of Shah in the case of non-degenerate submanifolds.
		
		We apply the above dynamical result to show that if the affine span of a real-analytic submanifold in a Euclidean space satisfies certain Diophantine and arithmetic conditions, then almost every point on the manifold is not Dirichlet-improvable.
	\end{abstract}
	
	\maketitle

	%%%%%%%%%%%%%%%%%%%%%%%%%%%%%%%%%%%%%%%%%%%
	%%%%%%%%%%%%%%%%%%%%%%%%%%%%%%%%%%%%%%%%%%%
	\section{Introduction}
	%%%%%%%%%%%%%%%%%%%%%%%%%%%%%%%%%%%%%%%%%%%
	%%%%%%%%%%%%%%%%%%%%%%%%%%%%%%%%%%%%%%%%%%%
	
	\subsection{Background}
	Let $G$ be a Lie group, $\Gamma$ a discrete subgroup of $G$, and $Y$ a smooth submanifold of $G/\Gamma$. In 2002, Margulis \cite{Mar02} asked the following question in a chapter of the book \emph{A panorama of number theory or the view from {B}aker's garden}:
	
	(Q) What is the distribution of $gY$ in $G/\Gamma$ when $g$ tends to infinity in $G$?\\
	He further divided this question into two subquestions:
	
	(Q1) What is the behavior of $gY$ ‘near infinity’ in $G/\Gamma$?
	
	(Q2) What is the distribution of $gY$ in the ‘bounded part’ of $G/\Gamma$?
	
	These questions are connected to various problems in Diophantine approximation and number theory, for instance, the lattice counting problems \cite{DRS93,EM93,EMS96}, the Sprind\v{z}uk conjecture on very well approximable (VWA) vectors \cite{Spr80, KM98} and Dirichlet-improvable (DI) matrices \cite{DS6970, DS70, KW08, Sha09Invention}.
	
	In this paper, we are particularly interested in the distribution of translates $g_tY$, where $\{g_t\}_{t\in\bbR}$ is a diagonalizable flow in $G$, and $Y$ is a real-analytic submanifold which is expanded by $g_1$. We assume that $\Gamma$ is a lattice in $G$, i.e., $G/\Gamma$ has finite volume. We shall refer to (Q1) and (Q2) as \emph{non-escape of mass} and \emph{equidistribution}, respectively.
	
	If $Y$ is an open subset of the full expanding horosphere, using the thickening method originating in Margulis' thesis and the mixing property of the diagonal flows in homogeneous spaces, one can show that the translates $g_tY$ get equidistributed in $G/\Gamma$ as $t\to\infty$; see e.g.\ \cite{EM93}.
	
	When $Y$ is a real-analytic curve, Kleinbock and Margulis studied (Q1) for $G=\SL_n(\bbR)$ in the seminal work \cite{KM98}, and later the first named author studied (Q2) for $G=\SL_n(\bbR)$ in \cite{Sha09Invention, Sha10} and $G=\SO(n,1)$ in \cite{Sha09Duke1}. Later Aka, Breuillard, Rosenzweig, and de Saxc\'{e} \cite{ABRS18} generalized the results in \cite{KM98} to Grassmannians; in particular, they discovered that the obstruction to non-divergence are certain Schubert subvarieties called constraining pencils. Based on these works, in \cite{Yan20Invent}, the second named author studied the case where $G$ is a semisimple algebraic group and defined unstable Schubert varieties as a natural generalization of the constraining pencils, which are potential obstructions to the non-escape of mass. We say that $Y$ is \emph{non-degenerate\/} if $Y$ is not contained in any unstable Schubert variety. In \cite{Yan20Invent}, it was shown that translates of a non-degenerate curve $Y$ has no escape of mass and a necessary and sufficient condition for equidistribution of the translates $g_tY$ is given (again assuming non-degeneracy).
	
	Now, the case where $Y$ is degenerate remains to be studied. The first nontrivial case is $G=\SL_3(\bbR)$. In a preprint \cite{CY19}, S. Chow and L. Yang obtained an effective equidistribution result where $Y$ is a Diophantine line, and the translating elements are in a specific cone of the full diagonal subgroup.  Later, in a joint work~\cite{KSSY21} of authors with D.~Kleinbock and N.~de Saxc\'e, the flow $\diag(e^{2t}, e^{-t}, e^{-t})$ was studied, and necessary and sufficient conditions for both the non-escape-of-mass and equidistribution were provided. For $G=\SL_4(\bbR)$, R.~Shi and B.~Weiss \cite{SW17} gave an example of a line defined over a real quadratic number field, such that the translates are stuck in a fixed compact set; in particular, they do not get equidistributed.
	
	In this paper, we consider the case where $G=\SL_n(\bbR)$ and $g_t=\diag(e^{(n-1)t}, e^{-t}, \dots,e^{-t})$, and we study translates of a degenerate real-analytic submanifold $Y$ of the expanding horosphere; here nondegeneracy means $Y$ is contained in a proper affine subspace of the horosphere, which is isomorphic to $\bbR^{n-1}$ in this case. We provide necessary and sufficient Diophantine or arithmetic conditions on the affine span of $Y$ for both non-escape-of-mass and equidistribution.
	
	\subsection{Statement of the main results}\label{subsect:statement of the main results}
	Let {$G=\SL_{n}(\bbR)$, $\Gamma=\SL_{n}(\bbZ)$} and $X=G/\Gamma$. Let $\mu_{X}$ denote the unique $G$-invariant probability measure on $X$. Let 
 \[
 g_t=\diag(e^{(n-1)t}, e^{-t}, \dots,e^{-t}),
 \]so that the expanding horospherical subgroup {of $G$}  associated to $g_1$ is 
	\begin{equation} \label{eq:U+}
	U^+=\{g\in G\mid g_{-t}gg_t\to e, t\to{+}\infty \}={
		{\left\{\begin{pmatrix}
			1 & * & \cdots & * \\
			& 1 & & \\
			&  & \ddots & \\
			& &  & 1
			\end{pmatrix}\right\}}}.
	\end{equation}
	Let $x_0=e\Gamma\in X$.
	Let $\interval$ be an open ball in a finite-dimensional Euclidean space. Let $\phi\colon \interval\to U^+\cong\bbR^{n-1}$ be a nonconstant real-analytic map. Let $\lambda$ be a probability measure that is absolutely continuous with respect to the Lebesgue measure on the Euclidean space with support contained in $\interval$. Let $\lambda_\phi$ denote the push-forward of $\lambda$ under the map $s\mapsto\phi(s)x_0$ from $\interval$ to 
	$X$. Let $g_t\lambda_\phi$ denote the push-forward of $\lambda_\phi$ by $g_t$, i.e. $g_t\lambda_\phi(E)=\lambda_\phi(g_t^{-1}E)$ for any measurable $E\subset X$. We are interested in the limiting distribution of the family $\{g_t\lambda_\phi\}_{t\geq0}$ of translated measures.
	
	We say that a family $\{\lambda_i\}_{i\in\cF}$ of probability measures on $X$ have \eb{no escape of mass\/} if for every $\eps>0$ there exists a compact subset $K$ of $X$ such that $\lambda_i(K)>1-\eps$ for all $i\in\cF$. Furthermore, for $\cF=\bbN$ or $\bbR_{\geq0}$, we say that a family $\{\lambda_i\}_{i\in\cF}$ of probability measures on $X$ get \eb{equidistributed} in $X$ if
	\[
	\int f \,\mathrm{d}\lambda_i\stackrel{i\to\infty}{\longrightarrow}\int f\,\mathrm{d}\mu_{X},\;\forall f\in C_c(X);
	\]
	that is, $\lambda_i$ converges to $\mu_{X}$ with respect to the weak-$^\ast$ topology as $i\to\infty$.
	
	Our main results give criteria for the non-escape of mass and equidistribution of the translates $\{g_t\lambda_\phi\}_{t\geq0}$. It turns out that these phenomena depend only on the affine span of $\phi(\interval)$.
	
	Let us first parametrize the affine span $\Lphi$ of $\phi(B)$. Since $\phi$ is nonconstant, the dimension of $\Lphi$ is $d-1$ for some $2\leq d\leq n$. 
 
 Suppose $d<n$, then by permuting the coordinates, which commutes with the $g_t$-action, we may assume that $\Lphi$ is of the form
	\[
	\Lphi=\{ (\bfx,\widetilde{\bfx}A_\phi) \mid \bfx\in\bbR^{d-1} \},
	\]
	where $A_\phi\in M_{d,n-d}(\bbR)$. Here $\widetilde{\bfx}=(1,\bfx)$ for any row vector $\bfx\in\bbR^{d-1}$. We wish to phrase our criteria in terms of certain Diophantine and arithmetic properties of $A_\phi$ or $\Lphi$. Therefore, we introduce the following definitions.
	
	Let $m,l$ be positive integers, and let $\norm{\cdot}$ denote the sup-norm. Given a positive real number $r$, 
	let $\Wrml$ denote the set of matrices $A\in\MmlR$ for which there exists $C>0$ such that the inequality
	\begin{equation}\label{eq:Wrml}
	\norm{A\bfq+\bfp}\leq C\norm{\bfq}^{-r}.
	\end{equation} 
	has infinitely many solutions $(\bfp,\bfq)\in\bbZ^m\times\bbZ^l$.
	
	Similarly, let $\wrml$ denote the set of matrices $A\in\MmlR$ for which \eqref{eq:Wrml} has a nonzero solution $(\bfp,\bfq)\in\bbZ^m\times\bbZ^l$ for every $C>0$.
	
	We will prove the following criterion for the non-escape of mass.
	
	\begin{thm}\label{thm:main_nondivergence}
		The translates $\{g_t\lambda_\phi\}_{t\geq0}$ have no escape of mass if and only if one of the following holds:
  \begin{enumerate}
      \item $d=n$.
      \item $d<n$ and $A_\phi\notin\cW'_{n-1}(d,n-d)$.
  \end{enumerate}
	\end{thm}

It follows from the main theorem of \cite[page~353]{BD86} that the Hausdorff dimension of  $\cW'_{n-1}(d,n-d)$ is $d(n-d)-d+1$, which is strictly smaller than $d(n-d)$ since $d\geq 2$.
 
	Before stating our equidistribution criteria, we will introduce some notations. Let $N=\begin{pmatrix}
	n-2 \\ 2
	\end{pmatrix}.$ Given any $A=\begin{pmatrix}
	a_1 & a_2 & \cdots & a_{n-2} \\
	b_1 & b_2 & \cdots & b_{n-2} \\
	\end{pmatrix}\in M_{2,n-2}(\bbR)$, we define an associated matrix $\Aext\in M_{2n-3,N}(\bbR)$ in the following way. In the block matrix form, write \begin{equation}\label{eq:Aext}
	\Aext=\begin{pmatrix}
	X \\ Y \\ Z
	\end{pmatrix},\end{equation}
	where the matrices 
	$X=(x_{k,ij})_{\scriptscriptstyle{\substack{1\leq k\leq n-2 \\ 1\leq i<j\leq n-2}}}\in M_{n-2,N}(\bbR)$, 
	$Y=(y_{k,ij})_{\scriptscriptstyle{\substack{1\leq k\leq n-2 \\ 1\leq i<j\leq n-2}}}\in M_{n-2,N}(\bbR)$, 
	$Z=(z_{12},\dots, z_{ij},\dots)_{\scriptstyle{{1\leq i<j\leq n-2}}}\in M_{1,N}(\bbR)$ are given by
	\[
	x_{k,ij}=\begin{cases}
	a_j & k=i \\
	-a_i & k=j \\
	0 & k\neq i,j
	\end{cases},\;
	y_{k,ij}=\begin{cases}
	b_j & k=i \\
	-b_i & k=j \\
	0 & k\neq i,j
	\end{cases},\;
	z_{ij}=a_jb_i-a_ib_j.
	\]
	We will prove the following criterion for equidistribution.
	\begin{thm}\label{thm:main_equidistribution}
		The translates $\{g_t\lambda_\phi\}_{t\geq0}$ get equidistributed in $X$ if and only if \eb{none} of the following occurs:
		\begin{enumerate}
			\item \label{itm:mainthm_case1}$d<n$ and $A_\phi\in\cW_{n-1}(d,n-d)$.
			\item \label{itm:mainthm_case2} There exist integers $r\geq d$, $m\geq 2$ with $rm=n$, and a number field $\bbK\subset\bbR$ with $[\bbK:\bbQ]=m$, such that $\Lphi$ is contained in some $(r-1)$-dimensional affine subspace of $\bbR^{n-1}$ which is defined over $\bbK$.
			\item \label{itm:mainthm_case3}$n\geq 4$ is even, $d=2$,
			and $\Aext_\phi\in\cW_{\frac{n-2}{2}}(2n-3,N)$.
		\end{enumerate}
	\end{thm}

    When $n$ is a prime number, the (\ref{itm:mainthm_case2}) and (\ref{itm:mainthm_case3}) in the above theorem do not occur. Hence, we have the following:
    \begin{cor}
        Suppose that $n$ is a prime number. Then exactly one of the following holds:
        \begin{enumerate}
            \item The translates $\{g_t\lambda_\phi\}_{t\geq0}$ get equidistributed in $X$.
            \item $d<n$ and $A_\phi\in\cW_{n-1}(d,n-d)$.
        \end{enumerate}
    \end{cor}

    Compared to previous works \cite{Sha09Invention,Yan20Invent,KSSY21}, the main difficulty and novelty of this article is the classification of maximal intermediate subgroups, which arise when we apply Ratner's theorem and linearization technique. We use tools in geometric invariant theory to deal with non-closed group orbits in linear representations. In the remaining cases, where we have closed orbits, we first classify the maximal intermediate subgroup over an algebraically closed field and then use Galois cohomology to classify the possible $\bbQ$-groups that could arise.

    If $\phi$ is only assumed to be a smooth map, then non-divergence and equidistribution no longer depend only on the affine span of the submanifold, and one may need to impose some local conditions on the derivatives of $\phi$. The main difficulty in working with smooth curve case is that the $(C,\alpha)$-good properties needed for the linearization technique do not work for translates of a fixed piece of a smooth curve. To address this issue, one studies \emph{expanding translates of shrinking curves} as done in \cite{Sha09Duke2,SY18,SY24} for non-degenerate smooth curves. However, it is unclear how to adapt this article's geometric invariant theory technique to study the translates of shrinking curves. 
    
	\newcommand{\DI}{\rm{DI}}
	\subsection{An application to Dirichlet-improvable vectors on degenerate manifolds}
	The motivation for our study comes from the Diophantine approximation. {Denote by $\|\cdot\|$ the supremum norm on $\bbR^{n-1}$, where $n\geq 2$ (unless specified otherwise, all the norms will be taken to be the supremum norm).} Let $\cT=\{T_i\}\subset\bbR$ such that $T_i\to\infty$ as $i\to\infty$, and let $0<\delta\leq 1$. {Following} Davenport and Schmidt \cite{DS6970}, let $\DI(\cT,\delta)$ denote the set of vectors $\bfx\in\bbR^{n-1}$ such that for all large $T\in \cT$, the system of inequalities
 {\begin{equation}\label{lf}
		\begin{cases}
		\abs{\bfx\cdot\bfq+p}\leq \delta T^{-(n-1)} \\
		\|\bfq\| \leq  T\end{cases}
		\end{equation}
		has a solution 
		$(p,\bfq)$, where $p\in\bbZ$ and $\bfq\in\bbZ^{n-1}\setminus\{0\}$}.  Similarly, let $\DI'(\cT,\delta)$ denote the set of vectors $\bfx\in\bbR^{n-1}$ such that for all large $T\in\cT$, the system of inequalities 
  {\begin{equation}\label{vect}
		\begin{cases}
		{\|q\bfx+\bfp\|}\leq \delta T^{-1} \\
		\abs{q} \leq  T^{n-1}
		\end{cases}
		\end{equation}
		has a solution 
		$(\bfp,q)$, where $\bfp\in\bbZ^{n-1}$ and $q\in\bbZ\setminus\{0\}$}.
	By Dirichlet's theorem $\DI(\cT,1)=\bbR^{n-1}$ and $\DI'(\cT,1)=\bbR^{n-1}$. By a theorem of Davenport and Schmidt, for any $0<\delta<1$, the sets $\DI(\cT,\delta)$ and $\DI'(\cT,\delta)$ are Lebesgue null, see \cite{DS70,KW08}. After Davenport and Schmidt \cite{DS6970}, let $\interval$ be an open ball in a finite-dimensional Euclidean space and $\phi:\interval\to \bbR^{n-1}$ be a $C^\infty$ map, and we want to know under what condition on $\phi$ one can say that for Lebesgue a.e.\ $s\in \interval$, we have $\phi(s)\not\in \DI(\cT,\delta)\cup \DI'(\cT,\delta)$ for any $\delta<1$; in other words, the Dirichlet's theorem cannot be improved for $\phi(s)$ for a.e.\ $s\in \interval$. In a series of articles \cite{Sha09Invention, Sha10, SY18}, it was shown that the above non-improvability statement holds when the affine span of $\phi(\interval)$ equals $\bbR^{n-1}$. In \cite{KSSY21} it was proved that if $n=3$ and $\phi(\interval)$ is not contained in a rational line in $\bbR^2$, then for Lebesgue almost all $s\in \interval$, $\phi(s)\not\in \DI(\bbN,\delta)\cup\DI(\bbN,\delta)$ for any $\delta<1$. In this article, we generalize the result for all $n\geq 3$ in terms of Diophantine and algebraic properties of the affine span of $\phi(\interval)$.  
 
	As an application of our dynamical result \Cref{thm:main_equidistribution}, we will obtain the following using Dani's correspondence~\cite{Dan85,KW08}.
 
	\begin{thm}\label{thm:main_Dirichlet}
		Let $\phi\colon\interval\to\bbR^{n-1}$ be a real-analytic map. Suppose that \eb{none} of (\ref{itm:mainthm_case1})(\ref{itm:mainthm_case2})(\ref{itm:mainthm_case3}) in \Cref{thm:main_equidistribution} occurs. Let $\cT=\{T_i\}\subset\bbR$ such that $T_i\to\infty$ as $i\to\infty$, then  for Lebesgue-a.e.\ $s\in\interval$, $\phi(s)\notin \DI(\cT,\delta)\cup \DI'(\cT,\delta)$ for any $\delta<1$. In fact, for a.e.\ $s\in B$, for every $\delta>0$, there exists a sequence $i_j\to\infty$ such that \eqref{lf} and \eqref{vect} are not solvable for $\bfx=\phi(s)$ and $T=T_{i_j}$ for all $j$. 
	\end{thm}

\begin{proof}
    Let $\{e_i:1\leq i\leq n\}$ denote the standard basis of $\bbR^n$ and $\mathfrak{w}$ be the matrix such that $\mathfrak{w}e_i=e_{n-i+1}$ for all $i$. Let $L=G\times G$ and $\Lambda=\Gamma\times \Gamma$. Let $\rho(g)=(g,\mathfrak{w}(^t\!g^{-1})\mathfrak{w}^{-1})$ for all $g\in G$. Then $\rho: G\to L$ is an injective homomorphism. Then $\rho(\Gamma)\subset \Lambda$ and $\rho(\Gamma)$ is a lattice in $\rho(G)$. Let $\bar\rho:G/\Gamma\to L/\Lambda$ be the map $\bar\rho(g)=\rho(g)\Lambda$ for all $g\in G$. Then $\bar\rho$ is a continuous injective proper map. Now for any $f\in C_c(L/\Lambda)$, we have $f\circ\bar\rho\in C_c(G/\Gamma)$. So, by \Cref{thm:main_equidistribution}, we get
    \begin{align*}
       \lim_{t\to\infty} \frac{1}{\abs{B}} \int_B f(\rho(a_t)\rho(\phi(s))\Lambda)\,ds
       &=\lim_{t\to\infty} \frac{1}{\abs{B}} \int_{B} f\circ\rho(a_t\phi(s))\Gamma)\,ds\\
       &=\int_{G/\Gamma} f\circ\rho\, d\mu_X \\
       &=\int_{L/\Lambda} f \,d\mu_{\bar\rho(X)},
    \end{align*}
    where $\mu_{\bar\rho(X)}$ is the unique $\rho(G)$ invariant probability measure on the closed set $\bar\rho(G/\Gamma)\cong\rho(G)/\rho(\Gamma)$ of $L/\Lambda$. Having shown this, we argue exactly as in Section~2 of \cite{Sha09Invention} replacing \cite[Theorem~1.3]{Sha09Invention} by the above deduction. 
\end{proof}

 The readers are referred to \cite[Theorem 1.5]{KSSY21} or \cite[Section 1.2]{SY16} for a discussion on the deduction of \Cref{thm:main_Dirichlet} from \Cref{thm:main_equidistribution} using Dani's correspondence \cite{Dan85,KW08}. In fact, for the map $\phi$ as in the statement of \Cref{thm:main_Dirichlet}, one also obtains the conclusion of \cite[Theorem~1.4]{Sha09Invention} which is a generalization in the sense of simultaneous non-Dirichlet improvability along a sequence of natural numbers.

 \begin{remark}
     As we have noted above, if $n$ is a prime number, then (\ref{itm:mainthm_case2}) or (\ref{itm:mainthm_case3}) in \Cref{thm:main_equidistribution} do not occur. This is the situation in \cite{KSSY21}, where $n=3$.
 \end{remark}

 \subsection{Notation}
 In this paper, all algebraic groups are assumed to be affine.
	
	We use boldface capital letters $\bfG, \bfH, \bfF$, etc. to denote algebraic groups (over $\bbQ$ if not specified) and use Roman capital letters $G, H, F$, etc.\ to denote the groups of real points, and $G^0,H^0,F^0$, etc.\ to denote their connected components of identity, respectively.
	
	Let $\bbK$ be a field contained in $\bbR$ and $\bfG$ an algebraic group over $\bbK$. In this paper, a \emph{representation} of $\bfG$ always means a finite-dimensional algebraic representation, i.e., a pair $(\rho, V)$ where $V$ is a vector space over $\bbK$ and $\rho: \bfG\to\GL(V)$ is a morphism of algebraic groups over $\bbK$. We also say that $V$ is a \emph{$\bfG$-module}. Sometimes, we call $\rho$ or $V$ a representation of $\bfG$ for simplicity. We say that $\rho$ is \emph{faithful} if the kernel of $\rho$ is trivial and that $\rho$ is \emph{irreducible} if there is no non-trivial $\bfG$-invariant proper subspace of $V$.
	
	Given a positive integer $r$, let $I_r$ denote the identity $r\times r$ matrix.
	
	Let $\bfX$ be a variety with $\bfG$ acting morphically over $\bbK$, i.e. the $\bfG$-action is given by a morphism $\bfG\times \bfX\to \bfX$. Following \cite[\S I.1.7]{Bor91}, we define the \eb{transporter} between subsets $M$ and $N$ of $\bfX(\overline{\bbK})$:
	\begin{equation} \label{eq:Tran}
	\Tran_\bfG(M,N) = \{ g\in \bfG\mid gM\subset N \}.
	\end{equation}
    If $M$ and $N$ are defined over $\bbK$, then $\Tran_\bfG(M,N)$ is also defined over $\bbK$. Let $\Tran_G(M,N)$ denote the $\bbK$-points of $\Tran_\bfG(M,N)$.
	
	We write $A=O(B)$ or $A\ll B$ or $B\gg A$ if $A\leq CB$ for some constant $C>0$, and $A\asymp B$ if $C^{-1}B\leq A\leq CB$ for some constant $C\geq 1$. For sequences $\{A_i\}_{i\in\bbN}$ and $\{B_i\}_{i\in\bbN}$ we write $A_i=o(B_i)$ if there exists a sequence $c_i\to0$ such that $A_i=c_iB_i$.
	
	For a finite field extension $\bbK/\bbF$ and an algebraic group $\bfG$ over $\bbK$, we write $\Res_{\bbK/\bbF}\bfG$ for the restriction of scalar (or Weil restriction) of $\bfG$ from $\bbK$ to $\bbF$. 
	
	\subsection*{Acknowledgement}
	Our discussions with Dmitry Kleinbock and Nicolas de Saxc\'e during our joint work in the case of $G=\SL_3(\bbR)$ have been greatly valuable for the general case. We would also like to thank Manfred Einsiedler, Alexander Gorodnik, and Jialun Li for their helpful discussions. We would like to thank the referees for providing numerous careful suggestions that have helped us improve the readability of this article.
	
	%%%%%%%%%%%%%%%%%%%%%%%%%%%%%%%%%%%%%%%%%%%
	%%%%%%%%%%%%%%%%%%%%%%%%%%%%%%%%%%%%%%%%%%%
	\section{Instability in invariant theory} \label{sect:instability}
	%%%%%%%%%%%%%%%%%%%%%%%%%%%%%%%%%%%%%%%%%%%
	%%%%%%%%%%%%%%%%%%%%%%%%%%%%%%%%%%%%%%%%%%%
	%\subsection{Reduction to highest weight vectors in fundamental representations}
	In this section, we define and study the unstable vectors in representations. Our main tool is Kempf's numerical criterion from geometric invariant theory.

\newcommand{\K}{\bbK}
	
Let $\bfG$ be a connected reductive algebraic group defined over a field $\bbK$ of characteristic $0$. 
 There exists a $\bfG(\bbK)$-invariant norm $\norm{\cdot}$ on the set $X_*(\bfG):=\Hom(\Gm,\bfG)$ of cocharacters of $\bfG$ defined over $\K$. 
	
	Let $\rho: \bfG \to \GL(V)$ be a representation of $\bfG$ defined over $\K$. Let us first recall some notations from \cite{Kem78}. 
	For any $\lambda\in X_\ast(\bfG)$ and $v\in V$, we have the decomposition $v=\sum_{i\in\bbZ}v_i$, where $\lambda(t)v_i=t^iv_i$ for all $i$. Let 
 \begin{equation} \label{eq:mvdelta}
 m(v,\lambda)=\min\{i\in\bbZ:v_i\neq 0\}.
 \end{equation}
	
 We say that a nonzero vector $v$ is \eb{unstable} if the Zariski closure of $\bfG v$ contains the origin. Given any $\lambda\in X_\ast(\bfG)$, we associate a {\em parabolic subgroup\/} $\bfP(\lambda)$ defined over $\bbK$, and its {\em unipotent radical\/} $\bfU^+(\lambda)$ defined over $\bbK$ such that for any field $\bbF\supset\bbR$ we have
 %for any field $\bbF\supset \bbK$, its $\bbF$- points
\begin{align} \label{eq:def-parabolic}
\bfP(\lambda)(\bbF)&=\{ g\in\bfG(\bbF) \colon \lim_{t\to0}\lambda(t)g\lambda(t)^{-1} \text{ exists} \}\\
\label{eq:def-urad}
\bfU^+(\lambda)(\bbF)&=\{g\in\bfG(\bbF) \colon \lim_{t\to0}\lambda(t)g\lambda(t)^{-1}=e\}.
\end{align}
  
	Let us recall the following important theorem of Kempf \cite[Theorem 4.2]{Kem78}. The form we state here is a special case of Kempf's original theorem.
	
	\begin{thm}[Kempf, 1978]\label{thm:Kempf_main_theorem_4.2}
		Let $0\neq v\in V(\K)$ be an unstable vector. Then
		\begin{enumerate}
			\item The function $\lambda\mapsto m(v, \lambda)/
			\norm{\lambda}$ has a positive maximum value $B_v$.
			\item Let $\Lambda_v$ be the set of indivisible $\K$-cocharacters $\lambda$ such that $m(v, \lambda)=B_v\norm{\lambda}$. Then
			\begin{enumerate}
				\item $\Lambda_v$ is non-empty.
				\item There is a parabolic $\K$-subgroup $\bfP_v$ of $\bfG$ such that $\bfP_v=\bfP(\lambda)$ for any $\lambda\in\Lambda_v$.
				\item \label{itm:2.1c} $\Lambda_v$ is a principal homogeneous space under conjugation by the $\K$-points of the unipotent radical of $\bfP_v$.
				\item Any $\K$-maximal torus of $\bfP_v$ contains the image of a unique member of $\Lambda_v$.
    \item   \label{itm:2.1e} For any $g\in \bfG(\bbK)$, $\Lambda_{gv}=g\Lambda_v g^{-1}$ and $\bfP_{gv}=g\bfP_vg^{-1}$. \cite[Corollary 3.5]{Kem78}.
    \item For any $l\in \bfP_v(\bbK)$, $\Lambda_{lv}=\Lambda_v$ and $\bfP_{lv}=\bfP_v$. ( Consequence of (\ref{itm:2.1c}) and (\ref{itm:2.1e}).)
			\end{enumerate}
		\end{enumerate}
	\end{thm}
	
	If necessary, we will write $\Lambda_v^{\bfG}$ instead of $\Lambda_v$ to allow the ambient group to vary.

 In the remaining part of the section, we will further assume that $\bfG$ is a connected $\bbK$-split semisimple group\footnote{For the purpose of this article, it is sufficient to focus on the case of $\bfG=\SL_n$ and $\bbK=\bbQ$.}. We pick a maximal $\bbK$-split torus $\bfS$ of $\bfG$. Then $\bfS$ is also a maximal torus in $\bfG$ defined over $\bbK$.

 \begin{prop}\label{prop:reduction_to_eigenvector}
		Let $V$ be a representation of $\bfG$ defined over $\K$. Let $v$ be an unstable vector in $V(\K)$. Then there exists an irreducible representation $W$ of $G$ defined over $\K$, a highest weight vector ${w}\in W(\K)$, an element $g_0\in \bfG(\K)$, and constants $C>0$, $\beta>0$ such that for any $g\in G$ one has
		\begin{equation}\label{eq:temp1}
		\lVert gg_0{w} \rVert \leq C \lVert gv \rVert^\beta.
		\end{equation}
	\end{prop}

First we will recall some definitions, facts, and notation.
 
\subsection{Root system and highest weights}  \label{subsec:roots}
 
    Let $\cocharS=\Hom(\Gm, \bfS)\cong \bbZ^r$ denote the group of cocharacters of $\bfS$ defined over $\K$. Let $\charS=\Hom(\bfS,\Gm)\cong \bbZ^r$ denote the group of characters of $\bfS$ defined over $\K$. We have a pairing $\langle \cdot ,\cdot\rangle: X^\ast(\bfS)\times X_\ast(\bfS)\to \Hom(\Gm,\Gm)\cong \bbZ$ given by $\langle \chi, \lambda\rangle=\chi\circ\lambda$; that is, $\chi(\lambda(t))=t^{\langle \chi,\lambda\rangle}$ for any $t\in\bbC^\ast$. This pairing is a {\em perfect (dual) pairing}, meaning the canonical inclusions  $X^\ast(\bfS)\hookrightarrow\Hom(X_\ast(\bfS),\bbZ)$ and $X^\ast(\bfS)\hookrightarrow\Hom(X^\ast(\bfS),\bbZ)$, induced by the pairing, are isomorphisms.

    For the following material, the reader may refer to \cite[Appendix G. Root datum]{Conrad-notes}. 

    The group $\cW=N_{\bfG}(\bfS)(\bbK)/Z_{\bfG}(\bfS)(\bbK)=N_{\bfG}(\bfS)(\bbK)/\bfS(\bbK)$ is called the $\bbK$-Weyl group of $\bfG$ relative to $\bfS$. 
    
    Let $\Phi=\Phi(\bfG,\bfS)\subset X^\ast(\bfS)$ denote the set of (non-trivial) $\K$-roots on $\bfS$ for the Adjoint action of $\bfS$ on the Lie algebra of $\bfG$.  For each $\alpha\in \Phi$, $\bfG_\alpha:=[Z_\bfG((\ker\alpha)^\circ),Z_\bfG((\ker\alpha)^\circ)]$ is isomorphic to $\SL_2$ or $\PGL_2$ over $\bbK$, where $^\circ$ denote the irreducible component of the identity. We pick a unique co-root $\check{\alpha}\in X_\ast(\bfS)$ such that $\check{\alpha}(\Gm)$ is in $\bfS\cap \bfG_\alpha$ and $\langle \alpha,\check{\alpha}\rangle=2$. And we pick $n_\alpha\in N_{\bfG}(\bfS)(\bbK)\cap \bfG_\alpha(\bbK)$ that conjugates $\check{\alpha}$ to its inverse. We denote its image in $\cW$ by $w_\alpha$.
    We denote the actions of $w_\alpha$ on $X^\ast (\bfS)$ and $X_\ast(\bfS)$ by $s_{\alpha}$ and $s_{\check\alpha}$, respectively. Then $s_\alpha(\chi)=\chi-\langle \chi, \check\alpha\rangle \alpha$ and $s_{\check \alpha}(\delta)=\delta-\langle \alpha, \delta\rangle \check\alpha$. Also, $\langle s_\alpha \chi,\lambda\rangle=\langle\chi,s_{\check{\alpha}}\lambda\rangle$. Moreover, the Weyl group $\cW$ is generated by $\{w_\alpha:\alpha\in\Phi\}$ (see~\cite[Corollary~G.2.11]{Conrad-notes}).
    
    There exists a positive definite bilinear form $(\cdot,\cdot)$ on $X_\ast(\bfS)$ taking values in $\bbZ$, and it is invariant under the Weyl group $\cW$.  For example, let
    \[
    (\lambda,\lambda'):=\sum_{\alpha\in \Phi(\bfG,\bfS)} \langle \alpha,\lambda\rangle \langle \alpha, \lambda'\rangle.
    \]
    Let $\norm{\lambda}:=\sqrt{(\lambda,\lambda)}$. Being $\cW$-invariant, the norm $\norm{\cdot}$ on $X_\ast(\bfS)$ extends uniquely to a $\bfG(\bbK)$-invariant norm on $X_\ast(\bfG)$. This is so because the canonical injection 
    \begin{equation} \label{eq:modGK}
    \cW\backslash X_\ast(\bfS)\to \bfG(\bbK)\backslash X_\ast(\bfG)
    \end{equation}
    is a surjection \cite[20.19]{Bor69} and \cite[Lemma~2.1]{Kem78}.

Using this bilinear form and the perfect pairing $\langle \cdot,\cdot\rangle$, we can identify $X_\ast(\bfS)$ with $X^\ast(\bfS)$, via the map $\delta\mapsto \hat{\delta}$, defined by
\begin{equation} \label{eq:hat}
\langle\hat\delta,\lambda\rangle=(\delta,\lambda),\,\forall \lambda\in X_\ast(\bfS). 
\end{equation}
This identification gives a positive definite $\cW$-invariant integral bilinear form on $X^\ast(\bfS)$, also denoted by $(\cdot,\cdot)$. Then for all $\alpha\in \Phi$, $(\check{\alpha})^\wedge=2\alpha/(\alpha,\alpha)$ (see \cite[Prop.G.2.5]{Conrad-notes}).
%In view of the above discussion, $(X^\ast(\bfS)\otimes \bbR,\Phi(\bfG,\bfS))$ is a {\em root system\/}(cf.~\Cref{def:root-system}).

There exists $\lambda_0\in X_\ast(\bfS)$ such that $\langle \alpha,\lambda_0 \rangle\neq 0$ for all $\alpha\in \Phi$. Let $\Phi^+=\{\alpha\in\Phi:\langle \alpha ,\lambda_0\rangle>0\}$. There exists a unique
$\Delta=\{\alpha_1,\alpha_2,\dots, \alpha_r\}\subset \Phi^+$, called the set of {\em simple roots\/} relative to $\lambda_0$, such that every element of $\Phi^+$ can be expressed as a non-negative integral combination of elements of $\Delta$, and $\Delta$ is a basis of $X^\ast(\bfS)\otimes \bbQ$ over $\bbQ$. Then $\{\check{\alpha}_i:i=1,\ldots,r\}$ is a $\bbQ$-basis of $X_\ast(\bfS)\otimes \bbQ$. The elements of its dual
 basis  $\{\mu_1,\dots,\mu_r\}\subset X^\ast(\bfS)$ are called the $\bbK$-fundamental weights; that is, 
    $\langle \mu_i,\check{\alpha_j}\rangle$ is $0$ if $i\neq j$ and is $1$ if $i=j$. 

    Given $\delta\in X_\ast(\bfS)$, we have $\hat\delta=\sum_{i=1}^r k_i\mu_i$, 
    where 
\begin{align}
    k_i&=\langle \hat{\delta},\check\alpha_i\rangle\in\bbZ \nonumber\\
    &=(\delta,\check \alpha_i)=(\check\alpha_i,\delta)=\langle (\check\alpha_i)^\wedge,\delta\rangle
    = \langle \frac{2\alpha_i}{(\alpha_i,\alpha_i)},\delta\rangle=\frac{2}{(\alpha_i,\alpha_i)} \langle \alpha_i,\delta\rangle. \label{eq:dominant}
\end{align}

    We will call the elements of the set
    \begin{equation}
        X_\ast(\bfS)^+=\{\lambda\in X_\ast(\bfS):\langle\alpha,\lambda\rangle\geq 0,\,\forall \alpha\in\Delta\}
    \end{equation}
    {\em dominant cocharacters}, and 
    \begin{equation} \label{eq:pos-Weyl}
    X_\ast(\bfS)=\cW\cdot X_\ast(\bfS)^+.
    \end{equation}
    The elements of the following set are called {\em dominant integral weights}\/:
    \[
    X^\ast(\bfS)^+=\{\chi\in X^\ast(\bfS):\langle \chi,\check{\alpha}\rangle\geq 0,\,\forall \alpha\in\Delta\}.
    \]
    In other words, a character on $\bfS$ is a dominant integral weight if and only if it is a non-negative integral linear combination of the fundamental weights. By \eqref{eq:dominant}, 
    \[
    \delta\in X_\ast(\bfS)^+\iff \hat{\delta}\in X^\ast(\bfS)^+.
    \]

    Let $\bfN^+=\bfU^+(\lambda_0)$, see~\eqref{eq:def-urad}. It is a maximal unipotent subgroup of $\bfG$. Also, $\bfU^+(\lambda)=\bfN^+$ for all $\lambda\in X_\ast(\bfS)$ such that $\langle \alpha,\lambda\rangle>0$ for all $\alpha\in\Phi^+$. So, $\bfN^+$ depends only on $\Phi^+$.

    A nonzero vector, say $w$, in a finite-dimensional representation of $\bfG$ is called a {\em highest weight vector\/} if it is fixed by $\bfN^+$, and $\bfS$ acts on the line containing $w$ via a character, called a {\em highest weight}. Any highest weight is a dominant integral weight. Conversely, we have the following:

    \subsubsection*{Highest weight theorem.}  Every dominant integral weight is a highest weight of a unique (up to isomorphism) irreducible, finite-dimensional representation of $\bfG$ defined over $\bbK$. 
    
    Moreover, any irreducible $\bbK$-representation of $\bfG$ admits a unique highest weight, and the corresponding weight space is one dimensional and defined over $\bbK$.

\bigskip
We fix a maximal compact subgroup $K$ of $G$ such that $S=\bfS(\bbR)$ is invariant under the Cartan involution of $G$ associated with $K$. By Iwasawa decomposition, $G=KS^0N^+$. Without loss of generality, we may assume that any norm on a finite-dimensional representation of $G$ over $\bbR$ considered in this section is $K$-invariant.

\subsubsection{Example} \label{exa:SLn} Let $\bfG=\SL_n$ over $\bbK=\bbQ$ and $\bfS$ be the full diagonal subgroup of $\SL_n$. Then $X_\ast(\bfS)\cong \bbZ^{n-1}$, where
any $\delta\in X_\ast(\bfS)$ is given by
$\delta(t)=\diag(t^{m_1},\ldots,t^{m_n})$, where $(m_1,\ldots,m_{n-1})\in \bbZ^{n-1}$, and $m_n=-(m_1+\cdots+m_{n-1})$. Then $\Phi=\{\alpha_{i,j}: i\neq j,\, 1\leq i,j\leq n\}$, where $\alpha_{i,j}(\diag(t_1,\ldots,t_n))=t_i/t_j$. Then $\bfG_{\alpha_{i,j}}$ is the copy of $\SL_2$ corresponding to the coordinates $(i',j')$, where $i',j'\in\{i,j\}$. So, 
\[
\check \alpha_{i,j}(t)=\diag(1,\ldots,t,\ldots,t^{-1},\ldots,1)
\]
where $i$-th entry is $t$, $j$-the entry is $t^{-1}$. One chooses $\Delta=\{\alpha_i:=\alpha_{i,i+1}:1\leq i\leq n-1\}$. Then the fundamental weight $\mu_i(\diag(t_1,\ldots,t_n))=t_1\cdots t_i$.

We consider the bilinear form $(\cdot,\cdot)$ on $X_\ast(\bfS)\cong \bbZ^{n-1}$ given by 
\[
((m_1,\ldots,m_{n-1}),(m'_1,\ldots,m'_{n-1}))=\sum_{i=1}^n m_im'_i,
\]
where $m_n,m'_n$ are as defined as above. It is invariant under the Weyl group of $\bfS$; the Weyl group is represented by permutations of the standard basis. 

We note that given $\delta$ as above corresponding to $(m_1,\ldots,m_{n-1})\in\bbZ^{n-1}$, we have 
\[
\hat\delta=\sum_{i=1}^{n-1}(m_i-m_{i+1})\mu_i.
\]
Thus, $\hat\delta$ is a dominant integral weight if and only if $m_1\geq m_2\geq \cdots\geq m_n$. 

Also, $\bfN^+$ is the group of all $n\times n$ upper triangular matrices with $1$ in all the diagonal entries. We choose $K=\SO(n)$. Then $G=KS^0N^+$; here, $S^0$ consists of all diagonal matrices in $G$ with positive entries. 

Let $V$ denote the standard representation of $\SL_n$ with the standard basis $e_1,\ldots,e_n$. Then for each $1\leq i\leq n-1$, $W_i=\wedge^iV$ is an irreducible representation of $\SL_n$ defined over $\bbQ$ with the highest weight $\mu_i$ and $w_i:=e_1\wedge \ldots \wedge e_i\in W_i(\bbQ)$ is a highest weight vector.

	\begin{proof}[Proof of \Cref{prop:reduction_to_eigenvector}]
		
	Given a nonzero unstable $v\in V(\bbK)$,
    by \Cref{thm:Kempf_main_theorem_4.2}, 
  we pick $\lambda\in \Lambda_v$. By \eqref{eq:modGK} and \eqref{eq:pos-Weyl}, we can pick $g_0\in \bfG(\bbK)$ such that  $\delta:=g_0^{-1} \lambda g_0\in \bfX_\ast(\bfS)^+$. Let $v'=g_0^{-1}v$. Then $\delta\in X_\ast(\bfS)^+\cap \Lambda_{v'}$.

        Let $\hat\delta\in X^\ast(\bfS)$ be as in \eqref{eq:hat}. Then by \eqref{eq:dominant}, $\hat{\delta}$ is a dominant integral weight. Therefore, by the highest weight theorem, we pick an irreducible representation $W$ of $\bfG$ defined over $\bbK$ and a vector ${w}\in W(\bbK)\neq 0$ such that ${w}$ is fixed by $\bfN^+$ and $\bfS$ acts on the line containing ${w}$ via the character $\hat{\delta}$. 
    
	    Let 
     \[
     \beta={(\delta,\delta)}/{m(v',\delta)}>0,
     \]
        where the function $m(\cdot,\cdot)$ is defined in \eqref{eq:mvdelta} and $m(v',\delta)>0$ by \Cref{thm:Kempf_main_theorem_4.2}(1). Let $g_0=(n_1g_1)^{-1}\in\bfG(\bbK)$. Then $g_0v'=v$. So, to prove \eqref{eq:temp1}, it suffices to show that there exists $C>0$ such that for any $g\in G$ we have
		\[
		\norm{g{w}}\leq C\norm{gv'}^\beta.
		\]
		
  We argue by contradiction. Suppose there exists a sequence $g_i$ in $G$ such that 
		\begin{equation} \label{eq:temp3}
		     \lim_{i\to \infty} \lVert g_iv' \rVert^\beta / \lVert g_i{w} \rVert =0.
		\end{equation}

  Let $\bfS_1=\ker(\hat\delta)^\circ$, which is a $\bbK$-split subtorus of $\bfS$ of rank $r-1$. Then $S^0=\bfS(\bbR)^0=\delta(\bbR_{> 0})\bfS_1(\bbR)^0$. The centralizer of $\delta$ in $\bfG$, denoted by $Z_\bfG(\delta)$, is a connected reductive $\bbK$-subgroup of $\bfG$ (see~\cite[Prop.1.4.3]{Conrad-notes}). Let $\bfU_1=\bfN^+\cap Z_\bfG(\delta)$. Then $\bfN^+=\bfU_1\ltimes \bfU^+(\delta)$ and $N^+=U_1 U^+(\delta)$. Let $\bfH=\bfS_1[Z_\bfG(\delta),Z_\bfG(\delta)]$. Then $\bfH$ is a connected reductive $\bbK$-subgroup, and $\bfS_1$ is a maximal torus of $\bfH$. We note that $\bfU_1\subset [Z_\bfG(\delta),Z_\bfG(\delta)]$. Therefore, by the Iwasawa decomposition, 
  \[
  G=KS^0N^+=K\delta(\bbR_{>0})S_1U_1U^+(\delta)=K\delta(\bbR_{>0})HU^+(\delta).
  \]
  We express each $g_i=k_i\delta(\tau_i)h_iu_i$, where $k_i\in K$, $\tau_i>0$, $h_i\in S_1U_1\subset H$, and $u_i\in U^+(\delta)$.

  Then, since the norm is $K$-invariant and ${w}$ is fixed by $\bfS_1$ and $\bfN^+$, by \eqref{eq:hat}, 
		\begin{equation}\label{eq:temp2}
		\lVert g_i{w} \rVert=\lVert \delta(\tau_i){w}\rVert=\lvert\hat{\delta}(\delta(\tau_i))\rvert \cdot \lVert {w} \rVert=\tau_i^{\langle \hat{\delta},\delta\rangle}\norm{{w}}=\tau_i^{(\delta,\delta)}\norm{{w}}.
		\end{equation}

  Now we express $V=\oplus_{i\in\bbZ} V_i$, where for each $i$,
  \[
V_i=\{x\in V: \delta(t)x=t^ix,\,\forall t\in\bbC^\ast\}
  \]
  is defined over $\bbK$.
		Let $\pi:V\to V_{m(v',\delta)}$ be the corresponding projection, which is defined over $\bbK$. Since $\bfH$ centralizes $\delta(\bbG_m)$, each $V_i$ is $\bfH$-invariant, and hence $\pi$ is $\bfH$-equivariant.
		
		We claim that
  \begin{equation} \label{eq:u-fixes-pi}
  \pi(uv')=\pi(v'),\, \forall u\in U^+(\delta). 
  \end{equation}

  To prove this claim, let $u\in U^+(\delta)$. Then $\delta(t)u\delta(t)^{-1}\to e$ as $t\to 0$. By the definition of $m(v',\delta)$, see \eqref{eq:mvdelta}, $v'=\sum_{i\geq m(v',\delta)} v'_i$, where $v'_i\in V_i$ for each $i$. Now, for each $i$,
  \[
  t^{-i}\delta(t)uv'_i=t^{-i}(\delta(t)u\delta(t)^{-1})(\delta(t)v'_i)=(\delta(t)u\delta(t)^{-1})v'_i\to v'_i\text{, as $t\to 0$,}
  \]
  and hence $uv'_i\in v'_i+\oplus_{j>i} V_j$. Therefore,
  \[
  uv'\in v'_{m(v',\delta)}+\oplus_{j>m(v',\delta)} V_j. 
  \]
  Hence, $\pi(uv')=v'_{m(v',\delta)}=\pi(v')$, which proves the claim.

Since $g_i=k_i\delta(\tau_i)h_iu_i$, we get
		\begin{align}\lVert{g_iv'}\rVert&=\lVert{\delta(\tau_i)h_iu_iv'}\rVert \nonumber\\
    &\geq \lVert \pi(\delta(\tau_i)h_iu_iv') \rVert \nonumber\\
    &=\lVert \delta(\tau_i)\pi(h_iu_iv') \rVert \nonumber\\
    &=\tau_i^{m(v',\delta)}\lVert h_i\pi(u_iv') \rVert \nonumber\\
  &= \tau_i^{m(v',\delta)}\lVert h_i\pi(v') \rVert 
  \text{, by \eqref{eq:u-fixes-pi}. }
  \label{eq:temp4}
		\end{align}

Combining  \eqref{eq:temp2} and \eqref{eq:temp4}, 
  \[
  \lVert g_iv' \rVert^\beta / \lVert g_i{w} \rVert=\frac{\tau_i^{m(v',\delta)\beta}\lVert h_i\pi(v') \rVert^\beta}{\tau_i^{(\delta,\delta)}\norm{{w}}}=\frac{\lVert h_i\pi(v') \rVert^\beta}{\norm{{w}}},
  \]
  as $m(v',\delta)\beta=(\delta,\delta)$. So by \eqref{eq:temp3}, $\lVert h_i\pi(v') \rVert\to 0$ as $i\to\infty$. Therefore $\pi(v')\in V(\bbK)$ is $\bfH$-unstable in $V_{m(v',\delta)}$. 
  
  We apply \Cref{thm:Kempf_main_theorem_4.2} to the representation $V_{m(v',\delta)}$ of $\bfH$. Since $\bfS_1$ is a maximal $\bbK$-split torus of $\bfH$, there exists $l\in \bfH(\K)$ such that $\Lambda_{l\pi(v')}^{\bfH}$ contains a unique element of $X_{\ast}(\bfS_1)$, say $\delta_l$. Since $\bfS_1\subset \ker(\hat\delta)$,
		\begin{equation} \label{eq:perp}
		(\delta, \delta_l)=\langle\hat\delta,\delta_l\rangle=\hat\delta\circ \delta_l = 0\in \mathrm{Hom}(\Gm, \Gm). 
		\end{equation}
		Since $l\in H\subset Z_G(\delta)\subset P(\delta)=P_{v'}$, by \Cref{thm:Kempf_main_theorem_4.2} we have $\Lambda_{lv'}=\Lambda_{v'}$ and thus $\delta\in\Lambda_{lv'}$. By the definition of $\Lambda_v$ in \Cref{thm:Kempf_main_theorem_4.2}, we have $B_{lv'} = \frac{m(lv', \delta)}{\lVert\delta\rVert}$. Also, $m(lv',\delta)=m(v',\delta)$.
		
		For any positive integer $N$, let $\delta_N=N\delta + \delta_l\in X_\ast(S)$. For $N$ large enough, we claim that
		\begin{equation} \label{eq:destablize_faster}
		\frac{m(lv', \delta_N)}{\lVert\delta_N\rVert} > \frac{m(lv', \delta)}{\lVert\delta\rVert} = B_{lv'},
		\end{equation}
		which will contradict the maximality of $B_{lv'}$.
		
		To prove the claim, consider the weight space decomposition $V = \bigoplus V_{\chi}$, where $S$ acts on $V_{\chi}$ by multiplication via the character $\chi$ of $S$. For $v\in V$, let $v_\chi$ denote its $V_\chi$ component in the above decomposition.
		Let
		\[
		\begin{split}
		\Xi 
  &= \{ \chi\in\charS \mid (lv')_\chi\neq 0 \text{ and } \langle\chi,\delta \rangle=m(lv',\delta)=m(v',\delta) \} \\
		&=\{ \chi\in\charS \mid (\pi(lv'))_\chi\neq 0 \}.
		\end{split}
		\]

Since $\delta\in \Lambda_{lv'}$, there exists $R\in\bbZ$ such that for any $\chi\in\charS\setminus \Xi$ such that $(lv')_\chi\neq 0$, we have $\langle\chi,\delta \rangle\geq m(lv',\delta)+1$, and $\langle\chi,\delta_l \rangle\geq R$. Therefore, we can pick $N_0\in\bbN$ such that for all $\chi\in\charS\setminus \Xi$ such that $(lv')_\chi\neq 0$, we have
\[
\frac{\langle \chi, \delta_{N} \rangle}{\norm{\delta_{N}} } \geq \frac{m(lv',\delta)+1+R/N}{\norm{\delta}+\norm{\delta_l}/N}\geq \frac{m(lv',\delta)+1/2}{\norm{\delta}},\,\forall N\geq N_0.
\]

		So, to prove \eqref{eq:destablize_faster}, it suffices to show that for any $\chi\in\Xi$, for all sufficiently large $N$, 
		\begin{equation}\label{eq:higher_speed}
		\frac{\langle \chi, \delta_{N} \rangle}{\norm{\delta_{N}} } >
		\frac{\langle \chi, \delta \rangle}{\norm{\delta}}.
		\end{equation}

		To prove \eqref{eq:higher_speed}, we define an auxiliary function:
		\[
		f(s) 
		= \frac{\langle \chi, \delta + s\cdot \delta_l \rangle ^ 2}{\norm{\delta + s \cdot \delta_l} ^ 2}
		= \frac{\langle \chi , \delta \rangle ^ 2 + 2s\langle \chi, \delta \rangle\langle \chi, \delta_l \rangle + s^2 \langle \chi, \delta_l\rangle ^2}
		{(\delta, \delta) + 2s(\delta, \delta_l) + s^2(\delta_l, \delta_l)}.
		\]
		Compute its derivative at $0$:
		\[
		f'(0) = \frac{2\langle \chi,\delta\rangle\langle\chi,\delta_l\rangle(\delta, \delta) - 2(\delta, \delta_l)\langle\chi, \delta\rangle^2}
		{(\delta, \delta)^2}.
		\]
		Since $\delta_l\in \Lambda_{l\pi(v')}^{H}=\Lambda_{\pi(lv')}^{H}$ and $(\pi(lv'))_\chi\neq 0$, we know that $\langle \chi, \delta_l \rangle > 0$. Also, we know that $\langle \chi, \delta \rangle = m(v',\delta)>0$. And $(\delta, \delta_l)=0$ by \eqref{eq:perp}. Therefore $f'(0) > 0$. Hence, for $N$ large, we have
	\begin{equation}\label{eq:temp14}
		f(1/N) > f(0),
		\end{equation}
		and \eqref{eq:higher_speed} follows because each side of \eqref{eq:temp14} is the square of each side of \eqref{eq:higher_speed} respectively. 
\end{proof}

\begin{remark}
   Our proof of \Cref{prop:reduction_to_eigenvector} does not use the irreducibility of $W$. We proved the following: Let $V$ be a $\bbK$-representation of $\bfG$ and suppose that $v\in V(\bbK)$ is a nonzero $\bfG$-unstable vector. By \Cref{thm:Kempf_main_theorem_4.2}, pick a $g_0\in \bfG(\bbK)$ such that $g_0^{-1}\Lambda_{v}g_0\cap X_\ast(\bfS)^+=\{\delta\}$. Let $\beta=(\delta,\delta)/m(g_0^{-1}v,\delta)>0$. Let $\hat\delta\in X^\ast(\bfS)^+$ be such that $\langle \hat\delta,\lambda\rangle=(\delta,\lambda)$ for all $\lambda\in X_\ast(\bfS)$. Let $W$ be any $\bbK$-representation of $\bfG$ with a highest weight $w\in W(\bbK)$ corresponding to weight $\hat\delta$. Then there exists a $C>0$ such that $\norm{gg_0w}\leq C\norm{gv}^\beta$ for all $g\in G$.  
\end{remark}

    We recall a lemma from \cite{Kem78}.
    
    \begin{lem}[\cite{Kem78}, Lem.1.1]\label{lem:Kempf 1.1}
        Let $X$ be an affine $\bfG$-scheme over $\bbK$ and let $Y$ be a closed $\bfG$-subscheme of $X$ over $\bbK$. Then there exists a $\bfG$-equivariant morphism $f:X\to W$ over $\bbK$, where $W$ is a finite dimensional representation of $\bfG$ over $\bbK$, such that $Y=f^{-1}(0)$.
    \end{lem}

    \begin{cor}\label{cor:reduction_to_unstable}
        Let $V$ be a representation of $\bfG$ over $\bbK$, and $v\in V(\bbK)$ such that $\bfG v$ is not Zariski closed. Then there exists an irreducible representation $W$ of $\bfG$ defined over $\bbK$, a highest weight vector $w\in W(\bbK)$ with nonzero weight $\mu\in \charS^+$, an element $g_0\in \bfG(\bbK)$, and a constant $\beta>0$ with the following property: for any $R>0$, there exists a constant $C>0$ such that for any $g\in G$,  
        \[
        \text{ if $\norm{gv}\leq R$, then $\norm{gg_0w}\leq C\norm{gv}^\beta$.}
        \]
    
    \end{cor}
\newcommand{\Zcl}{\mathrm{Zcl}}
    
    \begin{proof}
    By closed orbit lemma~\cite[I.1.8]{Bor91}, $\bfG v$ is Zariski open in its Zariski closure $\Zcl(\bfG v)$ in $V$.  Since $v\in V(\bbK)$, $\bfG(\bbK)$ is Zariski dense in $\bfG$ (see~\cite[V.18.3]{Bor91}), and $\bfG$ acts on $V$ over $\bbK$, we have $\bfG(\bbK)v$ is Zariski dense in $\bfG v$. Therefore $Y=\Zcl(\bfG v)\setminus \bfG v$ is a non-empty Zariski closed subset of $V$ defined over $\bbK$.
    
        Therefore, by \Cref{lem:Kempf 1.1}, there exists a finite dimensional representation $V_1$ of $\bfG$ defined over $\bbK$ and a $\bfG$-equivariant morphism $f:V\to V_1$ defined over $\bbK$ such that $f(v)\neq 0$, and $f(Y)=\{0\}$. Since $Y\subset \Zcl(\bfG v)$, we have $0\in Zcl(\bfG f(v))$. Therefore $f(v)$ is unstable in $V_1$. Also $f(v)\in V_1(\bbK)$. Therefore, by \Cref{prop:reduction_to_eigenvector}, there exist an irreducible representation $W$ of $\bfG$ defined over $\bbK$, and a highest weight vector $w\in W(\bbK)$, an element $g_0\in \bfG(\bbK)$, and constants $\beta>0$ and $C_1>0$ such that for any $g\in G$ we have
        \[
        \norm{gg_0w}\leq C_1\norm{gf(v)}^\beta.
        \]
        Since $f$ is a continuous map, given $R>0$, there exists a constant $C_2>0$ such that for any $v'\in V$ with $\norm{v'}\leq R$, we have $\norm{f(v')}\leq C_2\norm{v'}$. Now, the conclusion of the corollary holds for $C=C_1C_2^\beta$. 
    \end{proof}

    Let the notation be as in \S\ref{subsec:roots}. For each fundamental weight $\mu_i$, where $1\leq i\leq r$, let $W_i$ be the (unique) irreducible representation of $\bfG$ over $\bbK$ with the highest weight $\mu_i$. We pick a highest weight vector $w_i\in W_i(\bbK)$ for each $i$. The $W_i$'s are called the {\em fundamental representations\/} of $\bfG$ over $\bbK$.
	
	\begin{lem}\label{lem:reduction_to_fundamental_rep}
		Let $W$ be an irreducible representation of $\bfG$ defined over $\bbK$ with the highest weight $n_1\mu_1+\cdots+n_r\mu_r$, where each $n_i\in\bbZ_{\geq0}$, and let $w\in W(\bbK)$ be a highest weight vector. Let $N=n_1+\cdots+n_r$. Let $\Omega$ be a non-empty compact subset of $G$ whose Zariski closure is irreducible. Then there exists a constant $C>0$ such that for any $h_1,h_2\in G$,
		\[
		\sup_{\omega\in \Omega} \norm{h_1\omega h_2w} \geq C\cdot\left(\min_{1\leq i\leq r}\sup_{\omega\in \Omega} \norm{h_1\omega h_2w_i}\right)^N.
		\]
	\end{lem}
	
	\begin{proof}
		By Iwasawa decomposition, for $g\in G$, we can write $g=ktu$ for $k\in K, t\in S^0$ and $u\in N^+$. Being highest weights, $w_i$'s and $w$ are fixed by $N^+$. Without loss of generality, we may assume that the norms on $W_i$'s and $W$ are induced by $K$-invariant inner products. Therefore 
  \[
  \norm{gw_i}=\mu_i(t)\norm{w_i},\,\forall i \text{, and }\norm{gw}=\left(\prod_{i=1}^r \mu_i(t)^{n_i}\right)\norm{w}.
  \]
  We re-scale the inner product on $W$ such that $\norm{w}=\prod_{1\leq i\leq r}\norm{w_i}^{n_i}$. Then 
  \[
  \norm{gw} = \prod_{1\leq i\leq r}\norm{gw_i}^{n_i},\, \forall g\in G.
  \]
		
		Now let $F(g)=\norm{h_1gh_2w}^2$ and $F_i(g)=\norm{h_1gh_2w_i}^2$, $\forall i$. Then $F$ and $F_i$'s are regular functions on $G$, and $F(g)=\prod_{1\leq i\leq r}F_i(g)^{n_i}$. Let $Z$ be the Zariski closure of $\Omega$ in $G$; by our assumption, $Z$ is an irreducible algebraic set. We use the norm $\norm{F}=\sup_{\omega\in\Omega}\abs{F(\omega)}$ on the space of regular functions on $Z$. 
  
  We claim that for any positive integers $d_1$ and $d_2$, there exists a constant $c=c(d_1,d_2)>0$ such that for any polynomials $E_1$ and $E_2$ of degrees $d_1$ and $d_2$ respectively on $Z$, we have $\norm{E_1E_2}\geq c\norm{E_1}\norm{E_2}$. Indeed, by homogeneity, we only need to check this for $\norm{E_1}=\norm{E_2}=1$, and then the possible values of $\norm{E_1E_2}$ form a compact subset of $\bbR_{>0}$. 
  
  Therefore, there exists a constant $C>0$ such that
		\[
		\norm{F}\geq C\prod_{1\leq i\leq r}\norm{F_i}^{n_i}\geq C\cdot\left(\min_{1\leq i\leq r} \norm{F_i}\right)^{N}.
		\] 
	\end{proof}

	%%%%%%%%%%%%%%%%%%%%%%%%%%%%%%%%%%%%%%%%%%%
	%%%%%%%%%%%%%%%%%%%%%%%%%%%%%%%%%%%%%%%%%%%

 \section{Expansion in linear representations} \label{sect:transporters}
	%%%%%%%%%%%%%%%%%%%%%%%%%%%%%%%%%%%%%%%%%%%
	%%%%%%%%%%%%%%%%%%%%%%%%%%%%%%%%%%%%%%%%%%%
	In this section, we first describe a certain transporter and then prove several results on expansion in linear representations of special linear groups.

 \newcommand{\Udel}{{\bfU^-(\delta)}}
 \newcommand{\udel}{U^-(\delta)}
 
	\subsection{Description of a transporter}
        Let $\bfG$ be a connected reductive algebraic group over $\bbR$. Let $\delta\colon \Gm\to \bfG$ be a cocharacter of $\bfG$ defined over $\bbR$. Let $\bfP(\delta)$ denote the parabolic subgroup of $\bfG$ associated to $\delta$ as in \eqref{eq:def-parabolic}. Let 
        \begin{equation} \label{eq:Horo}
        \udel=U^+(\delta^{-1})=\{g\in G: \lim_{t\to 0} \delta(t)^{-1} g \delta(t)=e\},
        \end{equation}
  which is called the expanding horospherical subgroup of $G$ associated to $\delta(t)$ for $\abs{t}<1$.

 \begin{lem} \label{lem:product structure}
 Let $\bfH$ be a reductive $\bbR$-subgroup of $\bfG$ containing $\delta(\Gm)$. Then
  \[
  H\cap P(\delta)\udel=(P(\delta)\cap H)(\udel\cap H).
  \]
 \end{lem}

 \begin{proof}
Let $\bfT$ denote a maximal $\bbR$-split torus of $\bfH$ containing $\delta(\Gm)$. Then $P(\delta)\cap H$ is a parabolic subgroup of $H$ containing $T$ and $U^-(\delta)\cap H$ is the unipotent radical of $P(\delta)\cap H$. Let $N$ denote the normalizer of $T$ in $H$. Then, by Bruhat decomposition on $H$ (see \cite[21.15]{Bor91}), \[
H=(P(\delta)\cap H)N(\udel\cap H).
\]

Let $h\in H$. Then $h\in (P(\delta)\cap H)nu$ for some $n\in N$ and $u\in \udel\cap H$. 
Since $nT n^{-1}=T\subset P(\delta)\cap H$, and $\delta(\bbR^\ast)\subset T$, we have  
\[
(P(\delta)\cap H)n\delta(\bbR^\ast)=(P(\delta)\cap H)n.
\]
Therefore, in the quotient space $(P(\delta)\cap H)\backslash H$, as $t\to 0$, we have 
\begin{equation} \label{eq:htow}
(P(\delta)\cap H)h\delta(t)=(P(\delta)\cap H)(n\delta(t))(\delta(t)^{-1}u\delta(t))\to (P(\delta)\cap H)n.
\end{equation}

Further, suppose that $h\in P(\delta)\udel$. Then, as $t\to 0$, we have
\begin{equation} \label{eq:PhtoP}
P(\delta)h\delta(t)\to P(\delta) \text{ in $P(\delta)\backslash G$.}
\end{equation}
Since $(P(\delta)\cap H)\backslash H$ is compact, the natural injection $(P(\delta)\cap H)\backslash H\hookrightarrow P(\delta)\backslash G$ is a proper continuous map. Therefore, by \eqref{eq:PhtoP}, as $t\to 0$, we have
\[
(P(\delta)\cap H)h\delta(t)\to (P(\delta)\cap H) 
\text{ in $(P(\delta)\cap H)\backslash H$.}
\]
So, by \eqref{eq:htow}, we conclude that
$(P(\delta)\cap H)n=(P(\delta)\cap H)$.
Hence
\[
h\in (P(\delta)\cap H) nu\subset (P(\delta)\cap H)(\udel\cap H).
\]
\end{proof}
	
	Let $\bfW$ be a Zariski closed subset of $\bfG$ containing the identity $e$ of $\bfG$. We write $\bfW^{\circ}$ for the union of irreducible components of $\bfW$ containing $e$. Let $V$ be a finite-dimensional representation of $\bfG$ defined over $\bbR$. Let $V_{\geq0}(\delta)$ denote the direct sum of non-negative weight spaces of $V$ with respect to $\delta$; that is, the weight spaces where $\delta(t)$ acts as non-negative powers of $t$. In other words,
 \begin{align*}
    V_{\geq 0}(\delta)=\{v\in V: \lim_{t\to 0}\delta(t)v \text{ exists in }V\}=\{v\in V: \lim_{t\to 0} \norm{\delta(t)v}<\infty\}.
    %\\
    %\bfP(\delta)&=\{g\in \bfG: \lim_{t\to 0} \delta(t)g\delta(t)^{-1} \text{ exists in } \bfG\}.
\end{align*}
 As in \eqref{eq:Tran},
 \[
 \Tran_G\left(w,V_{\geq0}(\delta)\right):=\{g\in G: gw\in V_{\geq0}(\delta)\}.
 \]

	\begin{lem}\label{lem:local_characterization_of_moving_into_weakly_stable}
		%Let $\bfG$ be a connected reductive group over $\bbR$. Let $\delta\colon\Gm\to \bfG$ be a cocharacter of $\bfG$ defined over $\bbR$. Let $V$ be a finite-dimensional representation of $\bfG$ over $\bbR$. 
  Suppose that $w\in V(\bbR)\setminus\{0\}$ is such that the isotropy group $\bfG_w$ of $w$ is reductive and contains $\delta(\Gm)$. Then we have the following sets of equalities:
		\begin{gather}
		   \Tran_G\left(w,V_{\geq0}(\delta)\right)\cap P(\delta)\udel=P(\delta)G_w\cap P(\delta) \udel=P(\delta)(G_w\cap \udel),
     \label{eq:ZU-}\\
		 \Tran_G\left(w,V_{\geq0}(\delta)\right)^\circ=(P(\delta)G_w)^{\circ}=P(\delta)G_w^{\circ}. \label{eq:Zo}
		\end{gather}
	\end{lem}
	\begin{proof}
We write $\bfZ=\Tran_\bfG\left(w,V_{\geq0}(\delta)\right)$, which is Zariski closed by \cite[Sect.I.1.7]{Bor91}. Since $w$ is fixed by $\delta(\Gm)$, $\bfZ$ contains the identity $e$. Using the definition of $P(\delta)$, it is straightforward to verify that $P(\delta)V_{\geq0}(\delta)=V_{\geq0}(\delta)$. Therefore $P(\delta)Z=Z=ZG_w$. In particular, 
$P(\delta)G_w\subset Z$. Hence 
\[
P(\delta)G_w\cap P(\delta) \udel\subset Z\cap P(\delta)\udel.
\]

Since $\delta(\Gm)$ is contained in the reductive subgroup $\bfG_w$ of $\bfG$, by \Cref{lem:product structure}, 
\[
G_w\cap P(\delta) \udel=(G_w\cap P(\delta))(G_w\cap \udel).
\]
Hence 
\[
P(\delta)G_w\cap P(\delta) \udel = P(\delta)(G_w\cap \udel).
\]

Since $P(\delta)Z=Z$, we have $Z\cap P(\delta)\udel=P(\delta)(Z\cap \udel)$. Therefore, to justify \eqref{eq:ZU-}, it remains to show that 
\[
Z\cap\udel\subset G_w.
\]

To prove this, since $\udel$ is a unipotent algebraic subgroup of $G$, the orbit $\udel w$ is Zariski closed in $V$, see \cite[Theorem~12.1]{Birkes71}. Hence, the map $f:\udel/(G_w\cap \udel)\to V$, given by $f(u(G_w\cap \udel))=uw$ for all $u\in \udel$, is a well-defined, injective, proper, and continuous. 

Let $u\in Z\cap \udel$. 
Then $uw\in V_{\geq0}(\delta)$.
So, $\lim_{t\to0}\delta(t)uw$ exists in $V$. Since $\delta(\Gm)$ fixes $w$, 
\[
\delta(t)uw=(\delta(t)u\delta(t)^{-1})w=f((\delta(t)u\delta(t)^{-1})(G_w\cap \udel))
\]
converges in $V$ as $t\to 0$. Since $f$ is a proper map, there exists a compact set $\Omega\subset \udel$ and $r>0$ such that $\delta(t)u\delta(t)^{-1}\subset \Omega(G_w\cap \udel)$ for all $\abs{t}<r$. Hence $u\in (\delta(t)^{-1}\omega\delta(t))(\bfG_w\cap \udel)$ for all $\abs{t}<r$. From \eqref{eq:Horo}, since $\udel$ is finite dimensional, we deduce that $\delta(t)^{-1}\Omega\delta(t)\to \{e\}$ as $t\to 0$. Therefore, $u\in G_w$. Thus, $Z\cap \udel\subset G_w$. This completes the proof of  \eqref{eq:ZU-}.

Since $P\udel$ is Zariski open dense in $G$ containing $e$, \eqref{eq:Zo} follows from \eqref{eq:ZU-}.
%due to \Cref{lem:product_is_Zariski_closed}.
 \end{proof}

	\subsection{Expansion in linear representations of special linear groups}
	Let $d\geq2$ be an integer. Let $\bfG=\SL_d$, and $G=\SL_d(\bbR)$. Let $\rho:\bfG\to\GL(V)$ be a representation of $\bfG$. Let $\bfP_1$ be the parabolic subgroup of $\bfG$ which is the stabilizer of the $(d-1)$-space spanned by $e_2,\dots,e_d$ in the standard representation.
	In this subsection, we prove several results on the expansion phenomenon in representations of $\bfG$.
	
	We first recall a variant of Shah's \emph{basic lemma} (cf.~\cite[Corollary 4.6]{Sha09Invention}):
	\begin{lem}\label{lem:Shah_basic_lemma}
		Let $\bfG=\SL_d$ over $\bbR$ for $d\geq 2$, and let $\lambda:\Gm\to\bfG$ be the multiplicative one-parameter subgroup defined by $\lambda(t)=\diag (t^{-(d-1)}, t, \dots, t)$ for all $t\in\Gm$. 
  Consider a representation of $\bfG$ on a finite-dimensional vector space $V$ defined over $\bbR$. 
  Then for any $v\in V(\bbR)$ such that $v$ is not fixed by $G$, the image of the transporter $\Tran_G(v,V_{\geq0}(\lambda))$ under the natural projection $\pi:G\to P_1\backslash G$ is contained in a union of at most two proper linear subspaces of $P_1\backslash G\cong\bbP^{d-1}(\bbR)$.
	\end{lem}
 
	\begin{proof} 
            We note that $\bfP_1=\bfP(\lambda)$. Let $\bfS$ be the full diagonal subgroup of $\bfG$. We choose a set of simple roots such that $\lambda\in X_\ast(\bfS)^+$, see \eqref{eq:pos-Weyl}. For example, take $\mathbf{t}=\diag(t_1,\dots,t_d))\mapsto t_{i+1}t_i^{-1}$, for $1\leq i\leq d-1$, as the set of simple roots. Let $W=N_G(S)/S$ and $W_P=N_{P_1}(S)/S$. Let $W^P=W/W_P$. Let $\{\sigma_1,\dots, \sigma_d \}\subset N_{\bfG}(\bfS)(\bbQ)$ denote the set of representatives of $W^P$, where $\sigma_i\cdot\lambda(t)$ has entry $t^{-(d-1)}$ at the $i$-th diagonal position, and the rest of the diagonal entries are $t$. Let $B$ be the standard minimal parabolic subgroup of $G$ such that $B\subset P_1$; that is, $B$ is the group of all lower triangular matrices in $G$.
            
		We first consider the case where $\bfG v$ is not Zariski closed. 
            By \Cref{cor:reduction_to_unstable}, there exists a finite-dimensional irreducible representation $W$ of $\bfG$ defined over $\bbR$, a highest weight vector ${w}\in W(\bbR)$ with nonzero weight $\mu\in \charS^+$, an element $g_0\in \bfG(\bbR)$, and $\beta>0$ such for any $R>0$ there exists a constant $C=C_R>0$ such that 
            \begin{equation} \label{eq:VW}
          \text{$\forall g\in G$, if $\norm{gv}\leq R$, then $\norm{gg_0{w}}\leq C\norm{gv}^\beta$.}
            \end{equation}
        From this, we will deduce the following: 
            \begin{equation} \label{eq:TranVW}
                \Tran_G(v, V_{\geq0}(\lambda)) 
                \subset \Tran_G(g_0{w}, W_{\geq0}(\lambda)) 
                = \Tran_G({w}, W_{\geq0}(\lambda))g_0^{-1}.
            \end{equation}
            
            To verify this, let $g\in G$ be such that $gv\in V_{\geq 0}(\lambda)$. Then $\norm{\lambda(t)gv}\leq R$ for all $0<\abs{t}\leq t_0$, for some $t_0>0$ and $R>0$. Then by \eqref{eq:VW}, $\norm{\lambda(t)gg_0{w}}\leq CR^\beta$ for all $0<\abs{t}\leq t_0$. Therefore $gg_0{w}\in W_{\geq 0}(\lambda)$. So $gg_0\in \Tran_G({w},W_{\geq 0}(\lambda))$. This verifies \eqref{eq:TranVW}.
        
            Suppose $g\in \Tran_\bfG({w}, W_{\geq0}(\lambda))$. Then $g{w}\in W_{\geq0}(\lambda)$. Consider the Bruhat decomposition (see \cite[21.15]{Bor91}) $G=\sqcup_{\sigma\in W^P} P_1\sigma^{-1}B$, and suppose $g\in P_1\sigma^{-1}B$. Since ${w}$ is fixed by $B$, and $W_{\geq0}(\lambda)$ is $P_1$-invariant, it follows that $\sigma^{-1}{w}\in W_{\geq0}(\lambda)$. Since
            \[
            \lambda(t)\sigma^{-1}{w}=\sigma^{-1}(\sigma \lambda(t)\sigma^{-1}){w}=\sigma^{-1}\mu(\sigma \lambda(t)\sigma^{-1}){w}=t^{\langle\mu,\sigma\cdot\lambda\rangle}\sigma^{-1}{w}, \quad \forall t\in\bbR^\ast,
            \]
            we conclude that $\langle\mu,\sigma\cdot\lambda\rangle\geq0$. 
            Note that $\mu$ is of the form $\mathbf{t}\mapsto t_1^{a_1}t_2^{a_2}\cdots t_d^{a_d}$, where $a_1\leq \cdots \leq a_d$ are integers, not all $0$, satisfying $a_1+\cdots+a_d=0$. So $a_d>0$. Let $j$ be the largest integer such that $a_j\leq0$. So $j<d$. For any $1\leq k\leq d$, 
            \begin{align*}
                \langle\mu, \sigma_k\cdot\lambda\rangle&=(a_1,\ldots,a_{k-1},a_k,a_{k+1},\ldots,a_d)\cdot(1,\ldots,1,-(d-1),1,\ldots,1)=-a_kd.
            \end{align*}
            Therefore, since $\langle\mu, \sigma\cdot\lambda\rangle\geq0$, we get $\sigma\in \{\sigma_1,\dots,\sigma_j\}$. Hence, 
            \[
            \Tran_G({w},W_{\geq0}(\lambda))\subset\bigsqcup_{i=1}^jP_1\sigma_i^{-1}B=P_1 H_{j},
            \]
            where $H_j$ denotes the stabilizer of $e_1\wedge\cdots\wedge e_j$ in $\bigwedge^j\bbR^n$. 
            Note that the image of $P_1H_j$ under $\pi$ is a proper linear subspace of $P_1\backslash G\cong\bbP^{d-1}(\bbR)$ of dimension $j-1<d-1$. Since the $G$-action sends linear subspaces to linear subspaces of the same dimensions, by \eqref{eq:TranVW}, the image of $\Tran_\bfG(v, V_{\geq0}(\lambda))$ under $\pi$ is also contained in a proper linear subspace.
		
		Now assume that $\bfG v$ is Zariski closed. 
		Take any $g\in\Tran_G(v,V_{\geq0}(\lambda))$. Then$\lambda(t)gv$ converges to some $w\in G v$ as $t\to0$. In particular, $w$ is fixed by $\lambda(\Gm)$. Also, since $\bfG w=\bfG v$ is Zariski closed, $G/G_w\cong Gw$ is an affine variety. And since $G$ is reductive, it follows from Matsushima's criterion (see~\cite[7.10]{Bor69}) that $G_w$ is reductive. 
		%It remains to apply \Cref{lem:local_characterization_of_moving_into_weakly_stable} for $w$.
    So, by \Cref{lem:local_characterization_of_moving_into_weakly_stable}, $\Tran_\bfG(w, V_{\geq0}(\lambda))$ is contained in the union of $P(\lambda)(G_w\cap\udel)$ and the complement of $P(\lambda)U^-(\lambda)$. Since $G$ is a simple Lie group, the smallest reductive subgroup of $G$ containing $\lambda(t)$, for some $0<\abs{t}<1$, and the expanding horospherical subgroup $U^-(\lambda)$ of $\lambda(t)$ (see \eqref{eq:Horo}) in $G$ equals $G$. Since $w$ is not $G$-fixed, and $\lambda(\bbR^\ast)\subset G_w$ we get that $G_w\cap U^-(\lambda)$ is a proper subgroup of $U^-(\lambda)$. Now the images of $P(\lambda)(G_w\cap U^{-}(\lambda))$ and the complement of $P(\lambda)U^-(\lambda)$ under the projection $\pi$ are proper linear subspaces of $P_1\backslash G\cong \bbP^{d-1}(\bbR)$. 
	\end{proof}
	
	The following technical lemma will be useful.
	
	\begin{lem}\label{lem:quantitative_Shah_basic_lemma}
		Let $G=\SL_d(\bbR)$ for $d\geq2$, and $g_t=\diag (e^{(d-1)t},e^{-t},\dots, e^{-t})$. Let $\Omega$ be a compact subset of $G$, such that the image $\pi(\Omega)$ of $\Omega$ under the natural projection $\pi:G\to P_1\backslash G$ is not contained in a union of any two proper linear subspaces of $P_1\backslash G\cong\bbP^{d-1}(\bbR)$. Let $V$ be a finite-dimensional real representation of $G$ with a norm $\norm{\cdot}$. 
  %Let $\beta$ be the smallest positive eigenvalue of the $g_1$-action on $V$. 
        There exists $D>0$ such that for any $v\in V$,
    \begin{equation} \label{eq:Dnormv}
    \sup_{\omega\in\Omega}\norm{g_t\omega v}\geq D\norm{v}.
    \end{equation}
    
    In fact, for any $v\in V$ not fixed by $G$, there exists $C(v)>0$ such that for all $t\geq0$, 
		\begin{equation} \label{eq:exponential-expansion}
		\sup_{\omega\in\Omega}\norm{g_t\omega v}\geq C(v)e^{t}\norm{v}.
		\end{equation}
  
		Moreover, if there is no nonzero $G$-fixed vector in $V$, then there exists a constant $C>0$ such that for all $v\in V$, and $t\geq 0$,
        \begin{equation} \label{eq:noGfixed}
		\sup_{\omega\in\Omega}\norm{g_t\omega v}\geq Ce^{t}\norm{v}.
		\end{equation}
	\end{lem}
	
	\begin{proof}
		Let $\lambda$ be as in \Cref{lem:Shah_basic_lemma}. Since $\pi(\Omega)$ is not contained in a union of any two proper linear subspaces of $\bbP^{d-1}(\bbR)$, it follows from \Cref{lem:Shah_basic_lemma} that the set $\Omega v$ is not contained in $V_{\geq0}(\lambda)$. Hence 
		\begin{equation}\label{eq:non-trivial plus part}
		C(v):=\sup_{\omega\in\Omega}\frac{\norm{\pi_{<0}(\omega v)}}{\norm{v}} > 0,
		\end{equation}
		where $\pi_{<0}:V\to V_{<0}(\lambda)$ is the $\lambda$-equivariant projection to the direct sum of the eigenspaces where $\lambda(s)$ acts as a negative power of $s\in\bbR^\ast$. Since $g_t=\lambda(e^{-t})$ for all $t\in\bbR$, 
  \[
  V_{<0}(\lambda)=\{v\in V: g_{-t}v=\lambda(e^{t})v\to 0 \text{ as $t\to \infty$}\} 
  \]
		
		For any $w\in V_{<0}(\lambda)$ and any $t\geq0$, putting $s=e^{-t}\leq 1$, 
  
		\begin{equation}\label{eq:smallest positive weight}
		\norm{g_tw}=\norm{\lambda(s)w}\geq s^{-1}\norm{w}=e^t\norm{w}.
		\end{equation}
		
		Now we have
		\begin{equation}
		\begin{split}
		\sup_{\omega\in\Omega}\norm{g_t\omega v}
		&\geq\sup_{\omega\in\Omega}\norm{\pi_{<0}(g_t\omega v)}\\
		&= \sup_{\omega\in\Omega}\norm{g_t\pi_{<0}(\omega v)}\\
		&\stackrel{\eqref{eq:smallest positive weight}}{\geq} e^{t}\sup_{\omega\in\Omega}\norm{\pi_{<0}(\omega v)}\\
		&\stackrel{\eqref{eq:non-trivial plus part}}{=} C(v)e^{t}\norm{v},
		\end{split}
		\end{equation}
		where the first equality follows from the $\lambda$-equivariance of $\pi_{<0}$. This proves \eqref{eq:exponential-expansion}. 
		
		Note that we have a well-defined map $f:\bbP(V)\to\bbR_{\geq 0}$ given by $[v]\mapsto C(v)$. Suppose there is no nonzero $G$-fixed vector in $V$, then the image of $f$ is contained in $\bbR_{>0}$. Since $f$ is continuous, the image of $f$ is compact. Hence \eqref{eq:noGfixed} holds for $C:=\min_{v\neq0}C(v)>0$. 

        One obtains \eqref{eq:Dnormv} by expressing $V$ as a direct sum of a $G$-fixed subspace and a complementary $G$-invariant subspace and then applying \eqref{eq:noGfixed} to the complementary subspace.
	\end{proof}

	%%%%%%%%%%%%%%%%%%%%%%%%%%%%%%%%%%%%%%%%%%%
	%%%%%%%%%%%%%%%%%%%%%%%%%%%%%%%%%%%%%%%%%%%
	\section{Consequences of boundedness in linear representations}\label{sect:consequences of linear focusing}
	%%%%%%%%%%%%%%%%%%%%%%%%%%%%%%%%%%%%%%%%%%%
	%%%%%%%%%%%%%%%%%%%%%%%%%%%%%%%%%%%%%%%%%%%
	Let $\bfG=\SL_n$, and let $\alpha_i$ be the root $\diag(t_j)\mapsto t_it_{i+1}^{-1}$ for $1\leq i\leq n-1$. Then $\Delta=\left\{ -\alpha_i \right\}_{1\leq i\leq n-1}$ form a set of simple roots. Let $\bfP_i$ be the maximal parabolic subgroup of $\bfG$ associated to $\Delta\setminus\{-\alpha_i\}$ for $1\leq i\leq n-1$, and let $P_i=\bfP_i(\bbR)$. In block matrix form,
	\begin{equation} \label{eq:Pi}
	P_i = \left\{
	\begin{pmatrix}
	A_{i\times i} & 0_{i\times(n-i)} \\
	C_{(n-i)\times i} & D_{(n-i)\times(n-i)}
	\end{pmatrix}\in \SL_n(\bbR)
	\right\}.
	\end{equation}
	Let $U_i$ be the unipotent radical of $P_i$. Let $L_i$ be the marked Levi subgroup of $P_i$, which is of the form
	\begin{equation}\label{eq:definition of L_i}
	L_i = \left\{
	\begin{pmatrix}
	A_{i\times i} & 0_{i\times(n-i)} \\
	0_{(n-i)\times i} & D_{(n-i)\times(n-i)}
	\end{pmatrix}\in \SL_n(\bbR)
	\right\}.
	\end{equation}

    Let
    \begin{equation} \label{eq:definition of H_i}
    H_i = 
	\left\{\begin{pmatrix}
	A & \\
	& I_{n-i} 
	\end{pmatrix}\colon A\in\SL_i(\bbR)\right\}
    \end{equation}
	
	We recall some notations from the introduction. Let $g_t = \diag(e^{(n-1)t},e^{-t},\dots,e^{-t})$. Let $\cM$ be a compact connected analytic subset of $P_1\backslash G\cong \bbP^{n-1}(\bbR)$, and let $\cL_\cM$ denote the linear span of $\cM$ in $P_1\backslash G$, and suppose the dimension of $\cL_\cM$ is $d-1$. Since the $(d-1)$-dimensional linear subspaces of $P_1\backslash G$ are parameterized by $P_d\backslash G$, $\cL_\cM$ can be written as $P_1P_d\gM$ for some $\gM\in G$, where the class $[\gM]=P_d\gM$ parametrizes $\cL_\cM$. Then, we can write 
 \begin{equation} \label{eq:Omega}
 \cM=\Omega \gM,
 \end{equation} 
 where $\Omega$ is a compact subset of $P_1P_d=P_1L_d=P_1H_d$. Without loss of generality, we may assume that $\Omega$ is a compact subset of $H_d$. We also know that the image of $\Omega$ in $P_1\backslash G\cong \bbP^{n-1}(\bbR)$ has linear span exactly equal to $P_1\backslash P_1L_d$.
	
	Suppose that there exists a representation $V$ of $\bfG$, a sequence $t_i\to\infty$, a sequence $\{\gamma_i\}$ in $\Gamma$, a nonzero vector $v_0\in V(\bbQ)$, and a constant $C>0$ such that
	\begin{equation} \label{eq:linear_focusing_g_t}\tag{$\spadesuit$}
	\sup_{\omega\in\Omega}\norm{g_{t_i}\omega\gM\gamma_i v_0}\leq C, \; \forall i.
	\end{equation}
 
	The main goal of this section is to derive consequences of \eqref{eq:linear_focusing_g_t}. More precisely, we will find those $\gM$ for which \eqref{eq:linear_focusing_g_t} could possibly hold.
	
	\begin{prop}\label{prop:consequence of linear focusing}
		Suppose there exists a representation $V$ of $\bfG$, a sequence $t_i\to\infty$, a sequence $\{\gamma_i\}$ in $\Gamma$, a nonzero vector $v_0\in V(\bbQ)$ not fixed by $G$, and a constant $C>0$ such that \eqref{eq:linear_focusing_g_t} holds. Then $d<n$, and at least one of the following three statements holds:
		\begin{enumerate}
			\item \label{itm:focusing-1} There exists $t_i'\to\infty$, $C'>0$ and $v_i\in \bbZ^n\setminus\{0\}$ such that
			\[
			\sup_{\omega\in\Omega}\norm{g_{t_i'}\omega\gM v_i}\leq C', \; \forall i.
			\]
			
			\item \label{itm:focusing-2} There exist integers $r\geq d$, $m\geq 2$, and a number field $\bbK\subset\bbR$ such that $[\bbK:\bbQ]=m$, $n=mr$, and
			\[
			\gM\in P_dP_r\bfG(\bbK).
			\]
			
			\item \label{itm:focusing-3} $n\geq4$ is even, $d=2$, and there exists $C'>0$ and $w_0\in\wedge^2\bbZ^n$ such that
			\[
			\sup_{\omega\in\Omega}\norm{g_{t_i}\omega\gM\gamma_i w_0}\leq C', \; \forall i.
			\]
		\end{enumerate}
	\end{prop}

        We remark that $w_0$ is usually not decomposable (see \Cref{rem:decomposable} for adefinition). The rest of the section is devoted to proving \Cref{prop:consequence of linear focusing}.

 \subsubsection*{Decomposition of the translating flow}
	Let
	\begin{equation} \label{eq:btct}
	b_t=\begin{pmatrix}
	e^{\frac{n-d}{d}t}I_d & \\ 
	& e^{-t}I_{n-d}
	\end{pmatrix},\;
	c_t=\begin{pmatrix}
	e^{\frac{nd-n}{d}t} & & \\
	& e^{-\frac{n}{d}t}I_{d-1} & \\
	& & I_{n-d}
	\end{pmatrix}
	\end{equation}

	One has $g_t=c_tb_t$. Since $b_t$ centralizes $L_d$, we can rewrite \eqref{eq:linear_focusing_g_t} as
	\begin{equation} \label{eq:linear_focusing_c_tb_t}\tag{$\clubsuit$}
	\sup_{i\in\bbN} \sup_{\omega\in\Omega}\norm{c_{t_i}\omega b_{t_i}\gM\gamma_i v_0}<\infty.
	\end{equation}
	
	Let 
 \begin{equation}
 \label{eq:Md}
    M_d=\left\{\begin{pmatrix}
	A & \\
	& t^{-1}I_{n-d}
	\end{pmatrix}\colon A\in\GL_d \text{ and } \det A =t^{(n-d)}\right\}\subset G.
 \end{equation}
	The next lemma will be important to analyze the cases when $\bfG v_0$ is Zariski closed in $V$.
	
	\begin{lem}\label{lem:v_infty fixed by H_d}
		Suppose that \eqref{eq:linear_focusing_c_tb_t} holds and that $\bfG v_0$ is Zariski closed in $V$. Then the set $\{ b_{t_i}\gM\gamma_i v_0 \}_{i\in\bbN}$ has an accumulation point, say $v_\infty\in Gv_0\subset V\setminus\{0\}$. Moreover, $v_\infty$ is fixed by $H_d$. In particular, a conjugate of the stabilizer of $v_0$ contains $H_d$, and hence $d<n$.

  Furthermore, if $\{\gamma_i v_0:i\in \bbN\}$ is bounded, then $v_\infty$ is fixed by $M_d$, and hence the stabilizer of $v_0$ contains a conjugate of $M_d$.
	\end{lem}
	
	\begin{proof}
		Applying \eqref{eq:Dnormv} of \Cref{lem:quantitative_Shah_basic_lemma} for $H_d$ and $\Omega$, from \eqref{eq:linear_focusing_c_tb_t} we conclude that
		\begin{equation*} %\label{eq:linear_focusing_b_t}
		\sup_{i\in\bbN} \norm{b_{t_i}\gM\gamma_i v_0}<\infty.
		\end{equation*}
		Hence, after passing to a subsequence, $b_{t_i}\gM\gamma_i v_0\to v_\infty$ for some $v_\infty\in V$. 

  Since $\bfG v_0$ is Zariski closed, we have that $Gv_0$ is closed in the Hausdorff topology (see e.g.~\cite[Section 3.2]{PR94}). Therefore, there exists $g\in G$ such that $v_\infty = gv_0$. We write $\bfF = \bfG_{v_0}$ and $H = G_{v_\infty}$, then we know $\bfF$ is a reductive $\bbQ$-subgroup of $\bfG$ by Matsushima's criterion (see~\cite[7.10]{Bor69}), and $F = g^{-1}H g$.
		
		Furthermore, we claim that $v_\infty$ is fixed by $H_d$. Indeed, when considering the restricted representation to $H_d$, we have a decomposition
		$
		V=V_{0} \oplus V_{1},
		$
		where $V_0$ is the subspace of all $H_d$-fixed vectors and $V_1$ is the $H_d$-invariant complement of $V_0$ in $V$. Let $\pi_0$ and $\pi_1$ be the $H_d$-equivariant projections from $V$ to $V_0$ and $V_1$, respectively. Suppose that $\pi_1(v_\infty)\neq 0$, then there exists $\delta>0$ such that $\norm{\pi_1(b_{t_i}\gM\gamma_i v_0)}>\delta$ for all sufficiently large $i$. By \eqref{eq:noGfixed} of \Cref{lem:quantitative_Shah_basic_lemma}, there exists a constant $C>0$ such that for all $v\in V_1\setminus\{0\}$ and all $t\geq 0$,
		\[
		\sup_{\omega\in\Omega}\norm{c_t\omega v}\geq Ce^{t} \norm{v}.
		\] 
		Hence $\sup_{\omega\in\Omega}\norm{c_{t_i}\omega b_{t_i}\gM\gamma_i v_0}\to\infty$ as $t_i\to\infty$, but this contradicts \eqref{eq:linear_focusing_c_tb_t}. Hence we must have $\pi_2(v_\infty)=0$, and thus $v_\infty$ is fixed by $H_d$.

        We note that if $d=n$, then $G=H_d$ fixes $v_0$, which contradicts the assumption of \Cref{prop:consequence of linear focusing}. Therefore, $d<n$.
        
        Further, suppose that $\{\gamma_i v_0:i\in\bbN\}$ is bounded, then after passing to a subsequence, we may assume that $\gamma_i v_0=\gamma_{i_0}v_0$ for all $i\geq i_0$. It follows that $b_{t_i}\gM\gamma_{i_0} v_0\to v_\infty$ as $i\to\infty$. Hence $v_\infty$ is fixed by $\{b_t\}$. 
	\end{proof}
	
	We will prove \Cref{prop:consequence of linear focusing} by analyzing the following three cases separately in the next three subsections:
	\begin{enumerate}
		\item $\bfG v_0$ is not Zariski closed in $V$.
		
		\item $\bfG v_0$ is Zariski closed, and $\gamma_i v_0 \not\to \infty$. Hence we may assume $\gamma_i v_0 = v_0\in V(\bbQ)$ for all $i$, by passing to a subsequence and replacing $v_0$ with $\gamma_i v_0$.
		
		\item $\bfG v_0$ is Zariski closed, and $\gamma_i v_0 \to \infty$ as $i\to\infty$.
	\end{enumerate}
	
	\subsection{Case (1)}\label{subsect:case 1}
	Under the assumption that $\bfG v_0$ is not Zariski closed, we analyze \eqref{eq:linear_focusing_g_t}. 
	
	The next lemma allows us to pass from fundamental representations of $\SL_n$ to the standard representation. Such a result appears in an unpublished draft by David Simmons in response to a question asked by Kleinbock in \cite[Section 6.1]{Kle08}; the draft was communicated to us by Dmitry Kleinbock. Another proof of the result due to Emmanuel Breuillard was communicated to us by Nicolas de Saxc\'e.  Both the proofs follow a similar line of geometric arguments. We will provide an algebraic proof using Pl\"ucker relations. 

 \subsubsection*{Pl\"ucker relations}
	Let $k\in\{1,\ldots,n-1\}$. For $v\in\bigwedge^k\bbR^{n}$, we can write
	\begin{equation} \label{eq:Plu}
	v=\sum_{1\leq i_1<i_2<\cdots<i_k\leq n}C_{i_1\cdots i_k}e_{i_1}\wedge\cdots\wedge e_{i_k},
	\end{equation}
 where $\{e_1,\ldots,e_n\}$ denotes the standard basis of $\bbR^{n}$, and define $\norm{v}=\sup_{i_1<\ldots<i_k} \abs{C_{i_1\cdots i_k}}$. We have that $[v]$ is in the image of the Pl\"ucker embedding $\Gr(k, n)(\bbR)\hookrightarrow\bbP(\bigwedge^k\bbR^{n})$ if and only if the coordinates $C_{i_1\cdots i_k}$ satisfy the following \emph{Pl\"ucker relations}: For any two ordered sequences 
	\begin{gather}\label{eq:Plucker relations}
	\cI=(i_1,\dots,i_{k-1})\text{, where }  i_1<\dots<i_{k-1} \text{ and } \\
 \cJ=(j_1,\dots,j_{k+1})\text{, where } j_1<\dots<j_{k+1} \text{, we have }\\
	\sum_{l=1}^{k+1}(-1)^lC_{i_1\cdots i_{k-1}j_l}C_{j_1\cdots\widehat{j_l}\cdots j_{k+1}}=0,
	\end{gather}
	where for a permutation $\sigma$ of $(i_1<\dots<i_k)$, $C_{\sigma(i_1)\cdots\sigma(i_k)}:=\mathrm{sgn}(\sigma)C_{i_1\cdots i_k}$. Here, we set $C_{l_1\cdots l_k}=0$ if $l_1,\ldots,l_k$ have repetitions.

    \begin{remark} \label{rem:decomposable} We say that a nonzero $v\in\bigwedge^k\bbR^n$ is \emph{decomposable} if $v=v_1\wedge\cdots\wedge v_k$ for some linearly independent $v_1,\dots,v_k\in\bbR^n$. In this case, let $\Delta_v$ denote the $\bbZ$-span of $\{v_1,\ldots,v_k\}$ and $\covol(\Delta_v)$ denote the co-volume of the lattice $\Delta_v$ in its $k$-dimensional $\bbR$-span. Then  $\covol(\Delta_v)$ equals the Euclidean norm of $v$ with respect to its coordinates as in \eqref{eq:Plu}, see \cite[Lemma~1.4]{Dani:nondiv}. Hence
    \begin{equation} \label{eq:covol-norm}
        \covol(\Delta_v)\leq \norm{v}\leq \sqrt{\binom{n}{k}}\covol(\Delta_v).
    \end{equation}
    \end{remark}
	
	\begin{lem}\label{lem:reducing to standard representation}
		There exists $C=C(n,d,\Omega)>0$ such that the following holds: for any $c>0$, for any $k\in\{1,\ldots,n-1\}$, any nonzero decomposable $w=w_1\wedge \cdots\wedge w_k\in \wedge^k\bbR^n$, where $w_1,\dots,w_k\in\bbR^n$, and any $t > (\log C+\log c-\log\norm{w})/(n-k)$, if
		\begin{equation}\label{eq:bdd}
		\sup_{\omega\in \Omega} \norm{g_t\omega {w}} \leq c,
		\end{equation}
		then there exists a nonzero $v$ in the $\bbZ$-span of $\{w_1,\ldots,w_k\}$ such that 
		\[
		\sup_{\omega\in \Omega} \norm{g_{t'}\omega v} \leq Cc^{1/k} \text{, where } t'=((n-k)t-\log c+\log\norm{{w}})/n.
		\]
	\end{lem}
 
	\begin{proof}
		In this proof, all the implicit constants in $``\ll"$ depend only on $n$, $d$, $\Omega$, and $k$. 
		
		Let $V_1$ denote the span of $\{e_1,\dots,e_d\}$ and $V_2$ denote the span of $\{e_{d+1},\dots,e_n\}$. We have the decomposition
		\[
		\wedge^k\bbR^n=\oplus_i \wedge^iV_1\otimes\wedge^{k-i}V_2.
		\]
		Let $\pi_i$ denote the projection from $\wedge^k\bbR^n$ to $W_i:=\wedge^iV_1\otimes\wedge^{k-i}V_2$. Note that $b_t$ acts on $W_i$ by scalar multiplication by $e^{(ni/d-k)t}$. Notice that $g_t\omega{w}=c_t\omega b_t{w}$. 
		
		For $i\geq1$ such that $W_i$ is non-trivial, we apply \Cref{lem:quantitative_Shah_basic_lemma} for $L_d$ and $c_t$. We note that the only non-negative eigenvalue of $c_t$ on $W_i$ is $e^{(n-ni/d)t}$. Then there exists $C_1>0$ such that from \eqref{eq:bdd} we conclude that $\norm{\pi_i(b_t{w})}\leq C_1 ce^{-(n-ni/d)t}$. Since $\pi_i$ is $b_t$-equivariant, we have $\norm{\pi_i({w})}\leq C_1ce^{-(n-k)t}$. For any choice of $C\geq C_1$, by our assumption, 
  \[
  (n-k)t > \log {C_1}+\log c-\log\norm{{w}}.
  \]
  Therefore $\norm{\pi_i({w})}<\norm{{w}}$ for all $i\geq 1$. Since we consider the sup-norm, 
     \[ 
        \norm{{w}}=\norm{\pi_0({w})}.
    \]
    
		Now $H_d$ acts trivially on $W_0$, so $\norm{\pi_0({w})}\leq ce^{kt}$. Write $\norm{\pi_0({w})}=ce^{\delta t}$, where $\delta\leq k$. We claim that for all $i\geq1$ we have
		\begin{equation}\label{eq:size}
		\norm{\pi_i({w})}\ll ce^{(-i(n-k)-(i-1)\delta)t}.
		\end{equation}
		We verify this claim by induction. The base case $i=1$ is already established. For the inductive step, suppose that \eqref{eq:size} holds for $1\leq i \leq m-1$. We write
		\[
		{w}=\sum_{i_1<i_2<\cdots<i_k}C_{i_1\cdots i_k}e_{i_1}\wedge\cdots\wedge e_{i_k}.
		\]
		Since we are taking the sup-norm, there exist $d+1\leq p_1<p_2<\cdots<p_k\leq n$ and $1\leq q_1 < q_2 <\cdots< q_m\leq d <d+1\leq q_{m+1}<\cdots < q_{k}\leq n$ such that $\norm{\pi_0({w})} = \abs{C_{p_1\cdots p_k}}$ and $\norm{\pi_{m}({w})} = \abs{C_{q_1\cdots q_k}}$. By \eqref{eq:Plucker relations},
		for the two ordered sequences $\cI=(q_1<\cdots<q_{k-1})$ of size $k-1$ and $\cJ=(q_k<p_1<\cdots<p_k)$ of size $k+1$, and the triangle inequality, we have
		\begin{align*}
		\norm{\pi_m({w})}\norm{\pi_0({w})}
		&=\abs{C_{q_1\cdots q_k}}\cdot\abs{C_{p_1\cdots p_k}}\\
		&\leq \sum_{l=1}^k \abs{C_{q_1\cdots q_{k-1}p_l}\cdot C_{q_kp_1\cdots\hat{p_l}\cdots p_k}}\\
		&\leq k\norm{\pi_{m-1}({w})}\norm{\pi_1({w})}
		\end{align*}
		Hence by the induction hypothesis, we have
		\begin{align*}
		\norm{\pi_m{{w}}}
		&\ll\norm{\pi_{m-1}({w})}\norm{\pi_1({w})}\norm{\pi_0({w})}^{-1}\\
		&\ll ce^{(-(m-1)(n-k)-(m-2)\delta)t}ce^{-(n-k)t}(ce^{\delta t})^{-1}\\
		&=ce^{(-m(n-k)-(m-1)\delta)t}.
		\end{align*}
		Therefore \eqref{eq:size} holds for $i=m$, and the claim is verified.
		
		It follows from \eqref{eq:size} that
		\begin{equation} \label{eq:covol}
		\norm{b_{\frac{n-k+\delta}{n}dt}{w}} \ll ce^{(\delta-\frac{kd(n-k+\delta)}{n})t}.
		\end{equation}
  
		Since $w=w_1\wedge\ldots \wedge w_k\neq 0$, the $\bbZ$-span of $\{w_1,\ldots,w_k\}$, say $\Delta_w$, is a lattice in its $k$-dimensional $\bbR$-span. In view of \eqref{eq:covol-norm} and \eqref{eq:covol}, the covolume of the lattice $b_{\frac{n-k+\delta}{n}dt}\Delta_w$ in its $k$-dimensional $\bbR$-span is $\ll ce^{(\delta-\frac{kd(n-k+\delta)}{n})t}$. Therefore by Minkowski's convex body theorem, there exists a nonzero $v\in\Delta_w$ such that  
		\[
		\norm{b_{\frac{n-k+\delta}{n}dt}v} \ll c^{1/k}e^{(\frac{\delta}{k}-\frac{d(n-k+\delta)}{n})t}.
		\] 
		Since $\delta\leq k$, we have
		\[
		\frac{\delta}{k}-\frac{d(n-k+\delta)}{n} \leq -(d-1)\frac{n-k+\delta}{n}.
		\]
		Let $t'=\frac{n-k+\delta}{n}t=((n-k)t+\delta t)/n$, and we recall that $\delta t=\log\norm{\pi_0(v)}-\log c=\log\norm{v}-\log c$. It follows that
		\[
		\norm{b_{dt'}v}\ll c^{1/k}e^{-(d-1)t'}.
		\]
		We write $v=v_1+v_2$, where $v_1$ belongs to the $\bbR$-spans of $\{e_1,\ldots,e_d\}$ and $v_2$ belongs to the $\bbR$-spans of  $\{e_{d+1},\ldots,e_n\}$. By \eqref{eq:btct}, $b_{dt'}v_1=e^{(n-d)t'}v_1$ and $b_{dt'}v_2=e^{dt'}v_2$. Therefore
  \[
  \norm{v_1}\ll c^{1/k} e^{-(n-1)t'} \text{ and } \norm{v_2}\ll c^{1/k}e^{t'}.
  \]
  Since $\Omega$ is a compact subset of $H_d$, 
\[
\sup_{\omega\in\Omega} \norm{\omega v_1}\ll c^{1/k} e^{-(n-1)t'} \text{ and } \Omega v_2=v_2.
\]
Since  $\norm{g_{t'}}\leq e^{(n-1)t'}$ and $g_{t'}v_2=e^{-t'}v_2$, we get
\[
\sup_{\omega\in\Omega} \norm{g_{t'}\omega v}\ll c^{1/k},
\]
and the conclusion of the lemma holds.
	\end{proof}
 
	\begin{cor}\label{cor:going to zero}
		Let $k\in\{1,\dots,n-1\}$. Let $W_k=\bigwedge^k \bbR^k$ be the $k$-th exterior power of the standard representation of $G$ and $w_k=e_1\wedge\cdots\wedge e_k$. Suppose that there exists $t_i\to\infty$ and $\gamma_i\in\Gamma$ such that
		\begin{equation}\label{eq:going to zero in wedge k}
		\sup_{\omega\in \Omega} \norm{g_{t_i}\omega\gM\gamma_i w_k}\to 0,\text{ as }i\to\infty.
		\end{equation}
		Then there exists $t_i'\to\infty$ and $v_i\in\bbZ^n\setminus\{0\}$ such that
		\begin{equation}\label{eq:going to zero in standard}
		\sup_{\omega\in \Omega} \norm{g_{t_i'}\omega\gM v_i}\to 0,\text{ as }i\to\infty.
		\end{equation}
	\end{cor}
	
	\begin{proof}
		Let $C>0$ be as in \Cref{lem:reducing to standard representation}. Let $c_i=\sup_{\omega\in \Omega} \norm{g_{t_i}\omega\gM\gamma_i w_k}$. Note that $\gM\Gamma w_k$ is a discrete subset of $W_k\setminus\{0\}$, so $\norm{\gM\gamma_iw_k}$ has a uniform lower bound for all $i$. Hence, for every sufficiently large $t_i$, the assumption of \Cref{lem:reducing to standard representation} is satisfied. We take $t_i'=((n-k)t_i-\log c_i+\log\norm{\gM\gamma_iw_k})/{n}$. By \Cref{lem:reducing to standard representation}, we know that \eqref{eq:going to zero in standard} holds. Finally, since $t_i\to\infty$, $c_i\to0$, and $\inf_{i}\norm{\gM\gamma_iw_k}>0$,we have $t_i'\to\infty$.
	\end{proof}

    Recall the definitions of $L_d$ and $H_d$ given by \eqref{eq:definition of L_i} and \eqref{eq:definition of H_i}.
 
	\begin{prop} \label{prop:notclosed}  Let $V$ be a representation of $\bfG=\SL_n$, and let $v_0\in V(\bbQ)$ such that $\bfG v_0$ is not Zariski closed. Let $\Omega$ be a non-empty compact subset of $H_d$ whose Zariski closure is irreducible. Let $\gM\in G$. Suppose that there exists $C>0$, $t_i\to\infty$ and $\gamma_i\in\Gamma$ such that
		\begin{equation}\label{eq:bnd-C}
		\sup_{\omega\in \Omega} \norm{g_{t_i}\omega\gM\gamma_i v_0}\leq C, \,\forall i.
		\end{equation}
		Then there exist $R>0$, $v_i\in \bbZ^n\setminus\{0\}$, and $t_i'\to\infty$ such that
		\begin{equation} \label{eq:bnd-R}
		\sup_{\omega\in \Omega} \norm{g_{t_i'}\omega\gM v_i}\leq R,\,\forall i.
		\end{equation}

    Moreover, if $P_1\backslash P_1\Omega$ is not contained in a union of two proper linear subspaces of $P_1\backslash P_1H_d$, then $d<n$.
	\end{prop}
	
	\begin{proof} 
  By \Cref{cor:reduction_to_unstable}, there exists an irreducible representation $W$ of $G$ defined over $\bbQ$, a highest weight vector $w\in W(\bbQ)$, an element $g_0\in G(\bbQ)$ such that, due to \eqref{eq:bnd-C}, we can pick a constant $D>0$ such that
		\[
		\sup_{\omega\in \Omega}\norm{g_{t_i}\omega\gM\gamma_i g_0 w}\leq D,\, \forall i.
		\]
		Combined with \Cref{lem:reduction_to_fundamental_rep} and \Cref{exa:SLn}, this implies that there exists $D'>0$ depending only on  $\Omega$, $w$ and $D$, such that for each $i$, there exists $k\in\{1,\ldots,n-1\}$ such that 
		\begin{equation} \label{eq:bnd-Dprime}
		\sup_{\omega\in \Omega} \norm{g_{t_i}\omega\gM\gamma_i g_0 w_k}\leq D',
		\end{equation}
        where $w_k=e_1\wedge\cdots\wedge e_k\in W_k(\bbZ)=\wedge^k\bbZ^n$. By passing to a subsequence, we may assume that $k$ is a constant for all $i$.

        Since $g_0\in G(\bbQ)$, pick $N\in\bbZ_{>0}$ such that $Ng_0$ is a $n\times n$ matrix with integer entries. 
        
		Let $\bfw_i=N\cdot\gamma_ig_0w_k\in\bigwedge^k\bbZ^n\setminus\{0\}$. Then \eqref{eq:bnd-Dprime} implies that
		\begin{equation}\label{eq:timesN}
		\sup_{\omega\in \Omega} \norm{g_{t_i}\omega\gM\bfw_i}\leq ND'.
		\end{equation}

		we note that $\bfw_i=(N\gamma_i g_0 e_1)\wedge \cdots \wedge (N\gamma_i g_0 e_k)$, where $N\gamma_i g_0 e_j\in \bbZ^n$ for each $j$, and $\log\norm{\bfw_i}\geq 0$. Now \eqref{eq:bnd-R} follows from \eqref{eq:timesN} and \Cref{lem:reducing to standard representation}.

  Furthermore, suppose that the image of $\Omega$ in $P_1\backslash G$ is not contained in a union of any two proper linear subspaces of $P_1\backslash P_1H_d$, we claim that $d<n$. Indeed, suppose $d=n$, then $H_d=G$. Since $G$ has no nonzero fixed vector in $\wedge^k\bbR^n$ for any $1\leq k\leq n-1$, the constant in \eqref{eq:exponential-expansion} of \Cref{lem:quantitative_Shah_basic_lemma} applied to the $G$ action on $\wedge^k\bbR^n$ is uniform in $v$. This contradicts \eqref{eq:bnd-Dprime}.
	\end{proof}

	We apply \Cref{prop:notclosed} to conclude from \eqref{eq:linear_focusing_g_t} that $d<n$ and there exists $t_i'\to\infty$, $C'>0$ and $v_i\in \bbZ^n\setminus\{0\}$ such that 
	\[
	\sup_{\omega\in\Omega}\norm{g_{t_i'}\omega\gM v_i}\leq C', \; \forall i.
	\]
	i.e. (\ref{itm:focusing-1}) of \Cref{prop:consequence of linear focusing} holds.
	
	\subsection{Case (2)}\label{subsect:case 2}
	$\bfG v_0$ is Zariski closed, and $\gamma_i v_0 = v_0$ for all $i$.  So \eqref{eq:linear_focusing_c_tb_t} takes the form
 \begin{equation} \label{eq:linear_focusing_c_tb_t_case2}
     \sup_{i\in\bbN} \sup_{\omega\in\Omega}\norm{c_{t_i}\omega b_{t_i}\gM v_0}<\infty.
 \end{equation}
By \Cref{lem:v_infty fixed by H_d}, $d<n$ and as $i\to\infty$,
 \begin{equation}
     \label{eq:linear_focusing_b_t_case2}
     b_{t_i}\gM v_0\to v_\infty\in Gv_0,
 \end{equation}
and $v_\infty$ is fixed by $M_d$, which denotes the group generated by $\{b_t\}$ and $H_d$.  

Let $H$ denote the stabilizer of $v_\infty$ in $G$, and $F$ the real points of the stabilizer $\bfF$ of $v_0$ in $\bfG$. Pick $g\in G$ such that $v_\infty=gv_0$. Then
\begin{equation}  \label{eq:HgFginv}
M_d\subset H=gFg^{-1}.
\end{equation}

	Let $\lambda(\tau)=\diag(\tau^{-(n-d)}I_d, \tau^dI_{n-d})$ be a multiplicative $1$-parameter subgroup corresponding to $b_t$. Here $\lambda(\tau)=b_t$ for any $\tau>0$ and $t=-d\log\tau$; so $\tau\to 0$ correspond to $t\to\infty$. So, in view of 
 \eqref{eq:linear_focusing_b_t_case2},  $\gM v_0 = \gM g^{-1}v_\infty \in V_{\geq0}(\lambda)$, and hence $\gM g^{-1}\in\Tran_G(v_\infty, V_{\geq0}(\lambda))$. Note that $P(\lambda)=P_d$, see~\eqref{eq:Pi}. 
 
We have $b_{t_i}\gM g^{-1}v_\infty\to v_\infty$ as $i\to\infty$.  Since $Gv_\infty$ is closed, the orbit map $G/H \to Gv_\infty$ is open. Hence $b_{t_i}\gM g^{-1}\in g_iH$ for some $g_i\to e$ in $G$. Since $H$ has only finitely many connected components, by passing to a subsequence and replacing $g$ by an element of $Hg$, without loss of generality, we may assume that $b_{t_i}\gM g^{-1}\in g_iH^0$ and $g_i\to e$ in $G$. We apply \Cref{lem:local_characterization_of_moving_into_weakly_stable} to conclude that 
\[
b_{t_i}\gM g^{-1}\in \Tran_\bfG\left(w,V_{\geq0}(\delta)\right)^\circ\subset P_dH,
\]
 and note that $b_{t_i}\in P_d$. Hence we have
	\begin{equation}\label{eq:gM_in_PdHg_PdgF}
	\gM\in P_dHg=P_d\,gF.
	\end{equation}
As noted in the proof of \Cref{lem:local_characterization_of_moving_into_weakly_stable}, $\bfF$ is a reductive $\bbQ$-subgroup of $\bfG$. By our assumption, $v_0$ is not $\bfG$ fixed, so $\bfF$ is a proper subgroup of $\bfG$. 
	\bigskip
	
	The rest of the subsection is devoted to classifying the pair $(g,F)$.
	%We start with a few observations. 
	
	\begin{lem} \label{lem:in-proper-parabolic}
		The reductive group $\bfF$ is contained in a proper $\Qbar$-parabolic subgroup of $\bfG$.
	\end{lem}
	
	\begin{proof}
		Suppose not, then $\bfF$-action on $\Qbar^n$ is irreducible. Let $\rho:\bfF\to \GL(\bbC^n)$ denote the restriction of the standard representation of $\SL_n$. Since $H_d\subset F$, we have that $\rho(\bfF)$ contains $\SL(W_1)\times 1_{W_2}$, where $W_1$ is the subspace $\bbC^n$ whose last $(n-d)$-coordinates are zero and $W_2$ is the subspace of $\bbC^n$ whose first $d$-coordinates are zero. Since $\bfF\neq\SL_n$ over $\Qbar$, by \Cref{thm:classification_intermediate_subgroups} we have that $n$ is even, $d=2$, and $\bfF\cong\Sp_n$. Hence $\bfF$ does not contain any conjugate of $\{g_t\}_{t\in\bbR}$; indeed, any diagonalizable element of $\Sp_n$ has eigenvalues $\tau_1^{\pm1},\dots,\tau_{n/2}^{\pm1}$, and thus cannot be conjugated to $g_t$ if $t\neq0$. This contradicts the fact that $\{g_t\}_{t\in\bbR}\subset M_d\subset H=gFg^{-1}$.
	\end{proof}
	
 As a consequence of \Cref{lem:in-proper-parabolic}, the reductive $\bbQ$-group $Z_\bfG(\bfF)$, the centralizer of $\bfF$ in $\bfG$,  has a non-trivial $\overline{\bbQ}$-cocharacter. By \cite[Theorem~18.2]{Bor91}, $Z_\bfG(\bfF)$ has a nontrivial maximal torus defined over $\bbQ$, say $\bfT$. Then $Z_\bfG(\bfT)$ is a proper reductive $\bbQ$-subgroup of $\bfG$. By a theorem of Chevalley (see e.g. \cite[Theorem 2.15]{PR94}), there exists a $\bbQ$-representation $W$ of $\bfG$ and a vector $w$ in $W(\bbQ)$ such that $\bfG_w = Z_\bfG(\bfT)$ and $\bfG w$ is Zariski closed in $W$. 
	
	Since $Gv_0$ is closed in $V$, the map $hF\mapsto hv_0$, $\forall h\in G$, from $G/F\to V$, is a proper map. Hence, the condition \eqref{eq:linear_focusing_c_tb_t_case2} is equivalent to the statement that $c_{t_i}\Omega b_{t_i}\gM F/F$ stays in a fixed compact subset of $G/F$ for all $i$. Since $F\subset Z_G(T)$, we deduce that $c_{t_i}\Omega b_{t_i}\gM Z_G(T)/Z_G(T)$ stays in a fixed compact subset of $G/Z_G(T)$ for all $i$. Therefore
	\begin{equation}\label{eq:Z_G(T)_in_place_of_F}
	\sup_{i\in\bbN} \sup_{\omega\in\Omega}\norm{c_{t_i}\omega b_{t_i}\gM w}<\infty.
	\end{equation}
	We can run the same arguments in this subsection with \eqref{eq:Z_G(T)_in_place_of_F} in place of \eqref{eq:linear_focusing_c_tb_t_case2}, and $Z_\bfG(\bfT)$ in place of $\bfF$. We know that $Z_\bfG(\bfT)$ is a Levi subgroup of a proper $\Qbar$-parabolic subgroup of $\bfG$.
	Hence by replacing $\bfF$ with $Z_\bfG(\bfT)$, we may assume that $\bfF$ is a Levi subgroup of a $\Qbar$-parabolic subgroup of $\bfG$. More explicitly, there exist positive integers $n_1, n_2, \dots, n_m$ and an element $l\in \bfG(\Qbar)$ such that $n= n_1+\cdots +n_m$, and 
	\begin{equation}\label{eq:definition of F}
	\bfF = l
	\begin{pmatrix}
	\GL_{n_1} & & \\
	& \ddots & \\
	& & \GL_{n_m}
	\end{pmatrix}
	l^{-1}\cap \bfG.
	\end{equation}
	It follows that 
	\begin{equation} \label{eq:ZGF}
	Z_\bfG(\bfF) = \left\{
	l\begin{pmatrix}
	t_1I_{n_1} & & \\
	& \ddots & \\
	& & t_mI_{n_m}
	\end{pmatrix}l^{-1}:\prod_{i=1}^m t_i^{n_i}=1\right\}
	\subset \bfF.
	\end{equation}
In particular, $Z_\bfG(\bfF)$ equals $Z(\bfF)$, the center of $\bfF$. 

	\begin{lem}\label{lem:Z_G(F) is Q-anisotropic}
		Suppose the $\bbQ$-torus $Z_\bfG(\bfF) =  Z(\bfF)$ is $\bbQ$-isotropic. Then there exists a representation $(\rho', W)$ of $\bfG$ and a rational nonzero vector $w_0\in W$ such that $w_0$ is unstable in $W$, and
		\begin{equation}\label{eq:temp5}
		\sup_{\omega\in\Omega}\norm{g_t\omega\gM w_0}\to 0 \text{ as } t\to\infty.
		\end{equation}
	\end{lem}

	\begin{proof}
		Our assumption is equivalent to saying that $\bfF$ is contained in a proper $\bbQ$-parabolic subgroup of $\bfG$, say $\bfP$. Without loss of generality, we may assume that $\bfP$ is $\bbQ$-maximal. Then there exists an integer $1\leq k \leq n-1$ and a nonzero $\bbQ$-vector $u$ in the $k$-th exterior power $W$ of the standard representation of $\bfG$, such that $\bfP$ acts on the line $[u]$ via a $\bbQ$-character $\chi$ of $\bfP$. 
  
  Let $g$ be as in \eqref{eq:gM_in_PdHg_PdgF}. Then $Hg=gF$ and $\gM\in P_dHg$. By its definition, $\Omega$ is a compact subset of $L_d\subset P_d$, see~\eqref{eq:Pi}. Therefore $\Omega\gM\subset P_dHg$. We note that $H_d\subset M_d\subset G_{v_\infty}=H$. Since $P_d\subset P_1H_d$, we conclude that $P_dH\subset P_1H$. Therefore $\Omega \gM\subset P_1Hg=P_1gF$. Therefore there exists a compact set $\Omega_1\subset P_1$ and $\Omega_2\in F$ such that $\Omega \gM\subset \Omega_1g\Omega_2$. Then $\Omega_1':=\overline{\cup_{t\geq 0} g_t\Omega_1g_t^{-1}}$ is a compact subset of $P_1$, and 
  \begin{equation} \label{eq:Omega12}
       g_t\Omega\gM\subset \Omega_1' g_t g\Omega_2, \, \forall t\geq 0.
  \end{equation}
  
  Since $g_t\in H=gFg^{-1}$, we have $g^{-1}g_t g\in F\subset P$ for all $t\in\bbR$, and thus $g^{-1}g_t g$ acts on $[u]$ via multiplication by $\chi(g^{-1}g_tg)$. Notice that in the $k$-th exterior power $W$ of the standard representation of $\bfG$, the only two weights of $g_t$ are $t\mapsto e^{(n-k)t}$ and $t\mapsto e^{-kt}$. Hence $g_t$ cannot fix $gu$ for $t\neq0$, and it follows that $\chi(g^{-1}g_tg)\neq1$ for all $t\neq0$. If $\chi(g^{-1}g_1g)<1$, we take $w_0=u$; if $\chi(g^{-1}g_1g)>1$, we replace $\bfP$ by its opposite parabolic subgroup containing $\bfF$. In either case, $g^{-1}g_tgw_0\to 0$ as $t\to\infty$. 
  
  Therefore by \eqref{eq:Omega12},  since element of $\Omega_2\subset F\subset P$, $(g^{-1}g_t g)\subset P$, and elements of $P$ act on $w_0$ by scalars, we get
  \begin{equation}
  g_t\Omega \gM w_0\subset \Omega_1' g_tg\Omega_2w_0=\Omega_1'g (g^{-1} g_tg)\Omega_2w_0=\Omega_1'g\Omega_2(g^{-1}g_t g)w_0,\,\forall t\geq 0.
  \end{equation}
  Since $\Omega_1'$ and $\Omega_2$ are compact, \eqref{eq:temp5} holds.
	\end{proof}
	
	\begin{remark}\label{rmk:reduce to case one}
		Suppose $Z_\bfG(\bfF)$ is not $\bbQ$-anisotropic. Then by \eqref{eq:temp5} of \Cref{lem:Z_G(F) is Q-anisotropic}, the orbit $Gw_0$ is not Zariski closed in $W$, and hence by our discussion in subsection~\ref{subsect:case 1}, (\ref{itm:focusing-1}) of \Cref{prop:consequence of linear focusing} would hold. 
		
		Therefore now we assume that $Z_\bfG(\bfF) =  Z(\bfF)$ is a $\bbQ$-anisotropic torus of $\bfG$.
	\end{remark}

	Let $\bbL$ denote the splitting field of the torus $Z(\bfF)$. Its Galois group is denoted by $\cG=\Gal(\bbL/\bbQ)$. Recall that $n_1,n_2,\dots,n_m$ are the sizes of the blocks in $\bfF$ as in \eqref{eq:definition of F}.
	
	\begin{lem} \label{lem:block_sizes_are_equal}
		Suppose that $Z_\bfG(\bfF) =  Z(\bfF)$ is a $\bbQ$-anisotropic torus of $\bfG$. Then there exists a positive integer $r$ such that $n_1=\cdots =n_m=r$.
	\end{lem}
	
	\begin{proof}
		The standard representation $E$ of $\bfG$ has a weight space decomposition with respect to $Z(\bfF)$
		\begin{equation}\label{eq:eigenspace decomposition}
		E=\bigoplus_{i=1}^m E_{\chi_i},
		\end{equation}
		where $\chi_i\in X^*(Z(\bfF))$ are characters of $Z(\bfF)$ defined over $\bbL$, and $E_{\chi_i}$ consists of vectors on which $Z(\bfF)$ acts as $\chi_i$. After relabelling, we have $\dim E_{\chi_i} = n_i$. In view of \eqref{eq:ZGF}, the only relation between the $\chi_i$'s is $\prod_{i=1}^{m}\chi_i^{n_i}=1$. 
  
  The $\cG$-action on $X^*(Z(\bfF))$ is given by
		\[
		({}^\sigma\!\chi)(t) = \sigma\circ\chi\circ\sigma^{-1}(t), \quad \forall \sigma\in\cG.
		\]
		Then, for all $\sigma\in\cG$ and $i\in\{1,\ldots,n\}$, $\sigma E_{\chi_i} = E_{{}^{\sigma}\!\chi_i}$, and hence $\dim{E_{\chi_i}}=\dim E_{{}^{\sigma}\!\chi_i}$. Therefore to show that $n_1=\cdots=n_m$, it is sufficient to verify that $\cG$-action on the set $\{\chi_1,\dots, \chi_m\}$ is transitive. Indeed, suppose $\cX$ is a $\cG$-invariant proper subset, then $\prod_{\chi_i\in\cX}\chi_i^{n_i}=1$, because it is a $\bbQ$-character and $Z(\bfF)$ is $\bbQ$-anisotropic. But this gives us an additional relation between $\chi_1,\dots,\chi_m$, a contradiction.
	\end{proof}
	
	By \Cref{lem:block_sizes_are_equal} there exists $l\in \GL_n(\bbL)$ such that for any $z\in Z(F)$ we have 
	\begin{equation} \label{eq:l-F}
	z=l\begin{pmatrix}
	\chi_1(z)I_r & & \\
	& \ddots & \\
	& & \chi_m(z)I_r
	\end{pmatrix}l^{-1}. 
	\end{equation}
	Hence $\bfF = l\bfHrm l^{-1}$, where
	\begin{equation} \label{eq:Hrm}
	\bfHrm = 
	\begin{pmatrix}
	\GL_{r} & & \\
	& \ddots & \\
	& & \GL_{r}
	\end{pmatrix} \cap \bfG.
	\end{equation}
	It follows that $Z(\bfF)=l\bfTrm l^{-1}$, where
	\[
	\bfTrm = Z(\bfHrm) =
	\begin{pmatrix}
	t_1I_r & & \\
	& \ddots & \\
	& & t_mI_r
	\end{pmatrix} \cap \bfG.
	\]
	
	\subsubsection{Description of $l$. }
	
	From the proof of \Cref{lem:block_sizes_are_equal} we know that $\cG$ acts transitively on $\{\chi_1,\dots, \chi_m \}$. Let $\cH$ be the stabilizer of $\chi_1$, so $[\cG:\cH]=m$. Let $\bbK = \bbL^\cH$ be the fixed field of $\cH$. Then $[\bbK:\bbQ]=[\cG:\cH]=m$. Let $\{\sigma_1,\dots,\sigma_m \}$ be the set of embeddings $\bbK\hookrightarrow\bbC$. We choose and fix a basis $\{x_1, \dots, x_m \}$ of $\bbK$ as a $\bbQ$-vector space, and let $l_0\in\GL_n(\bbL)$ be such that
	\begin{equation}\label{eq:l_0}
	l_0^{-1}=\begin{pmatrix}
	\sigma_1(x_1)I_r & \cdots & \sigma_1(x_m)I_r \\
	\vdots & \cdots & \vdots \\
	\sigma_m(x_1)I_r & \cdots & \sigma_m(x_m)I_r
	\end{pmatrix}\in\GL_n(\bbL).
	\end{equation}
	
	\begin{remark}\label{rmk:defined over Q}
		We claim that $l_0\bfHrm l_0^{-1}$ is a $\bbQ$-subgroup of $\bfG$. Indeed, for any $\sigma\in\cG$, we have $\sigma(l_0^{-1})=w_\sigma\cdot l_0^{-1}$, where the permutation matrix $w_\sigma$ is an element of $N_{\GL_n}(\bfHrm)$, see \eqref{eq:Hrm}. Then $\sigma(l_0\bfHrm l_0^{-1})=l_0w_\sigma^{-1}\bfHrm w_\sigma l_0^{-1}=l_0\bfHrm l_0^{-1}$. Hence $l_0\bfHrm l_0^{-1}$ is defined over $\bbQ$. Moreover, $l_0\bfHrm l_0^{-1}$ is isomorphic to $\Res_{\bbK/\bbQ}^{(1)}\GL_r$ as $\bbQ$-groups.
	\end{remark}
	
	\begin{lem}\label{lem:classification of F}
		There exists $g_\bbQ\in\GL_n(\bbQ)$ such that $\bfF = g_\bbQ l_0\bfHrm l_0^{-1}g_\bbQ^{-1}$. In particular, later, we will assume that $l=g_\bbQ l_0$.
	\end{lem}
	
	\begin{proof}
		Equivalently, we shall prove that $Z(\bfF)=g_\bbQ l_0\bfTrm l_0^{-1}g_\bbQ^{-1}$. We use the Galois cohomology approach, and the readers are referred to \cite[Section 1.3]{PR94} and \cite{SerreGaloisCohomology} for more details.
		
		Let $\bfG'=\GL_n$. Write $\bfNrm'=N_{\bfG'}(\bfTrm)$ and $\bfH_{r,m}'=Z_{\bfG'}(\bfTrm)$. Consider the homogeneous variety $\bfX = \bfG'/\bfNrm'$. Since $\bfF=l\bfHrm l^{-1}$, we have $Z(\bfF)=l\bfTrm l^{-1}$ which is also defined over $\bbQ$. For any $\sigma\in\cG$, we have $\sigma(Z(\bfF))=Z(\bfF)$, and it follows that $l^{-1}\sigma(l)\in\bfNrm'(\bbL)$. Hence $[l]\in \bfX$ is fixed by $\cG$, and therefore we have $[l]\in\bfX(\bbQ)$. Hence it suffices to prove that $[l]$ and $[l_0]$ are on the same $\bfG'(\bbQ)$ orbit in $\bfX(\bbQ)$. 
		
		By \cite[Prop.36, Cor.1]{SerreGaloisCohomology}, we have an exact sequence of $\cG$-sets
		\[
		\bfG'(\bbQ)\to\bfX(\bbQ)\stackrel{\partial}{\to}H^1(\cG,\bfNrm'(\bbL))\to H^1(\cG,\bfG'(\bbL)).
		\]
		Since the first Galois cohomology of $\bfG'=\GL_n$ vanishes by Hilbert's theorem 90, the above exact sequence tells us that the $\bfG'(\bbQ)$-orbits on $\bfX(\bbQ)$ are parameterized by $H^1(\cG,\bfNrm'(\bbL))$. Now $\partial[l]$ and $\partial[l_0]$ are represented by the 1-cocycles $b_l\colon\sigma\to l^{-1}\sigma(l)$ and $b_{l_0}\colon\sigma\to l_0^{-1}\sigma(l_0)$. So, it suffices to show that $b_l=b_{l_0}$.
		
		On the other hand, we consider the Weyl group 
		\[
		\Wrm=N_{\bfG'}(\bfTrm)/Z_{\bfG'}(\bfTrm)=\bfNrm'/\bfHrm'\cong\mathfrak{S}_m,
		\]
        where $\mathfrak{S}_m$ denotes the symmetric group of permutations on $m$ elements.
		We have another exact sequence of $\cG$-sets
		\[
		H^1(\cG, \bfHrm'(\bbL)) \to H^1(\cG, \bfNrm'(\bbL)) \stackrel{\iota}{\to} H^1(\cG, \Wrm).
		\]
		By Shapiro's Lemma and Hilbert's theorem 90, we have
		\[
		H^1(\bbQ,{}_{b_{l_0}}\!{\mathbf{H}_{r,m}'(\Qbar))}=H^1(\bbQ,\Res_{\bbK/\bbQ}(\GL_r)(\Qbar))=H^1(\bbK,\GL_r(\Qbar))=1,
		\]
        where ${}_{b_{l_0}}\!{\mathbf{H}_{r,m}'}$ is the twist of $\bfH_{r,m}'$ by the $1$-cocycle $b_{l_0}$ as defined in \cite[Ch.~1, \S5]{SerreGaloisCohomology} or \cite[Page~23]{PR94}. Therefore as $\bbQ$-groups, 
        \[
        {}_{b_{l_0}}\!{\mathbf{H}_{r,m}'}\cong l_0\bfH_{r,m}' l_0^{-1}\cong\Res_{\bbK/\bbQ}(\GL_r),
        \]
        where the first isomorphism follows from the definition of the twist, and the second isomorphism follows from the arguments as in \Cref{rmk:defined over Q}.
        
		Since the inflation-restriction sequence 
		\[
		1\to H^1(\cG, {}_{b_{l_0}}\!\bfH_{r,m}'(\bbL)) \to H^1(\bbQ,{}_{b_{l_0}}\!\bfH_{r,m}'(\Qbar)) \to H^1(\bbL,{}_{b_{l_0}}\!\bfH_{r,m}'(\Qbar))
		\]
		is exact (see \cite[Page 25]{PR94}), we have $H^1(\cG, {}_{b_{l_0}}\!\bfH_{r,m}'(\bbL))=1$. Hence, by \cite[Cor. 2 of Prop. 39]{SerreGaloisCohomology}, the fibre $\iota^{-1}(\iota(b_{l_0}))$ contains only one element $b_{l_0}$. Hence it suffices to show that $\iota(b_l)=\iota(b_{l_0})$.
		
		It is clear that the $\cG$-action on $\Wrm$ is trivial, and thus $H^1(\cG,\Wrm)$ is identified with the conjugacy classes of the group homomorphisms from $\cG$ to $\Wrm$. Note that $\Wrm$ is naturally identified with the symmetric group of the set $\{\chi_1,\dots, \chi_m \}$, as well as the set $\{\sigma_1, \dots, \sigma_m \}$. Hence $\iota(b_l)$ and $\iota(b_{l_0})$ induce $\cG$-actions on these two sets respectively. Once we show that these induced actions are conjugate to the natural Galois actions, then we are done.
		
		On one hand, by \eqref{eq:l-F}, for any $z\in Z(\bfF)(\bbQ)$ and any $\sigma\in\cG$, we have
		\[
		\begin{split}
		&l\begin{pmatrix}
		\chi_1(z)I_r & & \\
		& \ddots & \\
		& & \chi_m(z)I_r
		\end{pmatrix}l^{-1}
		=z=\sigma(z)\\
		=&\sigma(l)\begin{pmatrix}
		\sigma(\chi_1(z))I_r & & \\
		& \ddots & \\
		& & \sigma(\chi_m(z))I_r
		\end{pmatrix}\sigma(l)^{-1}\\
		=&\sigma(l)\begin{pmatrix}
		({}^\sigma\chi_1)(z)I_r & & \\
		& \ddots & \\
		& & ({}^\sigma\chi_m)(z)I_r
		\end{pmatrix}\sigma(l)^{-1},
		\end{split}
		\]
		and thus $\iota(b_l)(\sigma)\cdot\chi_i={}^\sigma\chi_i$ for $1\leq i\leq m$.
		
		On the other hand, 
		\[
		\sigma(l_0)^{-1}=\begin{pmatrix}
		\sigma\sigma_1(x_1)I_r & \cdots & \sigma\sigma_1(x_m)I_r \\
		\vdots & \cdots & \vdots \\
		\sigma\sigma_m(x_1)I_r & \cdots & \sigma\sigma_m(x_m)I_r
		\end{pmatrix},
		\]
		and thus $\iota(b_{l_0})(\sigma)\cdot\sigma_i=\sigma\sigma_i$ for $1\leq i\leq m$.
		
		Hence, up to conjugation, both actions indeed coincide with the natural Galois actions respectively.
	\end{proof}
	
	%Due to \Cref{lem:classification of F}, we may assume that $l=g_\bbQ l_0$ for some $g_\bbQ\in\GL_n(\bbQ)$.

    We defined $M_d$ in \eqref{eq:Md}. Let
	\[
     M_{-d} = \left\{ \begin{pmatrix}
        t I_{d} & \\
         & A
    \end{pmatrix} \colon A\in\GL_{n-d},\; \det A = t^{-d}
    \right\}.
    \]
    We note that $Z(M_d)=M_{-d}$ and $Z(M_{-d})=M_d$.
	By \eqref{eq:HgFginv}, we have $H=gFg^{-1}\supset M_d$. 
	Taking centralizers on both sides, we get $gZ(F)g^{-1}\subset M_{-d}$. Consider the right action of $G$ on the space $\bbC^n$ of $n$-dimensional row vectors, and let $\{e_1,\dots,e_n\}$ be its standard basis. Then for each $z\in Z(F)$, $gzg^{-1}$ acts on the span of $\{ e_1,\dots,e_d \}$ by a scalar multiple. Hence $(\bbC e_1+\cdots+\bbC e_d)g$ is contained in $E_\chi(\bbC)$ for some $\chi\in\{\chi_1,\dots, \chi_m \}$, where $E_\chi$ is defined as in \eqref{eq:eigenspace decomposition}. Since $\dim E_\chi = r$, we have $d\leq r$. We note that $(\bbC e_1+\cdots+\bbC e_d)g$ can be identified with the $\bbC$-span of the first $d$ rows of $g$ 
	
	We claim that $\chi$ is defined over $\bbR$. Indeed, suppose $\chi$ is not defined over $\bbR$, then $\chi$ and its complex conjugate $\overline{\chi}$ are distinct. It follows that $$E_\chi(\bbC)\cap \bbR^n\subset E_\chi(\bbC)\cap \overline{E_{\chi}}(\bbC) = E_\chi(\bbC)\cap E_{\overline{\chi}}(\bbC) = \{0\},$$ which contradicts the fact that $E_\chi(\bbC)\cap \bbR^n$ contains the $\bbR$-span of the first $d$ rows of $g$.
	
	By relabelling $\chi_1, \dots, \chi_m$ we may assume $\chi=\chi_1$ and $\chi$ is defined over $\bbK$. Hence, $\bbK$ is a subfield of $\bbR$.
	
	Now pick $f_{d+1},\dots,f_r\in\bbR^n$ such that $\{e_1,\dots,e_d,f_{d+1},\dots,f_r\}$ form a basis of $E_\chi(\bbR)g^{-1}$. Take any $g'\in G$ such that $e_ig'=e_i$ and $e_jg'=f_j$ for $1\leq i\leq d$ and $d+1\leq j\leq r$. It follows that $g'\in P_d$ and
	\[
 g'gZ(F)g^{-1}g'^{-1}\supset \begin{pmatrix}
	tI_r& \\ &s I_{n-r}
	\end{pmatrix}\cap G.
 \]
 Taking centralizers in $G$, we get
	$$g'gFg^{-1}g'^{-1}\subset \begin{pmatrix}
	\GL_r & \\
	& \GL_{n-r}
	\end{pmatrix}\cap G\subset P_r.$$ The linear span of the first $r$ rows of $g'g$ is $E_\chi$, and thus is defined over $\bbK\subset\bbR$. In other words, we have $[g'g]\in(\bfP_r\backslash\bfG)(\bbK)=\bfP_r(\bbK)\backslash\bfG(\bbK)$ (see e.g. \cite[Prop. 20.5]{Bor91}). Hence by \eqref{eq:gM_in_PdHg_PdgF},
	\begin{equation}
	\gM\in P_dgF=P_dg'gF=P_dg'gFg^{-1}g'^{-1}g'g\subset P_dP_r\bfG(\bbK).
	\end{equation}
	In other words, (\ref{itm:focusing-2}) of \Cref{prop:consequence of linear focusing} holds.
	
	We now summarize the above discussion in the following lemma.
	
	\begin{lem}\label{lem:linear focusing case 2}
		Under the assumptions of \Cref{prop:consequence of linear focusing}, suppose that $\bfG v_0$ is Zariski closed and $\gamma_iv_0=v_0$ for all $i$. Further suppose that (\ref{itm:focusing-1}) of \Cref{prop:consequence of linear focusing} does not hold; that is,
		\[
		\inf_{\gamma\in\Gamma}\sup_{\omega\in\Omega}\norm{g_t\omega \gM\gamma e_1}\to\infty\text{ as }t\to\infty.
		\]
		Then there exist integers $r\geq d$, $m\geq 2$, and a number field $\bbK\subset\bbR$ such that $[\bbK:\bbQ]=m$, $n=mr$, and
		$
		\gM\in P_dP_r\bfG(\bbK)
		$; that is, (\ref{itm:focusing-2}) of \Cref{prop:consequence of linear focusing} holds.
	\end{lem}
	
	We also prove a lemma in the converse direction. Although it is not used in the proof of \Cref{prop:consequence of linear focusing}, it will be used in the proof of the converse of \Cref{prop:consequence of linear focusing} later. For $1\leq i\leq n-1$, 
 let
	\begin{equation}
	Q_i = 
	\begin{pmatrix}
	I_i &  \\
	\mathrm{Mat}_{(n-i), i} & \SL_{n-i}
	\end{pmatrix}.
	\end{equation}
	
	Let $L_i$, $\gM$, $\LM$ be as at the beginning of the section, and recall that $d=\dim\LM+1$.
	
	\begin{lem}\label{lem:K-subspace get stuck}
		Suppose that there exist integers $r\geq d$, $m\geq 2$, and a number field $\bbK\subset\bbR$ such that $[\bbK:\bbQ]=m$, $n=mr$, and
		$\gM\in P_dP_rg_\bbK$ for some $g_\bbK\in\bfG(\bbK)$. Then, at least one of the following holds:
		\begin{enumerate}
			\item There exist a $\bbQ$-subgroup $\bfF\cong_\bbQ\Res_{\bbK/\bbQ}^{(1)}\GL_r$ of $\bfG$ and an element $g_0\in G$ such that $M_r\subset g_0Fg_0^{-1}\subset L_r$ and $g_\bbK\in Q_rg_0$.
			\item $\LM$ is contained in a proper linear subspace of $\bbP^{n-1}(\bbR)$ which is defined over $\bbQ$.
		\end{enumerate}
	\end{lem}
	
	\begin{proof}
		We identify $P_r\backslash G$ with $\Gr(r,n)$ by mapping $P_rg$ to the span of the first $r$ rows of $g$. As before we choose and fix a basis $\{x_1, \dots, x_m \}$ of $\bbK$ as a $\bbQ$-vector space. Consider the $r\times n$ matrix
		\[
		B=\begin{pmatrix}
		x_1I_r & \cdots & x_mI_r
		\end{pmatrix}.
		\]
		Its columns give a $\bbQ$-basis of $\bbK^n$. Let $[g_\bbK]$ be the top $r\times n$ block of $g_\bbK$. We discuss the following two possibilities.
		
		Suppose first that the columns of $[g_\bbK]$ are $\bbQ$-linearly dependent, then $\cL$ is contained in a proper linear subspace of $\bbP^{n-1}(\bbR)$ defined over $\bbQ$, where $\cL$ denotes the $(r-1)$-dimensional linear subspace parametrized by $P_rg_\bbK$. Recall that $\LM$ is the $(d-1)$-dimensional linear subspace of $\bbP^{n-1}(\bbR)$ parametrized by $P_d\gM$. Since $\gM\in P_dP_rg_\bbK$, we have $\LM\subset\cL$. Hence (2) of the lemma holds.
		
		Now suppose that the columns of $[g_\bbK]$ are linearly independent over $\bbQ$. Then there exists $g_\bbQ\in\GL_n(\bbQ)$ such that $[g_\bbK]=Bg_\bbQ$. Recall that after possibly relabelling the $\chi_i$'s, we have assumed that $\chi_1$ is defined over $\bbK\subset\bbR$, and $\sigma_1$ is the inclusion of $\bbK$ into $\bbR$. Therefore $B$ consists of the first $r$ row vectors of $l_0^{-1}$, where $l_0$ is as in \eqref{eq:l_0}. Hence we have $P_rg_\bbK=P_rl_0^{-1}g_\bbQ$. Now take $\bfF=g_\bbQ^{-1}l_0\bfHrm l_0^{-1}g_\bbQ$, and by \Cref{rmk:defined over Q} we know that $l_0\bfHrm l_0^{-1}$ is defined over $\bbQ$. Hence $\bfF$ is defined over $\bbQ$. We again consider the eigenspaces of $Z(F)$ in the standard representation. Since $E_{\chi_1}$ is defined over $\bbK\subset\bbR$, $W:=\bigoplus_{i=2}^mE_{\chi_i}$ is also defined over $\bbR$. Let $g_0\in G$ be the matrix such that the block of the first $r$ rows is $[g_\bbK]$, and the remaining $(n-r)$ rows form a basis of $W$. Then $g_\bbK\in Q_rg_0$. One can also verify that $Z(L_r)\subset g_0Z(F)g_0^{-1}\subset M_{-r}$. By taking centralizers, we see that (1) of the lemma holds.
	\end{proof}
	
	\begin{cor}\label{cor:K-subspace get stuck}
		Under the assumptions of \Cref{lem:K-subspace get stuck}, suppose (1) of \Cref{lem:K-subspace get stuck} holds. Then, there exists a compact subset $\Sigma$ of $G$ such that $g_t\Omega \gM$ is contained in $\Sigma F$ for all $t\geq 0$, where $\Omega$ is as in \eqref{eq:Omega}.
	\end{cor}
	
	\begin{proof}
		By assumption, we write $\gM=p_dp_rg_\bbK$ for some $p_d\in P_d$ and $p_r\in P_r$. Observe that $P_1P_dP_r=Q_1M_r$. Since $\Omega$ is a compact subset of $P_1P_d$, there exist compact sets $\Omega_1\subset Q_1$ and $\Omega_2\subset M_r$ such that $\Omega p_dp_r\subset \Omega_1\Omega_2$. Hence $\Omega \gM=\Omega p_dp_rg_\bbK\subset \Omega_1\Omega_2g_\bbK$. Since $\Omega_1\subset Q_1\subset P_1$ is compact, $g_t\Omega_1g_t^{-1}$ is contained in a fixed compact set for all $t\geq0$. Hence, it suffices to show that $g_t\Omega_2g_\bbK F/F$ is contained in a fixed compact set of $G/F$ for all $t\geq0$. By our assumption, we can write $g_\bbK=q_rg_0$ for some $q_r\in Q_r$. Hence, there exists a compact set $\Omega_3\subset M_r$ and $u_r\in U_r$ such that $g_t\Omega_2g_\bbK=g_t\Omega_2q_rg_0=g_tu_r\Omega_3g_0$. Since $u_r\in U_r\subset P_1$, we have $g_tu_rg_t^{-1}\to e$ as $t\to\infty$. Hence, it suffices to show that $g_t\Omega_3g_0F/F$ is contained in a fixed compact set of $G/F$ for all $t\geq0$. Since $g_t\in M_r$, we have $g_t\Omega_3\subset M_r\subset g_0Fg_0^{-1}$. Hence, $g_t\Omega_3g_0F/F=g_0F/F$ is a single point in $G/F$, and the proof is finished.
	\end{proof}
	
	\subsection{Case (3)} $\bfG v_0$ is Zariski closed, and $\gamma_i v_0 \to \infty$ as $i\to\infty$.
	
	As before, let $\bfF=\bfG_{v_0}$. Since $\gamma_iv_0\to\infty$, $v_0$ is not fixed by $\bfG$, and $\bfF$ is a proper subgroup of $\bfG$. By \Cref{lem:v_infty fixed by H_d}, $d<n$ and we know that a conjugate of $\bfF(\bbC)$ contains $H_d$, where $\bfF(\bbC)$ denotes the $\bbC$ points of $\bfF$. Hence, this enables us to use \Cref{thm:classification_intermediate_subgroups} to classify $\bfF_\bbC$, which is obtained from $\bfF$ by extension of scalars to $\bbC$.
	
	We first show that if $\bfF_\bbC$ is contained in a proper parabolic subgroup of $\bfG_\bbC$, then one can reduce this case to Case~(1) as in \Cref{subsect:case 1} or Case~(2) as in \Cref{subsect:case 2}.
	
	\begin{lem}\label{lem:reducing to adjoint representation}
		Let $\rho:\bfG\to\GL(V)$ be a representation of $\bfG$, and $v_0\in V$ such that $\bfG v_0$ is Zariski closed. We assume that \eqref{eq:linear_focusing_c_tb_t} holds for $(\rho, V, v_0)$, i.e. 
		\begin{equation}\label{eq:temp6}
		\sup_{\omega\in\Omega}\norm{c_{t_i}\omega b_{t_i}\gM\gamma_i v_0}\leq C, \; \forall i.
		\end{equation}
		Let $\bfF=\bfG_{v_0}$. Suppose that $\bfF_\bbC$ is contained in a proper parabolic subgroup of $\bfG_\bbC$, then there exists a nonzero $w_0\in\fg_\bbQ$ such that \eqref{eq:linear_focusing_c_tb_t} holds for $(\Ad, \fg, w_0)$, i.e.
		\begin{equation}\label{eq:temp7}
		\sup_{\omega\in\Omega}\norm{c_{t_i}\omega b_{t_i}\gM\gamma_i w_0}\leq C, \; \forall i.
		\end{equation}
	\end{lem}
	
	\begin{proof}
		Suppose that $\bfF_\bbC$ is contained in a proper parabolic subgroup of $\bfG_\bbC$. Since $\bfF$ is reductive, it is contained in a Levi subgroup of that parabolic, and thus $Z_\bfG(\bfF)_\bbC$ contains a non-trivial multiplicative one-parameter subgroup. Since $v_0\in V(\bbQ)$, $\bfF$ is defined over $\bbQ$, and hence  $Z_{\bfG}(\bfF)$ is a nontrivial reductive subgroup of $\bfG$ defined over $\bbQ$. It follows that $\Lie(Z_\bfG(\bfF))$ is a non-trivial Lie algebra over $\bbQ$. Hence we may pick a nonzero element $w_0\in \Lie(Z_\bfG(\bfF))(\bbQ)$. Since $\bfF$ stabilizes $w_0$, we have a continuous $G$-equivariant map $G/F\to \fg$ which sends $gF$ to $gw_0$. In particular, it sends compact sets to compact sets. Hence \eqref{eq:temp6} implies \eqref{eq:temp7}.
	\end{proof}
	
	\begin{lem}\label{lem:adjoint representation is O(1)}
		Consider the adjoint representation $\Ad:\bfG\to\GL(\fg)$, and suppose that there exists a nonzero $w_0\in\fg$ such that \eqref{eq:temp7} holds. Then the set $\{\gamma_iw_0 \}$ is bounded.
	\end{lem}
	
	\begin{proof}
        Recall from the beginning of the section that $P_1\Omega$ is not contained in any proper linear subspace of $P_1\backslash P_1L_d\cong (P_1\cap H_d)\backslash H_d$. Since $\Omega$ is analytic and connected, we deduce that $P_1\Omega$ is not contained in a union of finitely many proper linear subspaces of $(P_1\cap H_d)\backslash H_d$.
		Applying \Cref{lem:Shah_basic_lemma} to $H_d$, we know that for any $0\neq v\in\fg$, the set $\Omega v$ cannot be contained in $\fg_{<0}(c_t):=\{X\in \fg:\Ad(c_t)(X)\to 0 \text{ as } t\to\infty\}$. By \eqref{eq:btct}, 
		\begin{equation*}\label{eq:V^-(g_t) in V^-(c_t)}
		\fg_{<0}(g_t)=\oplus_{i=2}^n \fg_{ij}\subset \oplus_{i=2}^d \fg_{ij}=\fg_{<0}(c_t),
		\end{equation*}
		where $\fg_{ij}$ is the 1-dimensional linear subspace spanned by the elementary matrix whose only nonzero entry lies in the $i$-th row and $j$-th column. Therefor $\Omega v\not\subset \fg_{<0}(g_t)$ for any nonzero $v$. Let $\pi_{\geq0}$ denote the $g_t$-equivariant projection from $\fg$ to $\fg_{\geq0}(g_t)$, where $\fg_{\geq0}(g_t):=\{X\in \fg: \Ad(g_t)(X) \text{ is convergent as $t\to 0$}\}$. Then by compactness of $\bbP(\fg)$, the map $\bbR v\mapsto \sup_{\omega\in\Omega}\norm{\pi_{\geq0}(\omega v)}/\norm{v}$ has a positive minimum. Therefore, we can pick $D>0$ such that for any $v\in\fg$ and any $t\geq0$, we have
		\[
		\sup_{\omega\in\Omega}\norm{g_t\omega v}\geq D\norm{v}.
		\]
		Hence
		\[
		\begin{split}
		\sup_{\omega\in\Omega}\norm{c_{t_i}\omega b_{t_i}\gM\gamma_i w_0}=\sup_{\omega\in\Omega}\norm{g_{t_i}\omega\gM\gamma_i w_0}
		\geq D\norm{\gM\gamma_iw_0}.
		\end{split}
		\]
		So, by \eqref{eq:temp7}, $\{\gM\gamma_iw_0 \}$ is bounded. Hence $\{\gamma_iw_0 \}$ is bounded.
	\end{proof}
	
	If $\bfF_\bbC$ is contained in any proper parabolic subgroup of $\bfG_\bbC$, then by \Cref{lem:reducing to adjoint representation} and \Cref{lem:adjoint representation is O(1)}, we are reduced to the Case~(1) or Case~(2). Therefore, we assume that $\bfF_\bbC$ is not contained in any parabolic subgroup of $\bfG_\bbC$.
	
	\begin{lem}\label{lem:conjugate to Sp_n over C}
		Suppose that $\bfF_\bbC$ is not contained in any proper parabolic subgroup of $\bfG_\bbC$. Then $n$ is even, $d=2$, and there exists $l_\bbC\in\GL_n(\bbC)$ such that $\bfF_\bbC=l_\bbC \Sp_n l_\bbC^{-1}$.
	\end{lem}
	
	\begin{proof}
		It follows from our assumption that the inclusion $\bfF_\bbC\subset\bfG_\bbC=\SL_n$ gives a faithful irreducible representation of $\bfF_\bbC$ on $\bbC^n$. Furthermore, \Cref{lem:v_infty fixed by H_d} implies that a conjugate of $\bfF_\bbC$ contains $(\bfH_d)_\bbC$. Hence by \Cref{thm:classification_intermediate_subgroups} we know that $n$ is even, $d=2$, and $\bfF_\bbC=\Sp(\bbC^n,\omega)$ for some sympletic form $\omega$ on $\bbC^n$. Let $l_\bbC$ be the transformation matrix from the standard basis of $\bbC^n$ to a symplectic basis of $(\bbC^n, \omega)$. Then we have $\bfF_\bbC=l_\bbC \Sp_n l_\bbC^{-1}$.
	\end{proof}
	
	\begin{lem}\label{lem:conjugate to Sp_n over Q}
		Suppose that $\bfF_\bbC=l_\bbC \Sp_n l_\bbC^{-1}$ for some $l_\bbC\in\GL_n(\bbC)$. Then there exists $l_\bbQ\in \GL_n(\bbQ)$ such that $\bfF=l_\bbQ \Sp_n l_\bbQ^{-1}$.
	\end{lem}
	
	\begin{proof}
		Note that the normalizer of $\Sp_n$ in $\GL_n$ is $\GSp_n$. Consider the homogeneous variety $X=\GL_n/\GSp_n$ of conjugates of $\Sp_n$ in $\GL_n$. Then $\bfF_\bbC=[l_\bbC]\in X$. Since $\bfF$ is defined over $\bbQ$, we have that $\bfF_\bbC\in X(\bbQ)$. To prove the lemma, we only need to show that $\GL_n(\bbQ)$ acts transitively on $X(\bbQ)$. By \cite[Corollary 1 of Proposition 36]{SerreGaloisCohomology}, it suffices to show that the kernel of $H^1(\bbQ, \GSp_n)\to H^1(\bbQ, \GL_n)$ is trivial. 
		
		We consider following the exact sequence of algebraic groups:
		\[
		1\to\Sp_n\to\GSp_n\to \Gm\to 1.
		\]
		Since $H^1(\bbQ, \Sp_n)=H^1(\bbQ, \Gm)=1$, we also have $H^1(\bbQ, \GSp_n)=1$. Hence, we are done.
	\end{proof}
	
	By combining the discussions on all three cases, we will complete:
 %proof of \Cref{prop:consequence of linear focusing}.
	
	\begin{proof}[Proof of \Cref{prop:consequence of linear focusing}]
		Firstly, suppose that $\bfG v_0$ is not Zariski closed. We apply \Cref{prop:notclosed} to conclude that there exists $t_i'\to\infty$, $C'>0$ and $v_i\in \bbZ^n\setminus\{0\}$ such that
		\[
		\sup_{\omega\in\Omega}\norm{g_{t_i'}\omega\gM v_i}\leq C', \; \forall i;
		\]
		that is, (\ref{itm:focusing-1}) of \Cref{prop:consequence of linear focusing} holds.
		
		Now suppose that $\bfG v_0$ is Zariski closed, and $\gamma_i v_0 = v_0$ for all $i$. By \Cref{lem:linear focusing case 2}, we know that either (\ref{itm:focusing-1}) or (\ref{itm:focusing-2}) of \Cref{prop:consequence of linear focusing} holds.
		
		Finally, suppose that $\bfG v_0$ is Zariski closed, and $\gamma_i v_0 \to\infty$ as $i\to\infty$. We first assume that $\bfF_\bbC$ is contained in a proper parabolic subgroup of $\bfG_\bbC$. By \Cref{lem:reducing to adjoint representation} and \Cref{lem:adjoint representation is O(1)}, we can reduce to Case (2) in \Cref{subsect:case 2} and again by \Cref{lem:linear focusing case 2} we have either (\ref{itm:focusing-1}) or (\ref{itm:focusing-2}) of \Cref{prop:consequence of linear focusing} occurs.
		
		Now we assume that $\bfF_\bbC$ is not contained in any proper parabolic subgroup of $\bfG_\bbC$. 
		By \Cref{lem:conjugate to Sp_n over C} and \Cref{lem:conjugate to Sp_n over Q}, there exists $l_\bbQ\in\GL_n(\bbQ)$ such that $\bfF=l_\bbQ\Sp_nl_\bbQ^{-1}$; we note that $n$ is even. Let $w_0=l_\bbQ(e_1\wedge e_2+e_3\wedge e_4+\cdots+e_{n-1}\wedge e_n)$, so that $\bfG_{w_0}=\bfF$. Multiplying $w_0$ by a positive integer, we may assume that $w_0\in\wedge^2\bbZ^n$. Hence (\ref{itm:focusing-3}) of \Cref{prop:consequence of linear focusing} holds.
	\end{proof}

        %%%%%%%%%%%%%%%%%%%%%%%%%%%%%%%%%%%%%%%%%%%%%%%%%%%%%%%
        %%%%%%%%%%%%%%%%%%%%%%%%%%%%%%%%%%%%%%%%%%%%%%%%%%%%%%%
        \section{Interpretation of Diophantine and arithmetic conditions}\label{subsect:interpretation of Diophantine and arithmetic conditions}
        %%%%%%%%%%%%%%%%%%%%%%%%%%%%%%%%%%%%%%%%%%%%%%%%%%%%%%%
        %%%%%%%%%%%%%%%%%%%%%%%%%%%%%%%%%%%%%%%%%%%%%%%%%%%%%%%

	In this section, we reformulate our Diophantine and arithmetic properties of $A\in M_{d,n-d}(\bbR)$ in a group theoretic manner. 
 
	Let $d$ be an integer such that $2\leq d\leq n-1$. Recall that for any $A\in M_{d,n-d}(\bbR)$, we define the following affine subspace of $\bbR^{n-1}$:
	\begin{equation}
	\cL_A = \{ (\bfx,\widetilde{\bfx}A) \mid \bfx\in\bbR^{d-1} \},
	\end{equation}
	where $\widetilde{\bfx}:=(1,\bfx)\in \bbR^{d}$ for any $\bfx\in \bbR^{d-1}$.
	
	We write 
	\begin{equation}
	u_A:=\begin{pmatrix}
	I_d & A \\
	0 & I_{n-d}
	\end{pmatrix},
	\end{equation}
	and we note that $\cL_A=\cL_0u_A$. 
	
	As in \eqref{eq:Omega}, let $\Omega$ be a compact subset of $P_1L_d$ such that the linear span of $P_1\Omega$ in $P_1\backslash G$ is $P_1L_d$.
	
	Recall that we have defined the following two flows in $\SL_n$:
	\begin{equation}
	b_t=\begin{pmatrix}
	e^{\frac{n-d}{d}t}I_d & \\ 
	& e^{-t}I_{n-d}
	\end{pmatrix},\;
	c_t=\begin{pmatrix}
	e^{\frac{nd-n}{d}t} & & \\
	& e^{-\frac{n}{d}t}I_{d-1} & \\
	& & I_{n-d}
	\end{pmatrix},
	\end{equation}
	so that $g_t=b_tc_t$.
    We identify the expanding 
    horospherical group $U^+$ of $g_1$ with $\bbR^{n-1}$ in view of \eqref{eq:U+}.
	
	We need the following elementary fact, whose proof is left to the reader:
	\begin{lem}\label{lem:Omega preserves the size}
		Let $\Sigma$ be a compact subset of $\bbR^k$ whose linear span is the full $\bbR^k$. Then there exists $C>1$ such that for all $v\in\bbR^k$.
		\[
		C^{-1}\norm{v}\leq\sup_{X\in\Sigma}\abs{X\cdot v}\leq C\norm{v},
		\]
    where $\cdot$ denotes the inner product.
	\end{lem}
	
	The following two lemmas relate the Diophantine properties of $A$ with certain bounds in the standard representation of $\SL_n$. We recall the definitions of $\cW$ and $\cW'$ from \Cref{subsect:statement of the main results}.
	
	\begin{lem}\label{lem:interpretation of Wrmn}
		Given $0<r<d$. The following are equivalent:
		\begin{enumerate}
			\item \label{item:A in Wrmn}
			$A\in\cW_{\frac{n-d+r}{d-r}}(d,n-d)$.
			
			\item \label{item:b_t decay Wrmn}
			There exist $R>0$, $t_i\to\infty$ and $v_i\in\bbZ^n\setminus\{0\}$ such that
			\begin{equation}
			\norm{b_{dt_i}u_Av_i}\leq Re^{-rt_i}.
			\end{equation}
			
			\item \label{item:g_t decay Wrmn}
			There exist $C>0$, $t_i\to\infty$ and $v_i\in\bbZ^n\setminus\{0\}$ such that
			\begin{equation}
			\sup_{\omega\in\Omega}\norm{g_{t_i}\omega u_Av_i}\leq Ce^{(d-r-1)t_i}.
			\end{equation}
		\end{enumerate}
	\end{lem}
	
	\begin{proof}
		Write $v_i=\begin{pmatrix}
		\bfp_i \\ \bfq_i
		\end{pmatrix}$ where $\bfp_i\in \bbZ^d$ and $\bfq_i\in\bbZ^{n-d}\setminus{0}$. Then $u_Av_i=\begin{pmatrix}
		A\bfq_i+\bfp_i \\ \bfq_i
		\end{pmatrix}$. One can verify directly that (\ref{item:A in Wrmn}) and (\ref{item:b_t decay Wrmn}) are equivalent to the following: there exists $t_i'\to\infty$ such that
		\begin{equation}\label{eq:sizes}
		\norm{A\bfq_i+\bfp_i} = O(e^{-(n-d+r)t_i'}) \text{ and } \norm{\bfq_i}=e^{(d-r)t_i'}
		\end{equation}
		with the standard big-$O$ notation.
		One can also verify that \eqref{eq:sizes} implies (\ref{item:g_t decay Wrmn}). We now show that (\ref{item:g_t decay Wrmn}) implies \eqref{eq:sizes}.
  
		Since $P_1L_d=P_1(U^+\cap L_d)$, without loss of generality we may assume that $\Omega$ is a compact subset of $U^+\cap L_d\cong \bbR^{d-1}$ which is not contained in any proper affine subspace of $\bbR^{d-1}$. 
		Let $\psi$ denote the identification of $U^+\cap L_d$ with $\bbR^{d-1}$, and $\widetilde{\psi}:U^+\cap L_d\to\bbR^d$ is given by $\omega\mapsto(1,{\psi(\omega)})$.
		By computation, 
		\[
		\omega u_Av_i=\begin{pmatrix}
		\widetilde{\psi}(\omega)\cdot(A\bfq_i+\bfp_i) \\ (A\bfq_i+\bfp_i)_2 \\ \vdots \\ (A\bfq_i+\bfp_i)_d \\ \bfq_i
		\end{pmatrix}
		\]
		By our assumption on $\Omega$, $\widetilde{\psi}(\Omega)$ is compact and spans $\bbR^d$. Applying \Cref{lem:Omega preserves the size} to $\Sigma=\widetilde{\psi}(\Omega)$ we have
		\[\sup_{\omega\in\Omega}\norm{g_{t_i}\omega u_Av_i}\asymp e^{(n-1)t_i}\norm{A\bfq_i+\bfp_i}+e^{-t_i}\norm{\bfq_i}.\]
		It follows that (\ref{item:g_t decay Wrmn}) is equivalent to \eqref{eq:sizes}.
	\end{proof}
	
	Analogous to the above lemma, we also have the following:
	
	\begin{lem}\label{lem:interpretation of wrmn}
		Given $0<r<d$. The following are equivalent:
		\begin{enumerate}
			\item \label{item:A in wrmn}
			$A\in\cW'_{\frac{n-d+r}{d-r}}(d,n-d)$.
			
			\item \label{item:b_t decay wrmn}
			There exists $R_i\to 0$, $t_i\to\infty$ and $v_i\in\bbZ^n\setminus\{0\}$ such that
			\begin{equation}
			\norm{b_{dt_i}u_Av_i}\leq R_ie^{-rt_i}.
			\end{equation}
			
			\item \label{item:g_t decay wrmn}
			There exists $C_i\to 0$, $t_i\to\infty$ and $v_i\in\bbZ^n\setminus\{0\}$ such that
			\begin{equation}
			\sup_{\omega\in\Omega}\norm{g_{t_i}\omega u_Av_i}\leq C_ie^{(d-r-1)t_i}.
			\end{equation}
		\end{enumerate}
	\end{lem}
	
	\begin{proof}
		Write $v_i=\begin{pmatrix}
		\bfp_i \\ \bfq_i
		\end{pmatrix}$. Arguing in the same way as in \Cref{lem:interpretation of Wrmn}, one can show that (\ref{item:A in wrmn}),(\ref{item:b_t decay wrmn}), and (\ref{item:g_t decay wrmn}) are all equivalent to the following: there exists $t_i'\to\infty$ such that
		\begin{equation}
		\norm{A\bfq_i+\bfp_i} = o(e^{-(n-d+r)t_i'}) \text{ and } \norm{\bfq_i}=e^{(d-r)t_i'}
		\end{equation}
		with the standard small-$o$ notation.
	\end{proof}
	
	For the next lemma, we assume that $d=2$, and prove an elementary result concerning the exterior square of the standard representation of $\SL_n$.
	
	Now that $A\in M_{2,n-2}(\bbR)$, and let $N=\begin{pmatrix}
	n-2 \\ 2
	\end{pmatrix}$. Let $\Aext\in M_{2n-3,N}(\bbR)$ be as defined in \eqref{eq:Aext}. 
	
	\begin{lem}\label{lem:interpretation of sympletic}
		Let $W$ be the exterior square of the standard representation of $\SL_n$. The following are equivalent:
		\begin{enumerate}
			\item $\Aext \in \cW_{\frac{n-2}{2}}(2n-3,N)$.
			\item There exists $C>0$, $t_i\to\infty$ and $w_i\in W(\bbZ)\setminus\{0\}$ such that
			\begin{equation}\label{eq:bounded in exterior two}
			\sup_{\omega\in\Omega}\norm{g_{t_i}\omega u_Aw_i}\leq C.
			\end{equation}
		\end{enumerate}
	\end{lem}
	
	\begin{proof}
		We note that the $w_i$ is not necessarily decomposable, and thus we cannot apply \Cref{lem:reducing to standard representation}.
		
		Let $V$ be the standard representation of $\SL_{n}$; let $V_1$ and $V_2$ be the linear spans of $e_1,e_2$ and $e_3,\dots, e_n$ respectively. Let $W_1=V_1\wedge V_1+V_1\otimes V_2$, $W_2=V_2\wedge V_2$. We have $W=W_1\oplus W_2$, and let $\pi_1,\pi_2$ be the projections from $W$ to $W_1,W_2$ respectively. Let $\psi$ be the natural identification of $(U^+\cap L_2)$ with $\bbR$ and let $\widetilde{\psi}:U^+\cap L_2\to\bbR^2$ be given by $\omega\mapsto(1,\psi(\omega))$. Let $3\leq i\leq n$. Then $V_1\otimes\bbR e_i\cong\bbR^2$ is $(U^+\cap L_2)$-stable. And we have 
		\[
		\omega(a\cdot e_1\wedge e_i+b\cdot e_2\wedge e_i)=(\widetilde{\psi}(\omega)\cdot (a,b))e_1\wedge e_i + b\cdot e_2\wedge e_i, \, \forall (a,b)\in\bbR^2.
		\]
		Since $\psi$ is nonconstant, $\widetilde{\psi}(\Omega)$ spans $\bbR^2$. Applying \Cref{lem:Omega preserves the size} to $\Sigma=\widetilde{\psi}(\Omega)$ and $k=2$, one can show that 
		\[
		\sup_{\omega\in\Omega}\norm{\omega w}\asymp \norm{w},\,\forall\, w\in V_1\otimes \bbR e_i. 
		\]
		And, since $U^+\cap L_2$ acts trivially on $V_1\wedge V_1$ and on $W_2$, for any $w\in W$ we have
		\[
		\sup_{\omega\in\Omega}\norm{\pi_j(\omega w)}=\sup_{\omega\in\Omega}\norm{\omega \pi_j(w)}\asymp \norm{\pi_j(w)}, \text{ where }j=1,2.
		\]
		Hence 
		\begin{align*}
		\sup_{\omega\in\Omega}\norm{g_{t_i}\omega u_Aw_i} 
		&\asymp e^{(n-2)t}\sup_{\omega\in\Omega}\norm{\pi_1(\omega u_Aw_i)}+e^{-2t}\sup_{\omega\in\Omega}\norm{\pi_2(\omega u_Aw_i)}\\
		&\asymp e^{(n-2)t}\norm{\pi_1(u_Aw_i)}+e^{-2t}\norm{\pi_2(u_Aw_i)}.
		\end{align*}
		Hence \eqref{eq:bounded in exterior two} is equivalent to 
		\begin{equation}\label{eq:sizes exterior two}
		\begin{cases}
		\norm{\pi_1(u_Aw_i)}=O(e^{-(n-2)t_i})\\
		\norm{\pi_2(u_Aw_i)}=O(e^{2t_i})
		\end{cases}
		\end{equation}
		For simplicity, we fix $i$ and write $w, t$ for $w_i, t_i$.
		By taking a smaller $t$ if necessary, we rewrite \eqref{eq:sizes exterior two} as
		\begin{equation}\label{eq:sizes exterior twoo}
		\begin{cases}
		\norm{\pi_1(u_Aw)}=O(e^{-(n-2)t})\\
		\norm{\pi_2(u_Aw)}\asymp e^{2t}
		\end{cases}
		\end{equation}
		Write $w=\begin{pmatrix}
		\bfp \\ \bfq
		\end{pmatrix}$, where $\bfp\in \bbZ^{2n-3}$ and $\bfq\in\bbZ^{N}$. We now check that \eqref{eq:sizes exterior twoo} is equivalent to 
		\begin{equation}\label{eq:temp8}
		\norm{\Aext\bfq+\bfp}=O(\norm{\bfq}^{-\frac{n-2}{2}}).
		\end{equation}
		For convenience we rewrite
		\[
		A=\begin{pmatrix}
		a_3 & a_4 & \cdots & a_n \\
		b_3 & b_4 & \cdots & b_n
		\end{pmatrix}.
		\]
		For all $3\leq i\leq n$ and $i<j\leq n$, we have
		\begin{align*}
		\gA(e_1\wedge e_2)=&e_1\wedge e_2.\\
		\gA(e_1\wedge e_i)=&e_1\wedge e_i+b_ie_1\wedge e_2.\\
		\gA(e_2\wedge e_i)=&e_2\wedge e_i-a_ie_1\wedge e_2.\\
		\gA(e_i\wedge e_j)=&e_i\wedge e_j+(a_ib_j-a_jb_i)e_1\wedge e_2+a_ie_1\wedge e_j+b_ie_2\wedge e_j-a_je_1\wedge e_i-b_je_2\wedge e_i.
		\end{align*}
		We write $w=\sum_{1\leq k<l\leq n}C_{kl}e_k\wedge e_l$, and compute
		\begin{align*}
		\gA w=\sum_{3\leq i<j\leq n}C_{ij}e_i\wedge e_j+
        \sum_i(C_{1i}+\sum_{j<i}C_{ji}a_j-\sum_{j>i}C_{ij}a_j)e_1\wedge e_i\\
        +\sum_i(C_{2i}+\sum_{j<i}C_{ji}b_j-\sum_{j>i}C_{ij}b_j)e_2\wedge e_i\\
        +\big(\sum_{i<j}(a_ib_j-a_jb_i)C_{ij}+\sum_{i}b_iC_{1i}-\sum_{j}a_jC_{2j}+C_{12}\big)e_1\wedge e_2.
		\end{align*}
		Now \eqref{eq:sizes exterior twoo} is equivalent to
		\begin{equation}\label{eq:temp9}
		\begin{cases}
		C_{1i}+\sum_{j<i}C_{ji}a_j-\sum_{j>i}C_{ij}a_j=O(e^{-(n-2)t})\\
		C_{2i}+\sum_{j<i}C_{ji}b_j-\sum_{j>i}C_{ij}b_j=O(e^{-(n-2)t})\\
		\sum_{i<j}(a_ib_j-a_jb_i)C_{ij}+\sum_{i}b_iC_{1i}-\sum_{j}a_jC_{2j}+C_{12}=O(e^{-(n-2)t})\\
		C_{ij}=O(e^{2t})
		\end{cases}
		\end{equation}
		Therefore from the first two equations, we obtain
		\[
		\sum_{i}b_iC_{1i}-\sum_{j}a_jC_{2j}= -2\sum_{i<j}(a_ib_j-a_jb_i)C_{ij}+O(e^{-(n-2)t}).
		\]
		Inserting this data in the third equation, we conclude that \eqref{eq:temp9} is equivalent to
		\begin{equation}\label{eq:temp10}
		\begin{cases}
		C_{1i}+\sum_{j<i}C_{ji}a_j-\sum_{j>i}C_{ij}a_j=O(e^{-(n-2)t})\\
		C_{2i}+\sum_{j<i}C_{ji}b_j-\sum_{j>i}C_{ij}b_j=O(e^{-(n-2)t})\\
		-\sum_{i<j}(a_ib_j-a_jb_i)C_{ij}+C_{12}=O(e^{-(n-2)t})\\
		C_{ij}=O(e^{2t})
		\end{cases}
		\end{equation}
		By our definition of $\Aext$, \eqref{eq:temp10} is equivalent to \eqref{eq:temp8};
		and by our definiton of $\cW$, \eqref{eq:temp8} is equivalent to $\Aext \in \cW_{\frac{n-2}{2}}(2n-3,N)$. 
	\end{proof}
	
	\begin{lem}\label{lem:interpretation of Res GL_r}
		Let $r\geq d$ be an integer. Let $\bbK\subset\bbR$ be a real field. The following are equivalent:
		\begin{enumerate}
			\item $\LA$ is contained in some $(r-1)$-dimensional affine subspace of $\bbR^{n-1}$ which is defined over $\bbK$.
			\item $\LA$ is contained in some $(r-1)$-dimensional linear subspace of $\bbP^{n-1}(\bbR)$ which is defined over $\bbK$.
			\item $u_A\in P_dP_r\bfG(\bbK)$.
		\end{enumerate}
	\end{lem}
	\begin{proof}
		$(1)\Longleftrightarrow(2)$. This is obvious.
		
		$(2)\Longleftrightarrow(3)$. Let $[x_1:\dots:x_n]$ be the homogeneous coordinates of $\bbP^{n-1}$. Let $\cP_0$ be the $(r-1)$-dimensional linear subspace of $\bbP^{n-1}$ defined by $x_{r+1}=\cdots=x_n=0$. Now (2) is equivalent to $\cL_A=\cL_0u_A\subset\cP_0g_\bbK$ for some $g_\bbK\in\bfG(\bbK)$. Hence it suffices to show that for any $g\in G$, $g\in P_dP_r$ if and only if $\cL_0g\subset\cP_0$. This can be checked directly.
	\end{proof}

	%%%%%%%%%%%%%%%%%%%%%%%%%%%%%%%%%%%%%%%%%%%
	%%%%%%%%%%%%%%%%%%%%%%%%%%%%%%%%%%%%%%%%%%%
	\section{Sharp conditions for non-escape of mass and equidistribution}
	%%%%%%%%%%%%%%%%%%%%%%%%%%%%%%%%%%%%%%%%%%%
	%%%%%%%%%%%%%%%%%%%%%%%%%%%%%%%%%%%%%%%%%%%
	
	In this section, we will obtain sufficient representation theoretic conditions for the non-escape of mass and equidistribution. These representation theoretic conditions were interpreted in terms of Diophantine or arithmetic conditions in \Cref{subsect:interpretation of Diophantine and arithmetic conditions}. We will then prove that these conditions are also necessary for the non-escape of mass or equidistribution. 
	
	\subsection{Non-escape of mass and proof of {Theorem~\ref{thm:main_nondivergence}}}
	Let $\phi$ and $\lambda_{\phi}$ be as in \Cref{subsect:statement of the main results}. We can write $\phi(B)=\Omega u_A$, where $A=A_\phi\in M_{d,n-d}(\bbR)$, and $\Omega$ is a compact subset of $P_1L_d$ such that the linear span of $P_1\Omega$ in $P_1\backslash G$ is $P_1L_d$. Hence $\Omega$ satisfies the condition in the beginning of \Cref{sect:consequences of linear focusing}.
	
	The following theorem is due to Dani-Margulis\cite{DM93} and Kleinbock-Margulis\cite{KM98}. It provides a representation-theoretic criterion for no escape of mass to infinity. What we state here is a consequence of their original theorem.
	
	Let $W_k=\wedge^k\bbR^n$ and $w_k=e_1\wedge\cdots\wedge e_k$.
	
	\begin{thm}[c.f.~{\cite[Theorem 5.2]{KM98}}]\label{thm:KMnondivergence}
		Fix a norm $\lVert\cdot\rVert$ on $W_k$. For any $\epsilon > 0$ and $R > 0$, there exists a compact set $K \subset G/\Gamma$ such that for any $t > 0$ and any ball $J\subset \interval$, one of the following holds:
		\begin{enumerate}
			\item There exist $k \in \{ 1,\dots, n-1 \}$ and $\gamma\in\Gamma$ such that
			\[
			\sup_{s\in J}\lVert g_t\phi(s)\gamma w_k\rVert < R;
			\]
			\item
			\[
			\lvert \{ s\in J \colon g_t\phi(s)x_0 \in K \} \rvert \geq (1-\epsilon)\lvert J \rvert.
			\]
		\end{enumerate}
	\end{thm}
	
	A key ingredient in the proof is a growth property called the $(C, \alpha)$-good property of the family of functions $\{B\ni s\mapsto \norm{g_t\phi(s)w}:t\geq 1,\,w\in W_k\}$, see \cite[Proposition 3.4]{KM98}, \cite[Section 3.2]{Sha09Invention} and \cite[Section 2.1]{Sha09Duke1}. 
	
	\begin{proof}[Proof of \Cref{thm:main_nondivergence}]
		Suppose that the sequence of $g_t$-translates of $\lambda_{\phi}$ has the escape of mass. Then there exists $\eps>0$ and $t_i\to\infty$ such for any compact set $K\subset X$, we have $g_{t_i}\lambda_{\phi}(K)\leq 1-\eps$. By \Cref{thm:KMnondivergence}, after passing to a subsequence, there exists $k\in\{1,\dots,n-1\}$ and $\gamma_i\in\Gamma$ such that $\sup_{s\in \interval}\lVert g_{t_i}\phi(s)\gamma_i w_k\rVert < 1/i$. By \Cref{prop:notclosed}, $d<n$. Now by \Cref{cor:going to zero}, there exists $t_i'\to\infty$ and $v_i\in\bbZ^n\setminus\{0\}$ such that
		$\sup_{s\in \interval} \lVert{g_{t_i'}\phi(s)v_i}\rVert\to 0$ as $i\to\infty$. We apply \Cref{lem:interpretation of wrmn} (taking $r=d-1$) to conclude that $A_\phi\in \cW'_{n-1}(d,n-d)$.
		
		Conversely, suppose that $d<n$ and $A_\phi\in \cW'_{n-1}(d,n-d)$. Again by \Cref{lem:interpretation of wrmn}, there exists $t_i\to\infty$ and $v_i\in\bbZ^n\setminus\{0\}$ such that
		$\sup_{s\in \interval} \lVert{g_{t_i}\phi(s)v_i}\rVert\to 0$ as $i\to\infty$. By Mahler's compactness criterion, for any compact subset $K$, we have $g_{t_i}\lambda_\phi(K)=0$ for all large $i$. Hence the sequence of $g_t$-translates of $\lambda_{\phi}$ has escape of mass.
	\end{proof}
	
	\subsection{Consequence of the failure of equidistribution}
	We suppose that there is no escape of mass for $\{g_t\lambda_\phi:t>0\}$, and thus any limit measure is a probability measure on $X$. Now suppose that as $t\to\infty$, $g_t\lambda_\phi$ does not get equidistributed on $X=G/\Gamma$ with respect to the $G$-invariant probability measure $\mu_X$. Let $x_0=\bbZ^n\in G/\Gamma$. Then there exists a sequence $t_i\to\infty$, a function $f\in C_c(G/\Gamma)$ and an $\epsilon>0$ such that 
	\begin{equation} \label{eq:failure_of_equidistribution}
	\abs{\int_{B} f(g_{t_i}\phi(s)x_0)\,d\lambda(s)-\int f\,d\mu_X}>\epsilon. 
	\end{equation}
	Since there is no escape of mass, by further passing to a subsequence, we assume that there exists a probability measure $\mu_\phi$ on $X$ such that $g_{t_i}\lambda_\phi\to \mu_\phi$ as $i\to \infty$. 
	
 	We will analytically modify $\phi$ on the left by elements of the centralizer of the flow $\{g_t\}$, to get a new map $\psi:B\to G$ and its corresponding measure $\lambda_\psi$, concentrated on $\psi(B)x_0$, which is also the twist of $\lambda_\phi$ by the same elements in the centralizer. Then after passing to a subsequence, $g_{t_i}\lambda_\psi\to \mu_\psi$ for some probability measure $\mu_\psi$ on $X$. We intend to modify $\phi$ in such a way that $\mu_\psi$ is invariant under a nontrivial unipotent subgroup of $U$. Hence by Ratner's measure classification theorem \cite{Ratner91}, each ergodic component of the limit measure is homogeneous, and hence, if $\mu_\psi$ is not $G$-invariant, then it will be strictly positive on the image in $G/\Gamma$ of a $U$-invariant proper algebraic subvariety of $G$. Then, we can apply the linearization technique developed by Dani and Margulis~\cite{DM93} to obtain certain algebraic conditions as a consequence of \eqref{eq:failure_of_equidistribution}. 
	
	\subsubsection{Twisting $\phi$ by centralizer of $\{g_t\}$ and invariance under a unipotent subgroup}
	
	In view of \eqref{eq:U+} in \Cref{subsect:statement of the main results}, given any row vector $v\in\bbR^{n-1}$, let 
	\[
	u(v)=\begin{pmatrix} 1 & v\\&I_{n-1}\end{pmatrix}\in U^+.
	\]
	We may assume that $B$ is an open ball in $\bbR^{k}$ for some $k\in\bbN$, and let $\tilde\phi:B\to \bbR^{n-1}$ be the analytic map such that $\phi(s)=u(\tilde\phi(s))$ for all $s\in B$. Consider the derivative map $D\tilde\phi:B\to \Hom(\bbR^{k},\bbR^{n-1})$. Since $\phi$ is nonconstant and analytic, there exists $1\leq r\leq n-1$ such that the rank of $D\tilde\phi(s)$ is $r$ for all $s$ in $B$ outside a closed subset of zero Lebesgue measure. 
 Let $\cU_r$ denote subspace spanned by the first $r$ standard basis vectors of $\bbR^{n-1}$.
    By expressing the absolutely continuous measure $\lambda$ on $\bbR^k$ as a countable sum of absolutely continuous measures with small supports, without loss of generality, we may assume that the ball $B$ is small enough so that the rank of $D\tilde\phi(s)$ is $r$ for all $s\in B$, and there exists an analytic function $\tilde\zeta:B\to\SL({n-1},\bbR)$ such that
	\begin{equation} \label{eq:zeta}
	D\tilde\phi(s)(\bbR^k)=\cU_r\cdot\tilde\zeta(s),\ \forall s\in B,
	\end{equation} 
	where $\SL(n-1,\bbR)$ acts on $\bbR^{n-1}$ on the right. For all $s\in B$, define
	\[
	\zeta(s)=\begin{pmatrix} 1 \\ & \tilde\zeta(s) \end{pmatrix}\in Z_G(\{g_t\}).
	\]
	Then
	\begin{equation} \label{eq:uzeta}
	u(v)\zeta(s)=\zeta(s)u(v\tilde\zeta(s)),\ \forall v\in \bbR^{n-1},\, \forall s\in B,
	\end{equation}
	Define $\psi:B\to G$ by
	\begin{equation} \label{eq:psi-uzeta}
	\psi(s)=\zeta(s)\phi(s)=\zeta(s)u(\tilde\phi(s)),\ \forall s\in B.
	\end{equation}
	
	\begin{prop} \label{prop:invariant}
		Let $v\in \cU_r\subset \bbR^{n-1}$, $f\in C_c(G/\Gamma)$ and $x\in G/\Gamma$. Then for all $t>0$,
		\[
		\int_{B} f(u(v)g_t\psi(s)x)\,d\lambda(s) = \int_B f(g_t\psi(s)x)\,d\lambda(s)+o_t(1),
		\]
		where $o_t(1)\to 0$ as $t\to\infty$.
	\end{prop}
	
	\begin{proof} Let $t>0$ and $s\in B$. Then
		\begin{align*}
		u(v)g_t\psi(s)&=g_tu(e^{-nt}v)\psi(s)\\
  &=g_tu(e^{-nt}v)\zeta(s)u(\tilde\phi(s))\text{, by \eqref{eq:psi-uzeta}}\\
		&=g_t\zeta(s)u(e^{-nt}v\tilde\zeta(s)+\tilde\phi(s)) 
  \text{, by \eqref{eq:uzeta}}\\
		&=g_t\zeta(s)u(e^{-nt}D\tilde\phi(s)(\tilde v)+\tilde\phi(s)) \text{, for some $\tilde v\in\bbR^k$ by \eqref{eq:zeta}}\\
		&=g_t\zeta(s)u(\tilde\phi(s+e^{-nt}\tilde v)+o(e^{-nt})) \text{, due to Taylor expansion}\\
		&=g_t\zeta(s)u(o(e^{-nt}))u(\tilde\phi(s+e^{-nt}\tilde v))\\
		&=u(o_t(1))g_t\zeta(s)u(\tilde\phi(s+e^{-nt}\tilde v))\text{, as $g_tu(w)=u(e^{nt}w)g_t$}\\
		&=u(o_t(1))g_t\zeta(s)\zeta(s+e^{-nt}\tilde v)^{-1}\psi(s+e^{-nt}\tilde v) 
  \text{, by \eqref{eq:psi-uzeta}}\\
		&=(I+o_t(1))g_t\psi(s+e^{-nt}\tilde v)
 \text{, as $\zeta(B)\subset Z_G(g_t)$.}
 \end{align*}
		Since $f$ is uniformly continuous on $G/\Gamma$, we have
		\[
		f(u(v)g_t\psi(s)x)=f(g_t\psi(s+e^{-nt}\tilde v)x)+o_t(1).
		\]
		Since $\lambda$ is absolutely continuous with respect to the Lebesgue measure on $\bbR^k$ with support contained in $B$, there exists $h\in L^1(\bbR^k)$ such that $d\lambda(s)=h(s)ds$ for Lebesgue a.e.\ $s\in\bbR^k$. For any $w\in\bbR^k$, let $h_w(s)=h(s-w)$ for all $s\in \bbR^{k-1}$. Then $\norm{h_w-h}_1\to 0$ as $w\to 0$ in $\bbR^k$. Therefore
		\begin{align*}
		&\int_{B} f(u(v)g_t\psi(s)x)\,d\lambda(s)\\
		&= \int_{B} f(g_t\psi(s+e^{-nt}\tilde v)x)\,d\lambda(s)+o_t(1)\\
		&=\int_{\bbR^k} f(g_t\psi(s+e^{-nt}\tilde v)x)h(s)\,ds+o_t(1)\\
		&=\int_{\bbR^k} f(g_t\psi(s)x)h(s-e^{-nt}\tilde v)\,ds+o_t(1)\\
		&=\int_{\bbR^k} f(g_t\psi(s)x)h(s)\,ds+O(\norm{f}_{\infty}\norm{h_{e^{-nt}\tilde v}-h}_1)+o_t(1)\\
		&=\int_B f(g_t\psi(s)x)\,d\lambda(s)+o_t(1).
		\end{align*}
	\end{proof}
	
	\subsubsection*{Ratner's theorem and linearization technique for $(C,\alpha)$-good maps}
	Now let $\lambda_\psi$ be the probability measure, which is the pushforward of $\lambda$ under the map $s\mapsto \psi(s)x_0$ from $B$ to $X=G/\Gamma$. 
	We suppose that we are given a sequence $t_i\to\infty$ such that $g_{t_i}\lambda_{\psi}$ converges to a probability measure, say $\mu$, on $G/\Gamma$ as $i\to\infty$ with respect to the weak-$^\ast$ topology; that is, for any $f\in C_c(G/\Gamma)$, we have 
	\[
	\lim_{i\to\infty} \int_{B} f(g_{t_i}\psi(s)x_0)\,d\lambda(s)\to \int_X f\,d\mu.
	\]
	Then by \Cref{prop:invariant} we conclude that $\mu$ in invariant under $U_r=\{u(v):v\in\cU_r\subset \bbR^n\}$. Therefore, by Ratner's classification of ergodic invariant measures \cite{Ratner91}, almost every $U_r$-ergodic component of $\mu$ is a periodic (homogeneous) measure on $G/\Gamma$. 
	
	Therefore, we can argue as in \cite[Propositions~6.3-6.4]{KSSY21} using the linearization technique from Dani-Margulis~\cite[Proposition~4.2]{DM93}, which uses polynomial growth properties of one-dimensional unipotent orbits in a vector space. Since we are working with translates of analytic manifolds, here we use similar $(C,\alpha)$-good growth properties introduced by Kleinbock and Margulis~\cite{KM98} for the family of analytic functions 
	\[
	\{s\mapsto \xi(g_t\psi(s)v):t\in\bbR,v\in V,\,\xi\in V^\ast\}
	\]
	on $B$, where $V$ is any finite dimensional representation of $G$. To take care of higher dimensional manifolds (when $k>1$), we argue as in \cite[Theorem~5.2, Proposition~5.4]{Sha95Duke} and obtain the following linear dynamical boundedness condition:
	
	Let $\mathfrak{g}$ denote the Lie algebra of $G=\SL_n(\bbR)$ with its natural $\bbQ$-structure. For each $1\leq d<\dim \mathfrak{g}$, let $V_d=\bigwedge^d\mathfrak{g}$. 
	
	\begin{prop} \label{prop:consequence of failure of equidistribution} Suppose that the $U_r$-invariant limiting measure $\mu$ is not $G$-invariant. Then after passing to a sequence of $\{t_i\}$, there exists $1\leq d<\dim(G)$, a nonzero vector $v_0\in V_d(\bbQ)$ which is not fixed by $G$, a sequence $\{\gamma_i\}\subset \Gamma=\SL_n(\bbZ)$, and a constant $C_1>0$ such that the following holds: 
		\begin{equation}
		\sup_{s\in B}\norm{g_{t_i}\psi(s)\gamma_iv_0}\leq C_1, \, \forall i. 
		\end{equation}
	\end{prop}
	
	Since $\psi(s)=\zeta(s)\phi(s)$ and $\zeta(s)$ commutes with $\{g_t\}$, by shrinking $B$ a little, if needed, we can assume that $\{\zeta(s):s\in B\}$ is contained in a compact set, so there exists a constant $C>0$ such that
	\begin{equation}\tag{$\spadesuit$}
	\sup_{s\in\interval}\norm{g_{t_i}\phi(s)\gamma_i v_0}\leq C, \; \forall i.
	\end{equation}

	\subsection{Getting failure of equidistribution}
	In this subsection, we show that if any one of the conditions (\ref{itm:mainthm_case1}),(\ref{itm:mainthm_case2}), or (\ref{itm:mainthm_case3}) in \Cref{thm:main_equidistribution} holds, then equidistribution would fail. 
	
	We need the following consequence of the reduction theory for arithmetic subgroups.
	
	\begin{lem}\label{lem:reduction theory lemma}
		Let $\bfF$ be a connected reductive $\bbQ$-subgroup of a connected reductive $\bbQ$-group $\bfG$. Suppose that $\rank_\bbQ\bfF<\rank_\bbQ\bfG$. Let $\Gamma\subset \bfG(\bbQ)$ be a discrete subgroup of $G$ commensurable with $\bfG(\bbZ)$. Then, for any compact set $C\subset G$, $C F\Gamma$ is a closed proper subset of $G$.
	\end{lem}

    \begin{proof}
    By \cite[7.7]{Bor69}, there exists a finite dimensional representation $V$ of $G$ defined over $\bbQ$ and a vector $v\in V(\bbQ)$ such that the stabilizer of $v$ is $F$. Since $\Gamma$ is commensurable with $\bfG(\bbZ)$, the coordinates of elements of $\Gamma v$ have bounded denominators with respect to a $\bbQ$-basis of $V$. Therefore, $\Gamma v$ is discrete in $V$. Therefore, $\Gamma F$ is closed in $G$, because $\Gamma F$ is the inverse of $\Gamma v$ under the map $g\mapsto gv$. Hence, $F\Gamma$ is closed in $G$.
    
        Let $\bfT$ be a maximal $\bbQ$-split torus of $\bfF$. By our assumption, there exists a maximal $\bbQ$-split torus $\bfS$ in $\bfG$ containing $\bfT$ as a proper subtorus \cite[V.15.4]{Bor91}. So we pick a non-trivial $\bbQ$-character $\beta\in X^\ast(\bfS)$ such that $\bfT\subset \ker\beta$ \cite[III.8.2(c)]{Bor69}. 
        
        Let $\bfP$ be a minimal $\bbQ$-parabolic subgroup of $\bfG$ containing $\bfS$. Let $\Delta\subset X^\ast(\bfS)$ denote the corresponding set of simple roots of $\bfG$ with respect to $\bfS$, for the ordering that is associated to $\bfP$. Let $A=\bfS(\bbR)^0$. Let $A_1=\{a\in A: \alpha(a)\leq 1,\,\forall \alpha\in\Delta\}$. Let $W\subset N_\bfG(\bfS)(\bbQ)$ be a finite set of representatives the Weyl group $N_\bfG(\bfS)(\bbQ)/Z_\bfG(\bfS)(\bbQ)$ of $\bfG$ with respect to $\bfS$. Then $A=\cup_{w\in W} wA_1w^{-1}$.
         
        Now $\beta:A\to \bbR^\ast_{>0}$ is a nontrivial continuous homomorphism. So, $\ker\beta$ is a strictly lower dimensional closed subgroup of $A$. Since $\Int(A_1)$, the interior of $A_1$, is a nonempty open subset of $A$, and we pick
        \[
        a\in \Int(A_1)\setminus \cup_{w\in W} w^{-1} (\ker \beta) w.
        \]
        Then, for all $\alpha\in\Delta$, $\alpha(a)<1$, and for all $w\in W$, $0<\beta(w a w^{-1})\neq 1$. Let $a_n'=a^n$ for all $n\in\bbN$. Then, as $n\to\infty$,        
        \begin{gather} \label{eq:an1}
        \alpha(a_n')\to 0, \,\forall \alpha\in \Delta \text{, and }\\
        \beta(wa_n'w^{-1})\to 0 \text{ or } \beta(wa_n'w^{-1})\to\infty,\, \forall w\in W.
        \label{eq:an2}
        \end{gather}

        Let $\bfP_F$ be a minimal $\bbQ$-parabolic subgroup of $F$ containing $\bfT$. As a consequence of the reduction theory due to Borel and Harish-Chandra~\cite[13.1]{Bor69}, there exists a finite set $\Sigma_F\subset \bfF(\bbQ)$ such that $F=\GS_F\Sigma_F(F\cap\Gamma)$, where $\GS_F$ is a Siegel subset of $F$ with respect to $\bfT$ and a choice of a minimal $\bbQ$-parabolic subgroup of $\bfF$ containing $\bfT$. And by \cite[12.2, 12.3]{Bor69}, $\GS_F\subset C_FT^0$, where $C_F$ is a compact subset of $F$. Thus, we have 
        \[
        F=C_FT^0\Sigma_F(F\cap\Gamma).       
        \]
        
        Now, suppose that $G=CF\Gamma$ for some compact set $C\subset G$. Then, after passing to a subsequence, we write 
        \[
        a'_n=c'_nt_n\sigma'\gamma'_n,
        \]
where $c'_n\subset CC_F$, $t_n\in T^0$ and $\sigma'\in \Sigma_F$, and $\gamma_n'\in F\cap \Gamma$. 

        Since $T^0\subset A=\cup_{w\in W} wA_1w^{-1}$, after passing to a subsequence, there exists $w'\in W$ such that 
        \begin{equation} \label{eq:sn}
        s_n:={w'}^{-1}t_nw'\in A_1.
        \end{equation}
        So we get
        \begin{equation} \label{eq:an'}
        a'_n=(c'_n{w'}^{-1}) s_n (w'\sigma')\gamma'_n.
        \end{equation}

         Let $\bfN$ be the unipotent radical of $\bfP$, defined over $\bbQ$. Then $\bfP=Z_{\bfG}(\bfS)\ltimes\bfN$. Let $\bfM$ be a maximal $\bbQ$-anisotropic subgroup of $Z_\bfG(\bfS)$. Then $Z_\bfG(\bfS)=\bfM \bfS$, and in fact, $Z_\bfG(\bfS)(\bbR)^0=M^0A$ and $P^0=M^0AN$. Let $K$ be a maximal compact subgroup of $G$ such that the corresponding Cartan involution preserves $A$. Then, $G=KP^0=KM^0AN$, and the map $(KM^0)\times A\times N\to G$ given by the group multiplication is a homeomorphism, see~\cite[(2.3)]{BJ07}. Now we express
         \begin{equation} \label{eq:mbv}
         w'(c'_n)^{-1}=k_nm_nb_nv_n,
         \end{equation}
         where $k_n\in K$, $m_n\in M^0$, $b_n\in A$, and $v_n\in N$. Since $W(CC_F)^{-1}$ is compact, the sequence $\{b_n\}$ is a relatively compact subset of $A$.

         Combining, \eqref{eq:an'} and \eqref{eq:mbv}, since $m_n\in Z_G(A)$, we get
         \begin{equation} \label{eq:aux1}
            s_n (w'\sigma')\gamma'_n=(w'{c'_n}^{-1})a'_n=k_nm_nb_nv_na_n'=(k_nb_na_n')[{a_n'}^{-1}m_nv_na_n']. 
         \end{equation}
         
         Since $\bfM$ is $\bbQ$-anisotropic, $M^0N/(M^0N\cap\Gamma)$ is compact. Let $\Omega$ be a compact neighborhood of $e$ in $M^0N$ such that $M^0N=\Omega(M^0N\cap\Gamma)$. Therefore,  ${a_n'}^{-1} m_nv_n a_n'=\omega_n\gamma''_n$
         for sequences $\{\omega_n\}\subset \Omega$ and 
         $\{\gamma''_n\}\subset M^0N\cap\Gamma$. 
         Inserting this expression in \eqref{eq:aux1} we get
         \begin{equation} \label{eq:Siegel}
         s_n\sigma\gamma_n=k_na_n\omega_n,
         \end{equation}
         where $\sigma=w'\sigma'\in \bfG(\bbQ)$, $\gamma_n=\gamma'_n{\gamma''_n}^{-1}\in\Gamma$, and $a_n=b_na_n'$. Since $\{b_n\}$ is relatively compact in $A$, by \eqref{eq:an1}, $a_n\in A_1$ for all large $n$.
         
         Let $\GS_0=KA_1\Omega$. By \eqref{eq:sn} and \eqref{eq:Siegel},
         \[
            \GS_0\sigma\gamma_n\cap\GS_0 \neq\emptyset,\, \forall n\gg1.
        \]
        Now, $\GS_0$ is a normal Siegel set in $G$ with respect to the $\bbQ$-parabolic $\bfP$ and our choice of the maximal compact subgroup $K$, see~\cite[12.3]{Bor69}. 
        Therefore, by Siegel property \cite[15.5]{Bor69}, which is generalization of a theorem of Siegel~\cite[4.6]{Bor69} for $G=\SL_n(\bbR)$, we get $\{\gamma_n\}$ is a finite set. Hence, after passing to a subsequence, we may assume that $\gamma_n=\gamma$ for all $n$. 

         By Bruhat decomposition \cite[11.4]{Bor69},  $G=NWZ_G(A)N$. So we can express
         \[
         \sigma\gamma=vwzu,
         \]
         where $v,u\in N$, $w\in W$, and $z\in AM=Z_G(A)$.
         Therefore, by \eqref{eq:Siegel}
         \[
        k_na_n(\omega_nu^{-1}z^{-1})=(s_nvs_n^{-1})s_nw.
         \]
Now, the sequences $\{a_n(\omega_nu^{-1}z^{-1})a_n^{-1}\}\subset MN$ and $\{s_n v s_n^{-1}\}\subset N$ are relatively compact, as $a_n,s_n\in A_1$. Thus,
\[
c_na_n=w^{-1}s_nw\in A,
\]
where $c_n=w^{-1}(s_n v s_n^{-1})^{-1}k_n(a_n(u_nu^{-1}z^{-1})a_n^{-1})$. Moreover, 
$\{c_n\}$ is a relatively compact subset of $A$. 
%Also $c_n=a_n^{-1}(w^{-1}s_nw)\in A$. 
Since $\bfT\subset\ker\beta$ and $\{t_n\}\subset T^0$, by \eqref{eq:sn},
\[
1=\beta(t_n)=\beta(w's_n{w'}^{-1})=\beta(w'w(w^{-1}s_n w) w^{-1}{w'}^{-1})=\beta(w_1 c_n w_1^{-1})\beta(w_1a_n w_1^{-1}),
\]
where $w_1=w'w\in W$. This contradicts \eqref{eq:an2}, because $\{\beta(w_1 c_n w_1^{-1})\}$ is a relatively compact subset of $\bbR^\ast_{>0}$. Therefore, $CF\Gamma\neq G$ for any compact set $C$.          
    \end{proof}
	
	Let the notations be the same as in \Cref{thm:main_equidistribution}. We take $\Omega=\phi(\interval)u_A^{-1}$.
	
	\begin{lem}\label{lem:focusing_quasiparabolic}
		Suppose $d<n$ and $A_\phi\in\cW_{n-1}(d,n-d)$, then the measures $g_t\lambda_\phi$ do not get equidistributed in $X$ as $t\to\infty$.
	\end{lem}
	
	\begin{proof}
		Suppose $A_\phi\in\cW_{n-1}(d,n-d)$, then by \Cref{lem:interpretation of Wrmn} where we take $r=d-1$, there exists $C>0$, $t_i\to\infty$ and $0\neq v_i\in\bbZ^n$ such that
		$
		\sup_{s\in \interval}\norm{g_{t_i}\phi(s) v_i}\leq C.
		$
		Without loss of generality, we may assume that all the $v_i$ are primitive; hence $v_i=\gamma_ie_1$ for $\gamma_i\in\Gamma$.
		We discuss the following two cases.
		
		Suppose first there exists $c>0$ such that $\inf_{s\in \interval}\norm{g_{t_i}\phi(s) v_i}\geq c$ for all $i$. Let $\mu$ be a weak-$^\ast$ limit of a subsequence of $\{g_{t_i}\lambda_\phi \}$, and $E$ be the support of $\mu$. We claim that $E$ is not equal to $X$. Indeed, take $M>0$ large enough such that $M^{-{n-1}}<c$ and $M>C$. Then the unimodular lattice $\bbZ M^{-{n-1}}e_1+\bbZ Me_2+\cdots+\bbZ Me_n$ is not in $E$, as every primitive vector in this lattice has length either $M^{-n-1}$ or at least $M$. Hence $\{g_t\lambda_\phi\}_{t\geq0}$ do not get equidistributed in $X$.
		
		Now suppose that after passing to a subsequence $\lim_{i\to \infty}\inf_{s\in \interval}\norm{g_{t_i}\phi(s) v_i}=0$. Let $V_1=\bbR e_1$ and $V_2=\bbR e_2+\cdots +\bbR e_n$ be eigenspaces of $g_t$ with eigenvalues $e^{(n-1)t}$ and $e^{-t}$ respectively. Let $\pi_1$ (resp.~$\pi_2$) be the projection from $\bbR^n$ to $V_1$ (resp.~$V_2$). By our assumption we have $\pi_2(\phi(s_i)v_i)=o(e^{t_i})$ for every $i$ and some $s_i\in \interval$. Since $\phi(s)\in U^+$, we have 
		\[
		\pi_2(\phi(s)v_i)=\pi_2(\phi(s_i)v_i)=o(e^{t_i}), \forall s\in \interval.
		\]
		On the other hand, since
		$
		\sup_{s\in \interval}\norm{g_{t_i}\phi(s) v_i}\leq C,
		$
		we have $
		{\pi_1(\phi(s) v_i)}= O(e^{-(n-1)t_i})
		$ for all $s\in\interval$. Hence there exists a sequence $t_i'\to\infty$ such that $\sup_{s\in \interval}\norm{g_{t_i'}\phi(s) v_i}\to0$ as $i\to\infty$. By Mahler's compactness criterion, we know that $g_{t_i'}\phi(\interval)$ leave any fixed compact set, and $g_{t_i'}\lambda_\phi$ weak-$^\ast$ converge to 0. In particular, in this case, we also have that $\{g_t\lambda_\phi\}_{t\geq0}$ do not get equidistributed in $X$.
	\end{proof}
	
	\begin{lem}\label{lem:focusing_Res GL_r}
		Suppose that there exist integers $r\geq d$, $m\geq 2$ with $rm=n$, and a number field $\bbK\subset\bbR$ with $[\bbK:\bbQ]=m$, such that $\Lphi$ is contained in some $(r-1)$-dimensional affine subspace of $\bbR^{n-1}$ which is defined over $\bbK$. Then $\{g_t\lambda_\phi \}_{t\geq0}$ do not get equidistributed in $X$. 
	\end{lem}
	
	\begin{proof}
		By \Cref{lem:interpretation of Res GL_r}, $u_A\in P_dP_r\bfG(\bbK)$. If $\Lphi$ is contained in some proper affine subspace of $\bbR^{n-1}$ defined over $\bbQ$, then $g_{t}\phi(B)$ leave any fixed compact set as $t\to\infty$, and equidistribution fails in this case. Otherwise, by \Cref{lem:K-subspace get stuck} and \Cref{cor:K-subspace get stuck} there exists a compact subset $\Sigma$ of $G$ such that $g_t\phi(\interval)\subset\Sigma F\Gamma$ for all $t\geq0$, where $F=\bfF(\bbR)$ and $\bfF\cong\Res_{\bbK/\bbQ}^{(1)}\GL_r$. Since $\rank_\bbQ(\bfF)=r-1<n-1=\rank_\bbQ(\bfG)$, by \Cref{lem:reduction theory lemma} we know that $\Sigma F\Gamma$ is a proper closed subset of $X=G/\Gamma$, and thus $\{g_t\lambda_\phi \}_{t\geq0}$ do not get equidistributed in $X$.
	\end{proof}
	
	\begin{lem}\label{lem:focusing_sympletic}
		Suppose that $n\geq 4$ is even, $d=2$ and $\Aext_\phi\in\cW_{\frac{n-2}{2}}(2n-3,N)$. Then $\{g_t\lambda_\phi \}_{t\geq0}$ do not get equidistributed in $X$. 
	\end{lem}
	\begin{proof}
		Let $V$ be the standard representation of $\SL_{n}$, and $W=\wedge^2 V$. By \Cref{lem:interpretation of sympletic}, there exists $C>0$, $t_i\to\infty$ and $w_i\in W(\bbZ)\setminus\{0\}$ such that
		\begin{equation}\label{eq:failure of equidistribution case three}
		\sup_{\omega\in\Omega}\norm{g_{t_i}\omega u_Aw_i}=\sup_{s\in \interval}\norm{g_{t_i}\phi(s)w_i}\leq C,\; \forall i.
		\end{equation}
		Without loss of generality we may assume that $w_i$ is primitive for every $i$. By the natural isomorphism $\wedge^2V\cong\wedge^2(V^*)^*$, we identify $W$ with the space of alternating bilinear forms on $V^*$. Hence we can talk about the rank of an element in $W$, which is an even number. After passing to a subsequence, we may assume that all the $w_i$ have the same rank $2k$. We discuss the following two cases.
		
		Suppose first that $2k=n$. Equivalently, $w_i$ has full rank for every $i$. Hence the Pfaffian $\pf(w_i)$ is a nonzero integer for every $i$. Since Pfaffian is a $\SL_n$-invariant, by \eqref{eq:failure of equidistribution case three} we know that $\pf(w_i)$ are bounded. Therefore, after passing to a subsequence we may assume that $\pf(w_i)$ is a constant, and thus all the $w_i$ are on the same $\SL_n(\bbR)$-orbit $Y$. Since $Y$ is a $\SL_n$-homogeneous variety defined over $\bbQ$, by a theorem of Borel and Harish-Chandra, there are only finitely many $\SL_n(\bbZ)$-orbits on $Y(\bbZ)$. \footnote{As a referee has pointed out, the number of $\SL_n(\bbZ)$-orbits in $Y(\bbZ)$ is equal to the number of ways of writing the Pfaffian as an $k$-fold product of positive integers.}
        After further passing to a subsequence, we may assume that all the $w_i$ are on the same $\SL_n(\bbZ)$-orbit. We write $w_i=\gamma_iw_0$, where $\gamma_i\in\Gamma$ and $w_0\in W(\bbZ)$ is of full rank. Let $\bfF=\bfG_{w_0}$ be the isotropy group of $w_0$, which is isomorphic to $\Sp_n$ as a $\bbQ$-group. We have an equivariant homeomorphism from $Gw_0$ to $G/F$. Now by \eqref{eq:failure of equidistribution case three}, there exists a compact subset $\Sigma$ of $G$ such that $g_t\phi(\interval)\subset\Sigma F\Gamma$ for all $t\geq0$. Since $\rank_\bbQ(\bfF)=\frac{n}{2}<n-1=\rank_\bbQ(\bfG)$, by \Cref{lem:reduction theory lemma} we know that $\Sigma F\Gamma$ is a proper closed subset of $X=G/\Gamma$, and thus $\{g_t\lambda_\phi \}_{t\geq0}$ do not get equidistributed in $X$. 
		
		Now suppose that $2k<n$. Consider the $\SL_n$-module morphism 
        \[f\colon\Sym^kW\to\wedge^{2k}V\]
		which sends $(v_1\wedge v_2)\cdot(v_3\wedge v_4)\cdots(v_{2k-1}\wedge v_{2k})$ to $v_1\wedge v_2\wedge\cdots\wedge v_{2k}$. Since each $w_i$ is of rank $2k$, we know that ${w}_i:=f(w_i^{k})$ is a pure tensor in $\wedge^{2k}V$. Equation \eqref{eq:failure of equidistribution case three} now implies that there exists $C'>0$ such that for all $i$,
		\[
		\sup_{s\in \interval}\norm{g_{t_i}\phi(s){w}_i}\leq C'.
		\]
		By \Cref{lem:reducing to standard representation} there exist nonzero vectors $v_i\in\bbZ^n$, a constant $C''>0$, and $t'_i\to\infty$ such that
		\[
		\sup_{s\in \interval} \norm{g_{t_i'}\phi(s)v_i} \leq C''.
		\]
		By the proof of \Cref{lem:focusing_quasiparabolic}, we know that $\{g_t\lambda_\phi \}_{t\geq0}$ do not get equidistributed in $X$. 
	\end{proof}
	
	\subsection{Proof of {Theorem~\ref{thm:main_equidistribution}}}
	Finally, we prove our main theorem on equidistribution, namely \Cref{thm:main_equidistribution}.
	
	\begin{proof}[Proof of \Cref{thm:main_equidistribution}]
		Suppose that $\{g_t\lambda_\phi \}_{t\geq0}$ do not get equidistributed in $X$. If $\{g_t\lambda_\phi\}_{t\geq0}$ has escape of mass, then by \Cref{thm:main_nondivergence} we have $n<d$ and $A_\phi\in\cW'_{n-1}(d,n-d)\subset \cW_{n-1}(d,n-d)$; that is, (\ref{itm:mainthm_case1}) of the theorem holds. Now assume that there is no escape of mass, then there exists a sequence $t_i\to\infty$ such that $g_{t_i}\lambda_\phi$ weak-$^\ast$ converge to a probability measure on $X$ which is not the Haar measure. Hence by \Cref{prop:consequence of failure of equidistribution} we know that there exists a representation $V$ of $\bfG$, a sequence $t_i\to\infty$, a sequence $\{\gamma_i\}$ in $\Gamma$, a nonzero vector $v_0\in V(\bbQ)$ which is not fixed by $G$, and a constant $C>0$ such that \eqref{eq:linear_focusing_g_t} holds. Therefore $d<n$, and (\ref{itm:focusing-1}) or (\ref{itm:focusing-2}) or (\ref{itm:focusing-3}) of \Cref{prop:consequence of linear focusing} holds. Applying \Cref{lem:interpretation of Wrmn}, \Cref{lem:interpretation of Res GL_r} and \Cref{lem:interpretation of sympletic}, we know that (\ref{itm:focusing-1}), (\ref{itm:focusing-2}), and (\ref{itm:focusing-3}) of \Cref{prop:consequence of linear focusing} are equivalent to (\ref{itm:mainthm_case1}), (\ref{itm:mainthm_case2}), and (\ref{itm:mainthm_case3}) of \Cref{thm:main_equidistribution}, respectively. Therefore, (\ref{itm:mainthm_case1}) or (\ref{itm:mainthm_case2}) or (\ref{itm:mainthm_case3}) of \Cref{thm:main_equidistribution} holds.
		
		Conversely, suppose that (\ref{itm:mainthm_case1}) or (\ref{itm:mainthm_case2}) or (\ref{itm:mainthm_case3}) of \Cref{thm:main_equidistribution} holds. Then by \Cref{lem:focusing_quasiparabolic}, \Cref{lem:focusing_Res GL_r}, and \Cref{lem:focusing_sympletic}, $\{g_t\lambda_\phi \}_{t\geq0}$ do not get equidistributed in $X$. 
	\end{proof}
	
	%%%%%%%%%%%%%%%%%%%%%%%%%%%%%%%%%%%%%%%%%%%
	%%%%%%%%%%%%%%%%%%%%%%%%%%%%%%%%%%%%%%%%%%%

	%====================================================
	%====================================================
	%		Start of Appendix
	%====================================================
	%====================================================
	\appendix
	\section{A classification theorem}
	In this section, we briefly recall the definition and basic properties of root systems and weights and then prove a classification theorem that is used for describing intermediate subgroups. Some statements are provided without proof, and readers are referred to \cite[Chapter III, VI]{HumphreysLieAlg} for a detailed discussion.
	
	\subsection{Classification of root systems with a certain property}
	Let $E$ be a Euclidean space with inner product $(\cdot, \cdot)$. A \emph{reflection} with respect to a vector $\alpha\in E$ is the isometry of $E$ given by 
	$\sigma_\alpha(\beta)=\beta-\frac{2(\beta, \alpha)}{(\alpha, \alpha)}\alpha$. Write $\nba = \frac{2(\beta, \alpha)}{(\alpha, \alpha)}$. 
	
	\begin{definition} \label{def:root-system}
		A subset $\Phi$ of the Euclidean space $E$ is called a \emph{(reduced) root system} in $E$ if the following axioms are satisfied:
		\begin{enumerate}
			\item $\Phi$ is finite, spans $E$, and does not contain 0.
			\item If $\alpha\in\Phi$, the only multiples of $\alpha$ in $\Phi$ are $\pm\alpha$.
			\item If $\alpha\in\Phi$, the reflection $\sigma_\alpha$ leaves $\Phi$ invariant.
			\item If $\alpha, \beta\in\Phi$, then $\nba \in \bbZ$.
		\end{enumerate}
	\end{definition}
	
	Let $\Phi$ be a root system in $E$. The \emph{Weyl group} of $\Phi$ is the subgroup of $\GL(E)$ generated by the reflections $\{\sigma_\alpha\mid \alpha\in\Phi \}$, and is denoted by $\Weyl$. The lattice generated by $\Phi$ in $E$ is called the \emph{root lattice}, and is denoted by $\Lambda_r$.
	
	A subset $\Delta$ of $\Phi$ is called a set of \emph{simple roots} if it is a basis of $E$ and each root $\beta\in\Phi$ can be written as a nonnegative or nonpositive integral combination of elements in $\Delta$. We call $\beta$ positive or negative respectively.
	
	We call $\Phi$ \emph{irreducible} if it cannot be partitioned into the union of two proper subsets such that each root in one set is orthogonal to each root in the other. It is known that the Weyl group acts irreducibly on an irreducible root system.
	In an irreducible root system, at most two root lengths occur. All roots of the same length are in the same orbit of the Weyl group $\Weyl$.
	In case $\Phi$ is irreducible, with two distinct root lengths, we speak of \emph{long roots} and \emph{short roots}. If all roots are of equal length, it is conventional to call all of them long.
	
	Let $\Phi$ be a root system in a Euclidean space $E$, with Weyl group $\Weyl$. Let $\Lambda$ be the set of all $\lambda\in E$ for which $\nla=\frac{2(\lambda,\alpha)}{(\alpha,\alpha)}\in\bbZ$ for all $\alpha\in\Phi$, and call its elements \emph{weights}. $\Lambda$ is a lattice in $E$, and is called the \emph{weight lattice}. Fix a system of simple roots $\Delta\subset\Phi$, and define $\lambda\in\Lambda$ to be \emph{dominant} if $\nla\geq 0$ for all $\alpha\in\Delta$. Let $\Lambda^+$ denote the set of dominant weights. Suppose $\Delta=\{\alpha_1,\dots,\alpha_r \}$. The \emph{fundamental weights} are $\omega_1,\dots,\omega_n$ such that $\lip\omega_i,\alpha_j\rip = \delta_{ij}$, where $\delta_{ij}$ is the Kronecker symbol.
	
	We call a subset $\Pi$ of $\Lambda$ \emph{saturated} if for all $\lambda\in\Pi$, $\alpha\in\Phi$ and $i$ between 0 and $\nla$, the weight $\lambda-i\alpha$ also lies in $\Pi$. Any saturated set is stable under $\Weyl$.
	
	Given a root system $\Phi$, we have a decomposition
	\[
	\Phi=\Phi_1\amalg\Phi_2\amalg\cdots\amalg\Phi_k,
	\]
	where $\Phi_i$'s are irreducible and orthogonal to each other.
	\begin{thm}\label{thm:classification_root_system}
		Let $\Phi=\Phi_1\amalg\cdots\amalg\Phi_k$ be a root system with Weyl group $\Weyl$, and let $\Lambda=\Lambda_1\oplus\cdots\oplus\Lambda_k$ be the corresponding weight lattice. Let $\Pi=\{\lambda_1, \dots, \lambda_n\}$ be a saturated subset of $\Lambda$ with a highest weight. Suppose there exists a root $\alpha\in\Phi_1$ such that 
		\begin{equation}\label{eq:reflection_numbers_3_cases}
		\nlia = 
		\begin{cases}
		1 & i = 1 \\
		-1 & i = 2 \\
		0 & 3\leq i \leq n 
		\end{cases}.
		\end{equation}
		Then there is a choice of simple roots such that $\lambda_1$ is the highest weight of $\Pi$; $\Pi$ is contained in the span of $\Phi_1$; and one of the following holds:
		\begin{enumerate}
			\item $\Phi_1 = \typeA_{n-1}$, and $\lambda_1\in\{\omega_1,\omega_{n-1}\}$.
			\item $n$ is even, $\Phi_1 = \typeC_{\frac{n}{2}}$, and $\lambda_1=\omega_1$.
		\end{enumerate}
	\end{thm}
	
	\begin{remark}
		The dominant weights correspond to the highest weights of irreducible representations of semisimple Lie algebras. The two cases in the theorem correspond to (1) standard and its contragradient representation of $\mathfrak{sl}_n$, and (2) standard representation of $\mathfrak{sp}_n$. 
	\end{remark}
	
	We make some preparations before proving \Cref{thm:classification_root_system}.
	
	\begin{lem}\label{lem:lambda_1_highest_weight}
		Under the assumption of \Cref{thm:classification_root_system}:
		\begin{enumerate}
			\item \label{item:1} there is a choice of simple roots such that $\lambda_1$ is the highest weight of $\Pi$;
			\item \label{item:2} $\Pi$ is contained in the span of $\Phi_1$.
		\end{enumerate}
	\end{lem}
	
	\begin{proof}
		We may choose simple roots $\Delta=\Delta_1\amalg\cdots\amalg\Delta_r$ such that $\alpha$ is a dominant weight, i.e. $(\beta,\alpha)\geq 0$ for all $\beta\in\Delta$. Under this choice, we have that $\lambda_1$ is the highest weight of $\Pi$. Indeed, by assumption $\Pi$ has a highest weight $\lambda$. Suppose $\lambda\neq\lambda_1$, by \Cref{eq:reflection_numbers_3_cases} we have $\lip \lambda, \alpha \rip < \lip \lambda_1, \alpha \rip$.	On the other hand, since $\lambda_1\prec\lambda$ we have $\lambda = \lambda_1 + \sum n_i\beta_i$, where $\beta_i\in\Delta$ and $n_i\geq 0$; it follows that $\lip \lambda, \alpha \rip \geq \lip \lambda_1, \alpha \rip$.	Contradiction. Hence (\ref{item:1}) is verified.
		
		Let $E_1=\Phi_1\otimes_\bbZ\bbR$. To prove (\ref{item:2}), we use the fact that a saturated set is in the convex hull of the Weyl group orbit of the highest weight, see \cite[Section 13.4 Lemma B]{HumphreysLieAlg}. 
		For any $\beta\in\Phi\setminus\Phi_1$, we have $(\beta, \alpha)=0$, and it follows that 
		$\lip\sigma_\beta(\lambda_1),\alpha\rip=\lip\lambda_1,\alpha\rip=1$.
		Hence $\sigma_\beta(\lambda_1)=\lambda_1$ for all $\beta\in\Phi\setminus\Phi_1$. Note that $\Weyl$ is generated by the simple reflections in $\Delta$ and that simple reflections in $\Delta_i$ and $\Delta_j$ commute if $i\neq j$. Hence we have $\Weyl\cdot \lambda_1=\Weyl_1\cdot\lambda_1$, where $\Weyl_1$ is the Weyl group of $\Phi_1$. Finally, $\Weyl_1\cdot\lambda_1$ is contained in $E_1$, and $\Pi$ is contained in the convex hull of $\Weyl_1\cdot\lambda_1$, thus also in $E_1$.
	\end{proof}
	
	In the next lemma, we will need the following notion (see \cite[P72, Section 13 Exercise 13]{HumphreysLieAlg}): we call $\lambda\in\Lambda^+$ \emph{minimal} (or \emph{minuscule}) if $\mu\in\Lambda^+, \mu\prec\lambda$ implies that $\mu=\lambda$. Each coset of $\Lambda_r$ in $\Lambda$ contains precisely one minimal $\lambda$. One can show that $\lambda$ is minimal if and only if the $\Weyl$-orbit of $\lambda$ is saturated with highest weight $\lambda$, if and only if $\lambda\in\Lambda^+$ and $\lip \lambda, \beta \rip = 0,1,-1$ for all roots $\beta$.
	
	\begin{lem}\label{lem:minimality_of_lambda_1}
		Under the assumption of \Cref{thm:classification_root_system}, one can choose a system of simple roots such that $\lambda_1$ is minimal in $\Lambda_1^+$.
	\end{lem}
	
	\begin{proof}
		By \Cref{lem:lambda_1_highest_weight}, we may choose a set of simple roots such that $\lambda_1\in\Lambda_1^+$. By the discussion above, it suffices to show that
		\[
		\lip \lambda_1, \beta \rip = 0, \pm 1,\quad \forall\beta\in\Phi.
		\]
		
		Suppose that $\abs{\lip \lambda_1, \beta \rip}\geq 2$, consider the string 
		\[\{ \lambda_1, \dots, \lambda_1-\lip \lambda_1, \beta \rip \beta \}.\]
		This string has length at least 3, and pairing it with $\alpha$, we get a finite arithmetic progression of length at least 3. By \Cref{eq:reflection_numbers_3_cases}, the only possibility is that $\lip\lambda_1,\beta\rip=2$ and $\lambda_1-2\beta = \lambda_2$. Since $\Pi$ is saturated, there exists $i$ such that $\lambda_i=\sigma_\alpha(\lambda_1)$. Again by \Cref{eq:reflection_numbers_3_cases} we have $i=2$ and $\lambda_2=\sigma_\alpha(\lambda_1)=\lambda_1-\alpha$. But this implies $\alpha=2\beta$, contradicting that the root system $\Phi$ is reduced.
	\end{proof}
	
	Now we are in a position to prove \Cref{thm:classification_root_system}
	\begin{proof}[Proof of \Cref{thm:classification_root_system}]
		The first assertion is proved in \Cref{lem:lambda_1_highest_weight}. By \Cref{lem:minimality_of_lambda_1}, we have that $\lambda_1$ is minimal. 
		
		There is a complete (finite) list of minimal weights in each irreducible root system (c.f. \cite[P72, Section 13 Exercise 13]{HumphreysLieAlg}). We consider each case separately to verify whether \Cref{eq:reflection_numbers_3_cases} holds. This case-by-case verification requires some efforts involving long and elementary calculations. Finally, one can show that \Cref{eq:reflection_numbers_3_cases} only holds for $(\Phi,\lambda_1)$ being $(\typeA_{n-1},\omega_1)$, $(\typeA_{n-1},\omega_{n-1})$ or $\big(\typeC_{\frac{n}{2}},\omega_1\big)$.
	\end{proof}
	
	\subsection{Classification of intermediate Lie subalgebras}\label{subsect:rep_semisimple_Lie_alg}
	Let $\fg$ be a semi-simple Lie algebra over an algebraically closed field of characteristic 0, and let $\fh$ be a Cartan subalgebra of $\fg$. Let $V$ be a finite-dimensional $\fg$-module. Then the $\fh$-action on $V$ is diagonalizable, and we have the decomposition $V=\bigoplus V_\lambda$, where $\lambda$ runs over $\fh^*$ and $V_\lambda=\{ v\in V\mid h.v=\lambda(h)v, \;\forall h\in H \}$. If $V_\lambda$ is nonzero, we call $\lambda$ a \emph{weight} of $V$, and $V_\lambda$ a \emph{weight space}. The set of weights of $V$ is saturated with a highest weight. The highest weight yields a one-to-one correspondence between dominant integral weights and irreducible representations of $\fg$.
	
	Let $\ff_{12}$
	be the Lie subalgebra of $\mathfrak{sl}_n$ consisting of traceless matrices in the upper-left $2\times 2$ block.
	
	\begin{thm}\label{thm:classification_Lie_algebra}
		Let $\varrho\colon \fg \to \End(V)$ be a finite-dimensional faithful irreducible representation of a semisimple Lie algebra $\fg$. Suppose that under an identification of $\End(V)$ with $\mathfrak{gl}_n$, $\varrho(\fg)$ contains $\ff_{12}$. Then $\fg$ is simple, and moreover one of the following holds:
		\begin{enumerate}
			\item $\fg=\mathfrak{sl}_n$, and $V$ is the standard representation or the contragradient representation of the standard representation of $\fg$.
			\item $n$ is even, $\fg=\mathfrak{sp}_n$, and $V$ is the standard representation of $\fg$.
		\end{enumerate}
	\end{thm}
	
	\begin{proof}
		Since $\fg$ is semisimple, $\varrho(\fg)$ is contained in $\mathfrak{sl}_n$. Let $\ff=\varrho(\fg)$. Take a chain of Cartan subalgebras $\fh_{12}\subset \fh_\ff\subset\fh$ in the chain $\ff_{12}\subset\ff\subset\mathfrak{sl}_n$. By taking a suitable basis $\{e_1, e_2, \dots, e_n\}$ of $V$ we may assume $\fh$ consists of diagonal matrices, and then $\fh_{12}$ consists of elements of the form $\diag(a,-a,0,\dots, 0)$. Now let $\alpha$ be the character $\diag(a_1,a_2,\dots, a_n)\mapsto a_1-a_2$. Since $\ff$ contains $E_{12}$ and $\fh_\ff$ acts on $E_{12}$ via $\alpha$, we have that $\alpha$ is a root of $\ff$ with respect to $\fh_\ff$. Let $\Pi_\ff=\{\lambda_1, \lambda_2, \dots, \lambda_n\}$ be the weights of $V$ with respect to $\fh_\ff$ counted with multiplicity, such that $e_1, e_2, \dots, e_n$ are their weight vectors respectively. Since $\fh_\ff^*$ is a Euclidean subspace of $\fh^*$, we have
		\[
		\nlia = 
		\begin{cases}
		1 & i = 1 \\
		-1 & i = 2 \\
		0 & 3\leq i \leq n 
		\end{cases}.
		\]
		Hence we can apply \Cref{thm:classification_root_system} to conclude that $\varrho$ factors through a simple factor $\fg_1$ of $\fg$. But $(\varrho, V)$ is faithful, hence $\fg_1=\fg$ and $\fg$ is simple. We apply \Cref{thm:classification_root_system} again, and the conclusion of the theorem follows.
		Indeed, note that in each highest weight module associate to each dominant weight appearing in (1)(2) of \Cref{thm:classification_root_system}, all weight spaces are one-dimensional. Hence $\{\lambda_1,\dots,\lambda_n\}$ are distinct, and the cardinality of $\Pi_\ff$ is $n$.
	\end{proof}
	
	\subsection{Classification of intermediate subgroups}
	Using the above classification theorem on intermediate Lie subalgebras, we are able to obtain the following classification of intermediate subgroups, which we are interested in.
	
	\begin{thm}\label{thm:classification_intermediate_subgroups}
		Let $\bfG$ be a reductive group over an algebraically closed field $\bbK$ of characteristic 0. Let $\rho\colon \bfG\to\GL(V)$ be a faithful irreducible representation of $\bfG$, such that $\rho(\bfG)$ is contained in $\SL(V)$. Suppose that there are linear subspaces $W_1$ and $W_2$ of $V$ with $V=W_1\bigoplus W_2$ and $\dim W_1\geq 2$, such that $\SL(W_1)\times 1_{W_2}$ is contained in $\rho(G)$. Then one of the following holds:
		\begin{enumerate}
			\item $\rho(\bfG)=\SL(V)$.
			\item $n$ is even and $\dim W_1=2$; there exists a sympletic form $\omega$ on $V$ such that $\rho(\bfG)=\Sp(V, \omega)$.
		\end{enumerate}
	\end{thm}
	
	\begin{proof}
		We first make a few reductions. Let $\bfG^0$ be the identity component of $\bfG$. Suppose the theorem holds for $\bfG^0$, then one easily sees that $\bfG=\bfG^0$. Therefore, without loss of generality, we may assume that $\bfG$ is connected. Also, it suffices to prove the theorem for $\dim W_1 = 2$. We first prove the theorem for $\bfG$ semisimple.
		
		Let $n=\dim V$. We take the differential of $\rho$ and get a Lie algebra representation $d\rho\colon \fg\to\End(V)$. By \Cref{thm:classification_Lie_algebra}, either $\fg=\mathfrak{sl}_n$ or $\fg=\mathfrak{sp}_n$ ($n$ even). 
		
		If $\fg=\mathfrak{sl}_n$ and $(d\rho, V)$ is the standard or the contragradient representation of $\fg$, then $d\rho$ lifts to the standard or the contragradient representation of the simply-connected group $\SL_n$, which is faithful. Hence $\bfG=\SL_n$ and $\rho(\bfG)=\SL(V)$. 
		
		If $\fg=\mathfrak{sp}_n$ and $(d\rho, V)$ is the standard representation of $\fg$,  then $d\rho$ lifts to the standard representation of the simply-connected group $\Sp_n$, which is faithful. Hence $\bfG=\Sp_n$ and $\rho(\bfG)=\Sp(V,\omega)$ for some sympletic form $\omega$ on $V$.
		
		For general $\bfG$, consider $\bfG^{\mathrm{der}}=[\bfG,\bfG]$, and the conditions of the theorem still hold for $\bfG^{\mathrm{der}}$. Then $\bfG^{\mathrm{der}}$ satisfies either (1) or (2), and $\bfG=\bfG^{\mathrm{der}}$ as $\rho$ is faithful.
	\end{proof}
	
	%====================================================
	%====================================================
	%		End of Appendix
	%====================================================
	%====================================================
	
	%\providecommand{\bysame}{\leavevmode\hbox to3em{\hrulefill}\thinspace}
	%\providecommand{\MR}{\relax\ifhmode\unskip\space\fi MR }
	%\providecommand{\MRhref}[2]{%
		%\href{http://www.ams.org/mathscinet-getitem?mr=#1}{#2}}
	%\providecommand{\href}[2]{#2}

\end{document}